\documentclass[psamsfonts]{amsart}

\usepackage{amssymb, amsmath, extpfeil}
\usepackage{amsfonts}
\usepackage{mathabx}
\usepackage{bbm}
\usepackage{tikz-cd}
\usetikzlibrary{nfold}
\usepackage{rotating}
\usepackage{textcomp}
\usepackage{diagbox}
\usepackage{pdflscape}
\usepackage{caption}
\usepackage[normalem]{ulem}
\usepackage{mathtools}
\usepackage{adjustbox}
\usepackage{thmtools}
\usepackage{thm-restate}
\usepackage{todonotes}
\usepackage[hidelinks]{hyperref}
\usepackage{multirow}

\theoremstyle{plain}
\newtheorem{thm}{Theorem}[section]
\newtheorem{cor}[thm]{Corollary}
\newtheorem{prop}[thm]{Proposition}
\newtheorem{lem}[thm]{Lemma}

\newtheorem{hyp}[thm]{Hypothesis}
\newtheorem{mainthm}{Theorem}

\theoremstyle{definition}
\newtheorem{defn}[thm]{Definition}

\theoremstyle{remark}
\newtheorem{exmp}[thm]{Example}
\newtheorem{exmps}[thm]{Examples}
\newtheorem{notn}[thm]{Notation}
\newtheorem{notns}[thm]{Notations}

\newtheorem{rem}[thm]{Remark}
\newtheorem{rems}[thm]{Remarks}

\newtheorem{sch}[thm]{Scholium}

\newcommand{\Z}{\mathbb{Z}}
\newcommand{\Q}{\mathbb{Q}}
\newcommand{\R}{\mathbb{R}}

\newcommand{\RP}{\mathbb{R}\mathrm{P}}

\newcommand{\catname}[1]{{\normalfont\textbf{#1}}}
\newcommand{\op}{{\mathrm{op}}}

\newcommand{\circles}{\langle S^1 \rangle}
\newcommand{\spheres}[1]{\langle S^#1 \rangle}

\newcommand{\mset}{\catname{MSet}}
\newcommand{\finset}{\catname{FinSet}_*}
\newcommand{\epi}{\catname{Epi}}
\newcommand{\injfin}{\catname{FI}}
\newcommand{\sset}{\catname{sSet}}
\newcommand{\spaces}{\catname{S}_*}
\newcommand{\unpointed}{\catname{S}}
\newcommand{\connected}{\spaces^{1}}
\newcommand{\orbit}{\catname{O}}

\newcommand{\spectra}[1][X]{\ifthenelse{\equal{#1}{X}}{\catname{Sp}}{\catname{Sp}^{\geq #1}}}

\newcommand{\idem}{\catname{Idem}}

\newcommand{\fr}{\allfr^{\mathrm{fg}}}
\newcommand{\allfr}{\catname{Fr}}
\newcommand{\sfr}{\catname{sFr}}
\newcommand{\ab}{\catname{Ab}}

\newcommand{\ch}{\catname{Ch}}

\newcommand{\presh}{\mathcal{P}}
\newcommand{\comod}{\catname{comod}}
\newcommand{\DP}{\operatorname{DP}}
\newcommand{\module}{\catname{mod}}
\newcommand{\com}{\mathtt{Com}}
\newcommand{\Lie}{\mathtt{Lie}}
\newcommand{\aut}{\operatorname{aut}}

\newcommand{\sgr}{\catname{sGr}}
\newcommand{\ob}{\operatorname{ob}}

\newcommand{\classifying}{\operatorname{B}}
\newcommand{\DK}{L^H}

\newcommand{\cofibrant}[1]{\operatorname{q}#1}

\newcommand{\lkan}{\mathrm{L}}
\newcommand{\rkan}{\mathrm{R}}
\newcommand{\free}[1]{\operatorname{F}\,\langle #1\rangle}

\newcommand{\const}{\mathrm{c}}

\newcommand{\HH}{\operatorname{H}} 

\newcommand{\abelianization}{\mathrm{ab}}
\newcommand{\tensor}{\mathrm{T}}
\newcommand{\symm}{\mathrm{S}}

\newcommand{\Passi}{\mathrm{Pa}}
\newcommand{\ideal}{\operatorname{I}}

\newcommand{\power}{\mathrm{P}}

\DeclareMathOperator{\ext}{Ext}
\DeclareMathOperator{\map}{Map}
\DeclareMathOperator{\pointedmaps}{\map_*}
\DeclareMathOperator{\spectralmaps}{\underline{\underline{\map}}}
\DeclareMathOperator{\nat}{Nat}
\DeclareMathOperator{\spectralNat}{\underline{\underline{\nat}}}
\DeclareMathOperator{\sur}{Sur}
\DeclareMathOperator{\Part}{\rho}
\DeclareMathOperator{\Stirling}{Str}
\DeclareMathOperator{\Comp}{C}
\DeclareMathOperator{\Bell}{B}
\DeclareMathOperator{\Nerve}{N}

\DeclareMathOperator{\fun}{Fun}
\DeclareMathOperator{\poly}{Poly}
\DeclareMathOperator{\exc}{Exc}
\DeclareMathOperator{\splitcube}{Split}

\DeclareMathOperator{\ce}{cr}

\DeclareMathOperator{\colim}{colim \,}
\DeclareMathOperator{\hocolim}{hocolim \,}
\DeclareMathOperator{\ho}{Ho}
\DeclareMathOperator{\id}{id}

\title{Polynomial functors from free groups to a stable $\infty$-category}
\author{Gregory Arone}

\begin{document}
\begin{abstract}
We study the category of polynomial functors from finitely generated free groups to a stable $\infty$-category $\catname{D}$. We show that this category is equivalent to the category of excisive functors from pointed spaces to $\catname{D}$, and also to truncated right comodules over the commutative operad with values in $\catname{D}$. The latter formulation generalizes a result of Geoffrey Powell in characteristic zero. 

We use the equivalence of categories to calculate $\ext$ in the category of polynomial functors from free groups to abelian groups. We reproduce previous results of Christine Vespa and others, and do many new calculations. Using the work of Aurelien Djament, we give applications to stable cohomology of the group of automorphisms of free groups with coefficients in a polynomial functor.
\end{abstract}
\maketitle

\tableofcontents

\section{Introduction} 
Let $\fr$ be the category of finitely generated free groups, and $\ab$ the category of abelian groups. This paper was motivated by desire to calculate Ext groups in the category of functors $[\fr, \ab]$. Our approach involves translating the algebraic problem about groups into a more homotopy-theoretical one involving spaces. It gradually transpired that the language of $\infty$-categories provides a convenient framework for the sort of manipulations we need to do. Moreover, it turns out that our general result on the functor category $[\fr, \ch]$ (where $\ch$ denotes chain complexes) depends only on $\ch$ being a stable $\infty$-category. Therefore, the first half of the paper is written in the language of $\infty$-categories, and theorems~\ref{thm:main-general} and~\ref{mainthm-modules} are formulated for general stable $\infty$-categories rather than $\ch$ or $\ab$.
The second half of the paper is devoted to calculations. Readers who are primarily interested in the calculations can skip to Section~\ref{subsec:applications}, read from there to the end of the introduction, and then go straight to Part~\ref{part:calculations}, while treating Theorem~\ref{thm:fr-to-top_intro} as a black box.

With that said, let $\catname{D}$ be a stable $\infty$-category, such as $\ch$ or $\catname{Spectra}$. In this paper we study functors from $\fr$ to $\catname{D}$. To do this, we use the connection between the homotopy theories of groups and of spaces. 
Let $\spaces$ be the $\infty$-category of pointed spaces. The classifying space functor $\classifying\colon \fr \to \spaces$ induces a restriction on functor categories
\[
\rho_{\classifying}\colon \fun(\spaces, \catname{D}) \to \fun(\fr, \catname{D}).
\]
This functor is not an equivalence of categories, of course. But it restricts to an equivalence between categories of excisive and polynomial functors. Let $\exc_n(\spaces, \catname{D})$ denote the category of $n$-excisive functors from $\spaces$ to $\catname{D}$. Let $\poly_n(\fr, \catname{D})$ denote the category of degree $n$ polynomial functors from $\fr$ to $\catname{D}$ (defined in Section~\ref{sec:polynomial}). We prove the following  result. In the paper it is subsumed in Theorem~\ref{theorem:main general}.
\begin{mainthm}\label{thm:main-general}
Restriction along the classifying space functor induces, for each $n$, an equivalence of $\infty$-categories
\[
\rho_{\classifying}\colon \exc_n(\spaces, \catname{D}) \xrightarrow{\simeq} \poly_n(\fr, \catname{D}).
\]
The inverse to $\rho_{\classifying}$ is given by a composition of two functors, both of which are equivalences:
\[
\poly_n(\fr, \catname{D})\xrightarrow{\lkan_{\classifying}} \exc_n(\connected, \catname{D})\xrightarrow{\rkan} \exc_n(\spaces, \catname{D}).
\]
Here $\connected$ is the category of \emph{connected} pointed spaces, $\lkan_{\classifying}$ is left Kan extension along $\classifying$, and $\rkan$  is right Kan extension along the inclusion functor $\connected \hookrightarrow \spaces$.
\end{mainthm}
Let us note that one can use Theorem~\ref{thm:main-general} to further relate the category of polynomial functors from $\fr$ to $\catname{D}$ to categories of right (co)modules over the commutative or the Lie operad. Let $\finset^{\le n}$ the category of pointed sets with at most $n$ non-basepoint elements. There is an inclusion of categories $\finset^{\le n}\hookrightarrow \spaces$, which induces a restriction functor
\[
\fun(\spaces, \catname{D})\to \fun(\finset^{\le n}, \catname{D}).
\]
By Lemma~\ref{lem: excisive finset}, this functor restricts to an equivalence of categories, for each $n$
\[
\exc_n(\spaces, \catname{D})\xrightarrow{\simeq} \fun(\finset^{\le n}, \catname{D}).
\]
Furthermore, let $\epi^{\le n}$ be the category of unpointed sets with at most $n$ elements, and surjective functions between them. By~\cite{Helmstutler, Walde} there is an equivalence
\[
\fun(\finset^{\le n}, \catname{D})\simeq \fun(\epi^{\le n}, \catname{D})
\]
The category $\fun(\epi^{\le n}, \catname{D})$ is the same as the category of $n$-truncated right comodules in $\catname{D}$ over the (non-unital) commutative operad. Thus we also denote this category by $\com\mathrm{-}\comod_{\le n}(\catname{D})$. Furthermore, there is an equivalence between the category of right comodules over $\com$ and \emph{divided power right modules} over $\Lie$, where $\Lie$ is the bar construction on $\com$, also known as the shifted spectral Lie operad. A version of this equivalence is proved in~\cite[Theorem 3.82]{Arone-Ching_Cross-effects}, albeit with a somewhat ad hoc notion of divided power module. We denote the category of $n$-truncated divided power modules by $\DP\Lie\mathrm{-}\module_{\le n}(\catname{D})$. Putting all these equivalences together, we obtain the  following theorem
\begin{mainthm}\label{mainthm-modules}
For all $n$, there are equivalences of $\infty$-categories
\[
\poly_n(\fr, \catname{D})\simeq \com\mathrm{-}\comod_{\le n}(\catname{D}) \simeq \DP\Lie\mathrm{-}\module_{\le n}(\catname{D}).
\]
\end{mainthm}
The first of these equivalences is proved in Corollary~\ref{cor:com comodule main}. We only mention the connection with divided power $\Lie$-modules in the introduction, for the sake of the broad picture. They are not needed for applications in this paper.
\begin{rem}
There is a forgetful functor from the category of divided power right modules over an operad to the category of ordinary right modules. When $\catname{D}=\ch_{\Q}$ is the category of rational chain complexes, this forgetful functor is an equivalence. Therefore, theorem~\ref{mainthm-modules} gives equivalences of categories
\begin{equation}\label{eq:Powell}
\poly_n({\fr}, \ch_{\Q})\simeq \com\!-\!\comod_{\le n}(\ch_{\Q}) \simeq \Lie\mathrm{-mod}_{\le n}(\ch_\Q).
\end{equation}
A very similar equivalence of categories (as well as a contravariant version of it) was proved by Geoffrey Powell~\cite{Powell_analytic}. We say very similar, because Powell actually proves an equivalence of categories of functors with values in $\Q$-vector spaces rather than functors with values in rational chain complexes. It would be good to understand precisely the relationship between~\eqref{eq:Powell} and the results of Powell in~\cite{Powell_analytic}. But putting this aside, Theorem~\ref{mainthm-modules} may be viewed as a generalisation of Powell's result in the sense that it works not just in characteristic zero, but for functors with values in any stable $\infty$-category, e.g., spectra. Furthermore, Theorems~\ref{thm:main-general} and~\ref{mainthm-modules} work for contravariant as well as covariant functors with values in, say, $\ch$ or $\catname{Spectra}$. In fact, one can write a version of Theorems~\ref{thm:main-general} and~\ref{mainthm-modules} for bifunctors, i.e., functors from $\fr\times{\fr}^{\op}$ to $\catname{D}$. See Proposition~\ref{prop:bivariant}.
\end{rem}
\subsubsection*{Idea of proof of the main theorems}
We prove Theorem~\ref{thm:main-general} by factoring the functor $\rho_B$ as a composition of three functors, and showing that each one of them is an equivalence. The factorization takes the following form
\begin{equation}\label{eq:factorization}
\exc_n(\spaces, \catname{D})\xrightarrow{\alpha} \exc_n(\connected, \catname{D}) \xrightarrow{\beta} \exc_n(\sgr, \catname{D})\xrightarrow{\gamma} \poly_n(\fr, \catname{D}).
\end{equation}
Explanation: Firstly, $\connected$ is the category of \emph{connected} pointed spaces, and $\alpha$ is restriction along the inclusion $i\colon\connected \hookrightarrow \spaces$. We show that restriction along $i$ induces an equivalence of categories of excisive functors. A version of this statement was proved by Brantner and Mathew~\cite[Theorem 3.36]{Brantner-Mathew}. We give a different proof, which works without requiring that $\catname{D}$ is stable (but in this paper we only use the result for stable $\catname{D}$). We also show that the inverse to $\alpha$ is given by right Kan extension along $i$ (Proposition~\ref{prop:right Kan}).

Secondly, $\sgr$ is the category of simplicial groups, and $\beta$ is restriction along the classifying space functor $\classifying\colon \sgr\to \connected$. That $\beta$ is an equivalence follows essentially from Kan's theorem that there is an equivalence of $\infty$-categories between $\sgr$ and $\connected$~\cite{KanLoop}.

Thirdly and finally, $\gamma$ is induced by restriction along the inclusion $\fr\hookrightarrow \sgr$. We show that $\gamma$ is an equivalence (Proposition~\ref{prop:poly to exc}). Its inverse is given by left Kan extension.

\subsection{Applications to Ext groups} \label{subsec:applications} Our main application of Theorem~\ref{thm:main-general} is to calculation of Ext groups in the category of functors from $\fr$ to $\ab$. These Ext groups were studied by a number of people, perhaps most prominently by Christine Vespa~\cite{Vespa2018}.

Let $\ch$ be the $\infty$-category of chain complexes or, equivalently, of $H\Z$-modules. For the application we specialize Theorem~\ref{thm:main-general} to the case $\catname{D}=\ch$. Functors from $\fr$ to $\ab$ can be viewed as functors from $\fr$ to $\ch$, and Ext in the category $\fun(\fr, \ab)$ can be interpreted as homotopy groups of mapping objects in $\fun(\fr, \ch)$.

Let $\spheres{1}\subset \spaces$ be the full subcategory of $\spaces$ (pointed spaces), spanned by finite wedges of $S^1$. Note that the fundamental group functor induces an equivalence of topological categories $\pi_1\colon \spheres{1}\xrightarrow{\simeq}\fr$. We prove the following result (an abbreviated version of Theorem~\ref{thm:fr-to-top})
\begin{mainthm}\label{thm:fr-to-top_intro}
Suppose $F\colon \fr\to \ab$ is a polynomial functor of degree $n$. Then there exists an $n$-excisive functor $\widehat F\colon \spaces\to \ch$, unique up to equivalence, for which there is an equivalence of functors from $\spheres{1}$ to $\ch$
\[
\widehat F|_{\spheres{1}} \cong F\circ \pi_1.
\]

Furthermore if $G\colon \fr \to \ab$ is another polynomial functor, then there is an isomorphism of graded groups
\begin{equation}\label{eq:Ext=Nat}
\ext^*_{\fun(\fr, \ab)}(F, G)\cong \pi_{-*}\left(\spectralNat\left(\widehat F, \widehat G\right)\right)
\end{equation}
where by $\spectralNat(\widehat F, \widehat G)$ we mean the spectral object of morphisms in the stable $\infty$-category $\fun(\spaces, \ch)$.
\end{mainthm}
\begin{notn}\label{notation:hat}
Throughout the paper, given a polynomial functor $F\colon\fr\to \ab$, we denote by $\widehat F$ the corresponding functor from $\spaces$ to $\ch$ that is given by Theorem~\ref{thm:fr-to-top_intro}.
\end{notn}
We refer to $\widehat F$ as the extension of $F$. But let us remind the reader that $\widehat F$ is neither a left nor a right Kan extension. Rather $\widehat F$ is a right Kan extension of a left Kan extension of $F$.
\begin{rem}\label{rem:little miracle} 
Theorem~\ref{thm:fr-to-top_intro} is useful for calculations 
because the right hand side of~\eqref{eq:Ext=Nat} is often easier to compute than the left hand side. The reader may wonder why this should be the case. We can offer a couple of heuristic reasons: 

First, we saw above that the right hand side of~\eqref{eq:Ext=Nat} can be identified with derived natural transformations between functors from $\epi$ to $\ch$, while on the left hand side we have functors from $\fr$ to $\ch$ (or $\ab$). It is clear that the category $\epi$ has simpler structure than $\fr$, so it stands to reason that the right hand side of~\eqref{eq:Ext=Nat} is easier to calculate than the left hand side.

The second, and perhaps deeper reason hinges on equivalence $\alpha$ on line~\eqref{eq:factorization}, i.e., on extending the domain category from connected spaces to all spaces. Consider for example the functor $P^n(X)=\widetilde\Z[X^n]$, which represents the homology of $X^n$. As a functor from connected spaces to $\ch$, $P^n$ is not a representable functor, and there is no easy way to calculate ${\spectralNat}_{\,\connected\!\!}(P^n, -)$. On the other hand, as a functor from  all spaces to $\ch$, $P^n$ is represented by the pointed set $n_+=\{0,1,\ldots, n\}$, and for any functor $G\in\fun(\spaces, \ch)$ the Yoneda lemma gives us the following equivalence, which amounts to a calculation of $\spectralNat_{\,\spaces\!\!}(P^n, G)$: 
\[
\spectralNat_{\,\spaces\!\!}(P^n, G)\simeq G(n_+).
\] 
It turns out that for many polynomial functors $F\colon\fr\to \ab$ that do not admit small projective resolutions in $\fun(\fr, \ab)$, their image $\widehat F \in \fun(\spaces, \ch)$ does admit a convenient presentation as a homotopy colimit of representable functors. This seems to be the ``little miracle'' that makes Theorem~\ref{thm:fr-to-top_intro} remarkably  useful for calculations.
\end{rem}

To illustrate the utility of Theorem~\ref{thm:fr-to-top_intro}, we list in Table~\ref{tab:functors} some commonly used polynomial functors from $\fr$ to $\ab$, and their extensions in $\fun(\spaces, \ch)$. 
{\def\arraystretch{1.5}
\begin{table}[] 
    \centering
    \begin{tabular}{c|ccl}
    {A functor $F\colon \fr\to \ab$}   & \multicolumn{3}{c}{The corresponding functor $\widehat F\colon \spaces\to \ch$} \\ \hline
       ${\abelianization}$  & $\widehat\abelianization(X)$ &$=$& $\Sigma^{-1} \widetilde \Z[X]$ \\
       $\tensor^n \circ \abelianization $ &$\widehat{\tensor^n\circ\abelianization} (X)$&$=$& $\Omega^n \widetilde \Z[X^{\wedge n}]$ \\
       $\Lambda^n \circ \abelianization $ & $\widehat{\Lambda^n \circ \abelianization}(X)$&$=$&$\Sigma^{-n} \widetilde \Z[X^{\wedge n}_{\Sigma_n}]$\\
       $\symm^n\circ \abelianization$ &  
       $\widehat{\symm^n\circ \abelianization}(X)$&$=$&$(\Omega^n \widetilde\Z[X^{\wedge n}])_{\Sigma_n}$\\
      $\Gamma^n\circ \abelianization$ &  $\widehat{\Gamma^n\circ \abelianization}(X)$&$=$&$\Sigma^{-2n}\widetilde\Z[(SX)^{\wedge n}_{\Sigma_n}]$   \\
       $\Passi_n$ & $ \widehat{\Passi_n}(X)$&$=$&$P_n \widetilde\Z[\Omega X]=\nat_{i\in \epi^{\le n}_0}(S^i, \widetilde \Z[X^{\wedge i}])$
    \end{tabular}
    \caption{A table of functors}
    \label{tab:functors}
\end{table}
}
Let us give a few explanations of the notation in the table
\begin{itemize}
    \item $\abelianization$ denotes the abelianization functor, 
    \item $\tensor^n$ denotes $n$-fold tensor power, \item $\Lambda^n$, $\symm^n$ and $\Gamma^n$ denote the $n$-th exterior, symmetric and divided power respectively, 
    \item $\widetilde \Z[-]$ denotes the functor from $\spaces$ to $\ch$ that represents reduced singular homology. Most of the time we work with $\infty$-categories, and it does not matter which model of this functor we use. For example, one can think of $\widetilde\Z[X]$ as  the Eilenberg-Mac Lane spectrum $H\mathbb Z\wedge X$, or the reduced singular chains complex of $X$, or the normalized chain complex of the pointed free abelian simplicial group generated by a pointed simplicial set $X$. Occasionally it will be convenient to use the latter construction as our specific model of $\widetilde\Z[X]$. 
    \item $\Passi_n$ denotes the $n$-th Passi functor (defined in Section~\ref{section:Passi}). 
    \item $P_n \widetilde\Z[\Omega X]$ is short for Goodwillie's $n$-th Taylor approximation of the functor $\widetilde\Z[\Omega -]$, evaluated at $X$.
    \item If $C$ is a chain complex and $n\in \Z$ then $\Sigma^nC$ denotes the $n$-fold shift of $C$.
    \item If $C$ is a (non-negatively graded) chain complex, then $\Omega^n C$ is defined as follows: starting with the standard simplicial model for $S^1$ define a simplicial model for $S^n$. Use the simplicial structure on $S^n$ to define $\Omega^n C$ as a cosimplicial chain complex. Normalize to get a (second quadrant) double complex, and then take total complex. 
    
    There is an equivalence $\Omega^n C\simeq \Sigma^{-n} C$. We favor the notation $\Omega^n C$ when we want to indicate that the symmetric group $\Sigma_n$ is acting non-trivially on the desuspension coordinates.
    \item $X^{\wedge n}_{\Sigma_n}$ denotes the strict orbit space of the action of $\Sigma_n$ on $X^{\wedge n}$. 
    \item The notation $(\Omega^n \widetilde\Z[X^{\wedge n}])_{\Sigma_n}$ indicates the strict coinvariants of the action of $\Sigma_n$ on the chain complex $\Omega^n \widetilde\Z[X^{\wedge n}]$. More details about it are given in Section~\ref{section:general formula}.
\end{itemize}


As we said in Remark~\ref{rem:little miracle}, most of the functors in the right column of  Table~\ref{tab:functors} admit explicit small resolutions in terms of representable functors, while the functors in the left column do not seem to admit such resolutions. Because of this, natural transformations are easier to calculate between functors in the right column than in the left column.

 The following result, originally due to C. Vespa~\cite{Vespa2018}, can be derived fairly easily from Theorem~\ref{thm:fr-to-top_intro} and Table~\ref{tab:functors}, plus the Yoneda lemma. Statement (1) of the corollary is proved in the paper as Corollary~\ref{cor: tensor-others}\ref{item:tensor-tensor}. Statements (2)-(4) are Proposition~\ref{prop:Vespa theorem 3}. 
\begin{cor}[\cite{Vespa2018}]\label{cor:previous}
There are isomorphisms of graded groups
\begin{enumerate}
\item 
\[
    \ext^i(T^m\circ \abelianization, T^n \circ \abelianization) \cong \left\{\begin{array}{cl} \Z^{\sur(n, m)} & i= n-m \\ 0 & \mbox{otherwise} \end{array}\right. 
\]
\item
\[
\ext^i(\Lambda^m\circ \abelianization, \Lambda^n \circ \abelianization)\otimes \Q \cong  
\left\{ \begin{array}{cl} \Q^{\Part(n, m)} & i=n-m \\ 0 &\mbox{otherwise} \end{array} \right.
\]
\item 
\[
\ext^i(\Lambda^m\circ \abelianization, \symm^n \circ \abelianization)\otimes \Q \cong  
\left\{\begin{array}{cc}
   \Q  & m=n\le 1, \, i=0  \\
   0  & \mbox{otherwise}
\end{array}\right. 
\]
\item \[
\ext^i(\Lambda^m\circ \abelianization, \tensor^n \circ \abelianization)\otimes \Q \cong  
   \left\{\begin{array}{cl} \Q^{\Stirling(n,m)} & i=n-m \\ 0 &\mbox{otherwise} \end{array}\right.
\]
\end{enumerate}
\end{cor}
Here $\sur(n, m)$ denotes the set of surjections from a set with $n$ elements to a set with $m$ elements, $\Stirling(n, m)=\sur(n,m)/_{\Sigma_m}$ is the set (or number) of ways to partition a \emph{set} with $n$ elements into $m$ parts, and $\Part(n,m)= {_{\Sigma_n}}\!\backslash\sur(n, m)/_{\Sigma_m}$ can be identified with the set of partitions of the \emph{number} $n$ with $m$ parts. 

We can use Theorem~\ref{thm:fr-to-top_intro} together with Table~\ref{tab:functors} to do many new calculations. As a general remark, modulo torsion many of calculations in this paper were done before, by Vespa and others. But all the calculations of torsion in this paper are new. 

In Section~\ref{section:ext from tensor} we calculate $\ext^*(\tensor^n\circ\ab, G)$ for several functors $G$. The main observation here is that $\ext^*(\tensor^n\circ\ab, G)$ is given by the cross-effects of $\widehat G$, and the cross-effects are usually easy to compute. The cases when $G=\tensor^n\circ \abelianization, \Lambda^n\circ \abelianization,$ or $\Gamma^n\circ \abelianization$ are calculated easily with our methods. The case when $G=\symm^n\circ\abelianization$ is less easy, because our formula for $\widehat{\symm^n\circ\abelianization}$ is not so explicit. We calculate $\ext^*(\abelianization, \symm^n\circ\abelianization)$ for $n\le 9$ (see Table~\ref{tab:ext}), and give a kind of general formula for $\ext^*(\tensor^m\circ \abelianization, \symm^n\circ\abelianization)$ in terms of $\ext^*(\abelianization, \symm^n\circ\abelianization)$. One of the reasons this calculation is interesting to the author is that it is our first example in which Ext groups have lots of torsion (and are trivial rationally).\footnote{Added in revision: further progress on $\ext^*(\abelianization, \symm^n\circ\abelianization)$ has been made in~\cite{Kim-Vespa} and also in work in progress by M. Kim, C. Vespa and the author.} 

In Section~\ref{section:ext from tensor} we also calculate $\ext^*(\abelianization, \Passi_n)$. This reduces to calculating the homology of the ``reduced Ran space'' of $S^1$. As far as we know, this calculation is new even rationally.

In Section~\ref{section:ext from exterior} we study $\ext^*(\Lambda^m\circ\abelianization, G)$ for various $G$. The key to calculating these groups is to understand $\nat(X^{\wedge m}_{\Sigma_m}, \widehat G)$. We accomplish this by writing a kind of ``small projective resolution'' of $X^{\wedge m
}_{\Sigma_m}$.

To illustrate, let us present a full calculation of $\ext^*(\Lambda^2\circ \abelianization, \Lambda^n \circ \abelianization)$. First, let us give a more general result. Suppose $G\colon \spaces\to \ch$ is a functor. Define $\ce_2G(S^0, S^0)$ to be the second cross-effects of $G$, i.e., the kernel of the homomorphism $G(S^0\vee S^0)\to G(S^0)\times G(S^0)$. Note that there is a natural homomorphism $\ce_2G(S^0, S^0)\to G(S^0)$. We prove the following fact in Examples~\ref{exmp:lambda 2 and 3}
\begin{prop}
    Let $G\colon\fr\to \ch$ be a polynomial functor. Let $\widehat G\colon \spaces\to \ch$ be the extension of $G$, in the sence of Theorem~\ref{thm:fr-to-top_intro}. The graded group $\ext^{*}(\Lambda^2\circ \abelianization, G)$ is isomorphic to $\pi_{-*-2}$ of the homotopy pullback of the following diagram of chain complexes\footnote{We use $\pi_*$ synonymously with $\HH_*$ when applied to chain complexes}
    \[
\widehat G(S^0) \to \pointedmaps(\RP^\infty_+, \widehat G(S^0))) \leftarrow \ce_2\widehat G(S^0, S^0)^{h\Sigma_2}.
    \]
\end{prop}
The proposition is proved by using an explicit resolution of the functor $X^{\wedge 2}_{\Sigma_2}$. 
As a corollary, we can calculate the groups $\ext^*(\Lambda^2\circ \abelianization, \Lambda^n\circ\abelianization)$ for all $n$ and $*$. The following lemma is Lemma~\ref{lem:Lambda2 to Lambdan} in the paper. 
\begin{lem} 
There is n isomorphism:
\[
\ext^i(\Lambda^2\circ \abelianization, \Lambda^n\circ\abelianization)\cong \left\{ \begin{array}{cl}
\Z^{\Part(n, 2)}    & i=n-2  \\
\Z/2   &  n<i, n \mbox{ is odd, } i \mbox{ is even.} \\
0 & \mbox{otherwise}
\end{array}\right.
\]   
\end{lem} 
Rationally this result agrees with Corollary~\ref{cor:previous} (2), but it shows that when $n$ is odd, $\ext^*(\Lambda^2\circ\abelianization, \Lambda^n\circ\abelianization)$ has $2$-torsion. It is worth mentioning that the torsion comes from the cohomology of $\RP^\infty$.

We also describe $\ext^*(\Lambda^3\circ\abelianization, \Lambda^n\circ \abelianization)$ in Corollary~\ref{cor:ext from lambda three to lambda n}. Unsurprisingly, in this case the answer has both $2$ and $3$-torsion. Just as the $2$-torsion part is captured by the cohomology of $\RP^\infty$, the $3$-torsion part is encoded in $\widetilde\HH^*(\classifying\Sigma_3/\classifying \Sigma_2)$. This is the classifying space of the collection of non-transitive subgroups of $\Sigma_3$.

In Theorem~\ref{thm:ext from exterior} we give a general formula for $\ext^*(\Lambda^m\circ\abelianization, G)$ as a homotopy limit of an explicit diagram of chain complexes. This gives a spectral sequence for calculating $\ext^*(\Lambda^m\circ\abelianization, G)$, whose $E^1$ and $E^2$ pages are accessible. We speculate that the spectral sequence collapses at $E^2$ in many cases of interest, for example when $G=\Lambda^n\circ \abelianization$.

In Section~\ref{section: ext from Passi} we investigate $\ext^*(\Passi_n, G)$. We show that it can be calculated in terms of the rank filtration of $\widehat G$. We apply the general formula to study $\ext^*(\Passi_m, \tensor^n\circ\abelianization)$ and obtain a refinement of a result of Vespa~\cite[Proposition 5]{Vespa2018}.

\subsubsection*{Applications to stable cohomology of $\aut(\free{n})$}
Let $\aut(\free{n})$ be the group of automorphisms of the free group on $n$ generators. Suppose $G\colon \fr\to \ab$ is a (polynomial) functor. Then for each $n$, $G(\free{n})$ is a module over $\aut(\free{n})$. By taking a suitable limit as $n\to \infty$ of $\HH^*(\aut(\free{n}); G(\free{n}))$ one may define the stable cohomology of $\aut(\free{n})$ with coefficients in $G$. We denote the stable cohomology by $\HH^*_s(\aut; G)$. By remarkable work of Djament~\cite{Djament2019} (which in turn relies on a famous theorem of Galatius~\cite{Galatius}), it is possible to calculate $\HH^*_s(\aut; G)$ in terms of $\ext^*(\tensor^m\circ\abelianization,  G)$. In Theorem~\ref{thm:Djament reinterpreted} we give the following reinterpretation of Djament's result:
\begin{mainthm}\label{thm:Djament reinterpreted_intro}
Let $G\colon \fr\to \ab$ be a polynomial functor of degree $n$. Let $\widehat G\colon \spaces\to \ch$ be the extension of $G$ given by Theorem~\ref{thm:fr-to-top_intro}. There is an isomorphism
\[
\HH^*_s(\aut;G)\cong \pi_{-*}\pointedmaps\left({\classifying\Sigma_\infty}_+, \prod_{d=0}^n \ce_d\widehat G(S^0, \ldots, S^0)^{h\Sigma_d}\right).
\]
\end{mainthm}
Here $\ce_d\widehat G$ denotes the $d$-th cross-effect of $\widehat G$. With help of Theorem~\ref{thm:Djament reinterpreted_intro} we calculate $\HH^*_s(\aut, \tensor^n\circ\abelianization)$, $\HH^*_s(\aut, \Lambda^n\circ\abelianization)$ and $\HH^*_s(\aut, \Gamma^n\circ\abelianization)$ in terms of cohomology of symmetric groups. See Proposition~\ref{prop:stable cohomology calculations}. This proposition refines rational results given in~\cite[Theorem 4]{Vespa2018}. 

\subsection*{Organization of the paper}
In Section~\ref{section:preliminaries} we go over some basics regarding $\infty$-categories. In particular, we review how Ext in the category of functors from a small category $\catname{I}$ to $\ab$ can be viewed as homotopy groups of spectral mapping objects in the $\infty$-category $\fun(\catname{I}, \ch)$.

In Section~\ref{section: fr to connected} we review the connection between the categories $\fr$ and $\connected$, via $\sgr$. We prove that the category of functors $\fun(\fr, \catname{D})$ is equivalent to the category of sifted colimits preserving functors $\fun_{\Sigma}(\connected, \catname{D})$.

In Section~\ref{section:cubical} we review, in preparation for discussion of polynomial and excisive functors, some notions surrounding cubical diagrams, such as cocartesian and strongly cocartesian diagrams. This is elementary material, but we need to pin down some details of the relation between $\infty$-categorical and $1$-categorical versions of these notions. We check that in some cases cocartesianness in an $\infty$-category can be detected in the homotopy category.

In Section~\ref{sec:polynomial} we review the notion of polynomial functors from a pointed category with coproducts to an idempotent complete additive $\infty$-category\footnote{Polynomial functors in this setting were considered by Johnson and McCarthy. See for example~\cite{JM03}}. We give a characterization of polynomial functors that is analogous to Goodwillie's definition of excisive functors.

In Section~\ref{section:excisive} we review Goodwillie's notion of excisive functors. We then focus on excisive functors from $\sgr$ or, equivalently, $\connected$, and compare them with polynomial functors from $\fr$ or, equivalently, $\circles$. We prove our first substantial result: the classifying space functor $\fr\to \connected$ induces an equivalence
\[
\exc_n(\connected, \catname{D})\xrightarrow{\simeq}\poly_n(\fr, \catname{D})
\]
from a category of \emph{excisive} functors to a category of \emph{polynomial} functors.

There is similarity and some overlap between our Sections~\ref{sec:polynomial} and~\ref{section:excisive} and~\cite[Section 2]{Barwick-Glasman-Mathew-Nikolaus}. However, we don't see that we could just quote the paper of Barwick-Glasman-Mathew-Nikolaus, because they only deal with functors whose domain is a stable $\infty$-category, while we are interested in functors whose domain category is not stable.

In Section~\ref{section:connected to all} we compare excisive functors on connected spaces with excisive functors on all spaces. We show that the inclusion $\connected\hookrightarrow \spaces$ induces an equivalence of categories
\[
\exc_n(\spaces, \catname{D})\xrightarrow{\simeq}\exc_n(\connected, \catname{D})
\]
and an inverse is given by right Kan extension. A version of this statment was proved by Brantner-Mathew~\cite[Theorem 3.36]{Brantner-Mathew}. We give a different proof, which works also when $\catname{D}$ is an unstable $\infty$-category.

In Section~\ref{section:main theorems} we prove Theorems~\ref{thm:main-general}, \ref{mainthm-modules} and~\ref{thm:fr-to-top_intro}. We also prove a version for bifunctors. 

Section~\ref{section:main theorems} finishes the first part of the paper which is devoted to proving the general results. The second part is devoted to calculations.

In Section~\ref{section:dictionary} we calculate $\widehat F$ for most functors in the left column of Table~\ref{tab:functors}. In some cases the easiest way to find $\widehat F$ is to ``guess'' the answer and check that it satisfies the condition of Theorem~\ref{thm:fr-to-top_intro}. In this way we calculate $\widehat F$ for $F=\tensor^n\circ\abelianization, \Lambda^n\circ\abelianization,$ and $\Gamma^n\circ\abelianization$. We also give a kind of general formula for $\widehat F$ when $F$ factors through abelianization. We use this to describe $\widehat{\symm^n\circ \ab}$. The general formula is of less topological nature and seems less easy to use that the special formulas we have for $\tensor^n\circ \abelianization$ etc., but it is perhaps not entirely useless. 

In section~\ref{section:Passi} we describe $\widehat{\Passi_n}$, i.e., the extensions of Passi functors. The Passi functors are our main example of functors $\fr\to\ab$ that do not factor through abelianization.

In Sections~\ref{section:ext from tensor},~\ref{section:ext from exterior} and~\ref{section: ext from Passi} we calculate $\ext$ between a few of the functors in Table~\ref{tab:functors}, reproducing and extending all previous calculations of this kind that we are aware of. 

Finally in Section~\ref{section:stable cohomology} we combine our results with those of Djament to describe the stable cohomology of $\aut(\free{n})$ with coefficients in a polynomial functor.

\subsection*{Acknowledgements} This paper is very much inspired by the work of Christine Vespa, and by  stimulating conversations we had during a visit to Strasbourg in the summer of 2013. It only took twelve years to write up! In the meantime I benefited from conversations with Lukas Brantner and also with Marco Nervo. Some of the calculations in Examples~\ref{exmp:lambda 2 and 3} first appeared in Marco's master thesis (University of Torino, 2021).

\part{Theory}

\section{Preliminaries and conventions regarding $\infty$-categories}\label{section:preliminaries}
In this section we review some preliminaries regarding $\infty$-categories. The main point of the section is to recall how classical Ext groups can be interpreted as homotopy (or homology) groups of spectral (or differential graded) mapping objects in a stable $\infty$-category.

Throughout the first half of the paper we use $\infty$-categorical language. Thus terms such as ``category'', ``limit'' and ``colimit'' are usually meant in $\infty$-categorical sense. Occasionally we may fall back into $1$-categorical usage of these terms, more so in the second part of the paper than the first. We will try our best to be clear as to which meaning is intended! For now, until further notice, terms are used in the $\infty$-categorical sense.

Given objects $X$, $Y$ of an $\infty$-category $\catname{C}$ we denote the space of maps from $X$ to $Y$ by $\map_{\catname{C}}(X, Y)$, or simply $\map(X, Y)$. This is a well-defined homotopy type. If $\catname{C}$ is a pointed $\infty$-category, then $\map_{\catname{C}}(X, Y)$ is a well-defined pointed homotopy type. For some categories we use a special notation for the space of maps. If $X$ and $Y$ are pointed spaces we sometimes use the notation $\pointedmaps(X, Y)$ to emphasise that we are considering the space of pointed maps. 

$\infty$-categories that have limits (resp. colimits) are cotensored (resp. tensored) over the category of spaces $\unpointed$. In a similar fashion, pointed $\infty$-categories with limits/colimits are cotensored/tensored over the category of pointed spaces $\spaces$. In practice all $\infty$-categories that we will encounter will be pointed. Given an object $X$ in a pointed $\infty$-category $\catname{C}$ and a pointed space $K$ we denote the tensoring and cotensoring of $K$ and $X$ by $K\otimes X$ (or $X\otimes K$), and $\pointedmaps(K, X)$ respectively.

Given $\infty$-categories $\catname{C}$ and $\catname{D}$, let $\fun(\catname{C}, \catname{D})$ denote the $\infty$-category of functors from $\catname{C}$ to $\catname{D}$. If $F, G\colon \catname{C}\to \catname{D}$ are objects of $\fun(\catname{C}, \catname{D})$ we may denote the space $\map_{\fun(\catname{C}, \catname{D})}(F, G)$ by $\nat_{\catname{C}}(F, G)$, or $\nat_X(F(X), G(X))$, or just $\nat(F, G)$. 

If $\catname{C}$ is a stable $\infty$-category then the mapping space $\map_{\catname{C}}(X, Y)$ is naturally an infinite loop space. The $n$-fold delooping of $\map_{\catname{C}}(X, Y)$ is equivalent to $\map_{\catname{C}}(X, \Sigma^n Y)$. Note that in~\cite{LurieHA} Lurie sometimes denotes $\pi_0\left(\map_{\catname{C}}(X, \Sigma^n Y)\right)$ by $\ext^n_{\catname{C}}(X, Y)$. As we shall recall shortly (Lemma~\ref{lem:ext-pi} below), when $\catname{C}$ is the category of chain complexes, this recovers the classic Ext groups.

The infinite delooping of $\map_{\catname{C}}(X, Y)$ is called the {\it mapping spectrum} from $X$ to $Y$. We will denote the mapping spectrum by 
$\spectralmaps\hphantom{.}_{\!\!\! \catname{C}}(X, Y)$ or just $\spectralmaps(X, Y)$. When $\catname{C}$ is stable, $\spectralmaps(X, Y)$ is a well-defined stable homotopy type, and there is a natural equivalence $\map(X, Y)\simeq \Omega^\infty \spectralmaps(X, Y)$.

In particular if $\catname{C}$ is a small $\infty$-category and $\catname{D}$ is a stable $\infty$-category then $\fun(\catname{C}, \catname{D})$ is a stable $\infty$-category. Given two objects $F, G\in \fun(\catname{C}, \catname{D})$ we denote the spectral mapping object from $F$ to $G$ by $\spectralNat_{\catname{C}}(F, G)$, or $\spectralNat_X(F(X), G(X))$, or just $\spectralNat(F, G)$.

Given a functor $\alpha\colon \catname{C}_1\to\catname{C}_2$, we denote the corresponding restriction functor $\fun(\catname{C}_2, \catname{D})\to \fun(\catname{C}_1, \catname{D})$ by $\rho_\alpha$, or just $\rho$. The left and right adjoints to restriction, a.k.a the left and right Kan extensions from $\catname{C}_1$ to $\catname{C}_2$ are denoted by $\lkan_\alpha$ and $\rkan_\alpha$, or just $\lkan$ and $\rkan$.

We will consider ordinary categories as $\infty$-categories via the nerve construction. More generally we will implicitly convert simplicially or topologically enriched $1$-categories into $\infty$-categories via the simplicial or topological nerve construction~\cite[Definition 1.1.5.5]{LurieHTT}. Usually we will not distinguish notationally between a category $\catname{C}$ and its (simplicial) nerve, and denote both by $\catname{C}$.

Aside from simplicial and topological categories, there is a well-known way to associate an $\infty$-category with a Quillen model category, which we will review now.

\subsubsection*{The underlying $\infty$-category of a model category} \label{section:model-to-infinity} Suppose $\catname{C}$ is a Quillen model category. Following Lurie~\cite{LurieHTT}, we denote by $\catname{C}^\circ$ the full subcategory of $\catname{C}$ consisting of fibrant-cofibrant objects, or any full subcategory thereof that contains at least one representative of each homotopy equivalence class. If $\catname{C}$ is a combinatorial simplical model category, then $\catname{C}$ is considered as an $\infty$-category via the simplicial nerve of $\catname{C}^\circ$. The $\infty$-category associated with $\catname{C}$ is denoted by $\Nerve(\catname{C}^{\circ})$ in~\cite[Appendix A.2]{LurieHTT}. We will generally not distinguish notationally between a model category $\catname{C}$ and the $\infty$-category $\Nerve(\catname{C}^{\circ})$, and denote both simply by $\catname{C}$, except in this section, where we have to explain precisely which construction we are using.

When $\catname{C}$ is a simplicial model category, the $\infty$-category $\Nerve(\catname{C}^{\circ})$ is the localisation (in the sense of~\cite[Definition 1.3.4.1]{LurieHA}) of $\catname{C}$ at the class of weak equivalences~\cite[Theorem 1.3.4.20 + Remark 1.3.4.16]{LurieHA}. This is the justification for calling $\Nerve(\catname{C}^{\circ})$ ``the underlying $\infty$-category of $\catname{C}$''. Furthermore, the construction of underlying $\infty$-category can be extended to more general model categories. Given a category $\catname{C}$ with a class of weak equivalences $\mathcal W$, one can form the Dwyer-Kan localisation (a.k.a the hammock localisation) of $\catname{C}$ at $\mathcal W$. This is a simplicial category, which we denote by $\DK(\catname{C})$. Using the simplicial nerve construction, we may also consider $\DK(\catname{C})$ as an $\infty$-category. When $\catname{C}$ is a combinatorial simplicial model category, there is a natural equivalence of $\infty$-categories 
\begin{equation}\label{eq: equivalence} 
\Nerve(\catname{C}^\circ)\xrightarrow{\simeq} \DK(\catname{C}).
\end{equation}
This is part of what is proved in~\cite[Section 1.4.3]{HinichDK}. The category $\DK(\catname{C})$ is well-defined for an arbitrary model category $\catname{C}$. Thus, one may define $\DK(\catname{C})$ to be the $\infty$-category associated with a general model category $\catname{C}$. Let us repeat that in the rest of the paper we will denote $\DK(\catname{C})$  simply by $\catname{C}$, but not in this section.

 It is known that a Quillen adjunction between model categories induces an adjunction of the associated $\infty$-categories~\cite[Proposition 1.5.1]{HinichDK}. It follows that a Quillen equivalence induces an equivalence of $\infty$-categories. It also is known that at least when $\catname{C}$ is combinatorial, the construction $\catname{C}\mapsto \DK(\catname{C})$ takes homotopy (co)limits in $\catname{C}$ to $\infty$-categorical (co)limits in $\DK(\catname{C})$~\cite[Corollary 1.5.2]{HinichDK}. 
\begin{rem} \label{rem:stable}
 Recall that an $\infty$-category $\catname{C}$ is {\it stable} if it has finite limits and colimits, and a square diagram in $\catname{C}$ is a pushout if and only if it is a pullback~\cite[Proposition 1.1.3.4]{LurieHA}. Similarly, a model category is stable if it has finite homotopy limits and colimits, and a square diagram in it is a homotopy pushout if and only if it is a homotopy pullback. It follows that the $\infty$-category associated with a combinatorial stable model category is stable.
\end{rem}
Furthermore, Hinich proves that holds when $\catname{C}$ is a model category with a weaker notion of simplicial enrichment than simplicial model category. The definition of a weak simplicial enrichment is given in~\cite[Definition 1.4.2]{HinichDK}. What matters to us is that the category of unbounded chain complexes $\ch$ with the projective module structure satisfies it. 

Let $\catname{A}$ be an abelian $1$-category and $\ch_{\catname{A}}$ the category of chain complexes in $\catname{A}$. Recall that the projective model structure on $\ch_{\catname{A}}$ is the one where weak equivalences are quasi-isomorphisms and fibrations are level-wise surjections. If $\catname{A}$ has a set of compact projective generators, then the projective model structure on $\ch_{\catname{A}}$ exists, and is cofibrantly generated, and therefore clearly combinatorial. For a proof of this see for example~\cite{Christensen-Hovey}. Given the projective model structure, we may convert $\ch_{\catname{A}}$ into an $\infty$-category by the same. The $\infty$-category $\ch_{\catname{A}}$ is stable, and it has a well-known $t$-structure, whose heart $\ch_{\catname{A}}^{\heartsuit}$ is the full subcategory consisting of chain complexes whose homology is concentrated in degree zero. It is well known that there is an equivalence of $\infty$ categories $\ch_{\catname{A}}^{\heartsuit}\simeq \catname{A}$, and thus $\catname{A}$ embeds fully faithfully in $\ch_{\catname{A}}$ as the heart.

The category $\ch_{\catname{A}}$ is enriched over $\ch$ (in the $1$-categorical sense), and one can convert the enrichment over $\ch$ into enrichment over simplicial sets by first taking $-1$-connected cover and then applying the Dold-Kan functor. See \cite[Page 79]{LurieHA} for more details on this. The simplicial enrichment of $\ch_{\catname{A}}$  does not make $\ch_{\catname{A}}$ into a simplicial model category, because $\ch_{\catname{A}}$ is not tensored over simplicial sets, but only ``weakly'' tensored in a certain sense. But the simplicial enrichment of $\ch_{\catname{A}}$ does satisfy Hinich's axioms of weak enrichment. See the first Example on page 37 of~\cite{HinichDK}. Hinich only states it for chain complexes of modules over a ring, but it is clear that it holds for chain complexes over a general abelian category with enough projectives. As a result, we obtain the following lemma
\begin{lem}\label{lem:DK for ch}
Let $\catname{A}$ be an abelian category with enough projectives, and let $\ch_{\catname{A}}$ be the category of chain complexes over $\catname{A}$ equipped with the projective model structure. Then the equivalence of~\eqref{eq: equivalence} holds for $\ch_{\catname{A}}$. I.e., there is an equivalence of $\infty$-categories
\[
\Nerve(\ch_{\catname{A}}^\circ)\xrightarrow{\simeq} \DK(\ch_{\catname{A}}).
\]
\end{lem}
It follows that for two chain complexes $P$ and $Q$, the space of maps from $P$ to $Q$ in the $\infty$-category of chain complexes is naturally equivalent to the simplicial mapping space from a projective cofibrant replacement of $P$ to $Q$. As a result, we obtain the well-known statement 
\begin{lem}\label{lem:ext-pi}
Let $X, Y$ be objects of $\catname{A}$, which we can also consider as objects of $\ch_{\catname{A}}$ via the natural embedding of $\catname{A}$ into $\ch_{\catname{A}}$. There is a natural isomorphism. 
\begin{equation}\label{eq: ext-pi}
\ext^*_{\catname{A}}(X, Y)\cong \pi_{-*}\left(\spectralmaps \hphantom{.}_{\!\!\!\ch_{\catname{A}}}(X, Y)\right).
\end{equation}
\end{lem}
\begin{proof}
The isomorphism is obtained from the following chain of isomorphisms ($\cofibrant{X}$ denotes a cofibrant replacement of $X$, i.e., a projective resolution of $X$)
\[
\ext^n_{\catname{A}}(X, Y)=[\cofibrant{X}, \Sigma^nY]=\pi_0\map_{\ch_{\catname{A}}}(X, \Sigma^nY)\cong \pi_{-n}\left(\spectralmaps_{\ch_{\catname{A}}}\!\!(X, Y)\right).
\]
\end{proof}


 \subsection{Homological and homotopical algebra of functor categories}\label{sec:homological} 
  We continue with the convention that $\catname{A}$ is an abelian $1$-category with enough projectives, and $\ch_{\catname{A}}$ denotes the category of chain complexes over $\catname{A}$. 

Now suppose $\catname{I}$ is a small $1$-category. Let $\fun(\catname{I}, \catname{A})$ denote the $1$-category of functors. Then $\fun(\catname{I}, \catname{A})$ is again an abelian category with enough projectives. If $\catname{A}$ has a set of compact projective generators, then so does $\fun(\catname{I}, \catname{A})$, and it follows that both categories $\ch_{\fun(\catname{I}, \catname{A})}$ and $\fun(\catname{I}, \ch_{\catname{A}})$ have a projective model structure, which is combinatorial. There is an obvious isomorphism of $1$-categories 
\begin{equation}\label{eq: complexes of functors}
\ch_{\fun(\catname{I}, \catname{A})}\cong \fun(\catname{I}, \ch_{\catname{A}}).
\end{equation}
Furthermore, this isomorphism induces a bijection between the projective model structures on both sides. Indeed, it is obvious that the weak equivalences and fibrations on both sides correspond to each other, and therefore the cofibrations correspond to each other as well. Subsequently, we will consider $\fun(\catname{I}, \ch_{\catname{A}})$ as a stable $\infty$-category. Right now, let us denote the $\infty$-category associated with $\fun(\catname{I}, \ch_{\catname{A}})$ by $\Nerve(\fun(\catname{I}, \ch_{\catname{A}})^{\circ})$. 

Suppose $\catname{C}$ is a simplicial model category. There is a canonical functor
\begin{equation}\label{eq:functor category presentation}
\Nerve(\fun(\catname{I}, \catname{C})^{\circ})\to \fun(\catname{I}, \Nerve(\catname{C}^{\circ}))
\end{equation}
At least if $\catname{C}$ is a combinatorial simplicial model category, this functor is an equivalence by~\cite[Proposition 4.2.4.4]{LurieHTT}. On the other hand, by the results of~\cite{Rezk-Schwede-Shipley} or~\cite{Dugger01}, every combinatorial model category $\catname{C}$ is Quillen equivalent to a combinatorial simplicial model category. Suppose there is a Quillen equivalence between $\catname{C}$ and a combinatorial simplicial category $\catname{C}^s$. Then it induces a Quillen equivalence between $\fun(\catname{I}, \catname{C})$ and $\fun(\catname{I}, \catname{C}^s)$. These Quillen equivalences, together with~\eqref{eq:functor category presentation} induce a chain of equivalences of $\infty$-categories
\begin{equation}\label{eq: functor categories}
\Nerve(\fun(\catname{I}, \catname{C})^{\circ})\simeq\Nerve(\fun(\catname{I}, \catname{C}^s)^{\circ})\simeq \fun(\catname{I}, \Nerve({\catname{C}^s}^{\circ}))\simeq \fun(\catname{I}, \Nerve(\catname{C}^{\circ}))
\end{equation}
In particular equivalences~\eqref{eq: functor categories} hold when $\catname{C}=\ch_{\catname{A}}$. Let us write it out explicitly: there is an equivalence of categories
\begin{equation}\label{eq:functors into chain complexes presentation}
\Nerve(\fun(\catname{I}, \ch_{\catname{A}})^{\circ}) \simeq \fun(\catname{I}, \Nerve(\ch_{\catname{A}}^{\circ}))
\end{equation}
In other words, the projective model structure on $\fun(\catname{I}, \ch_{\catname{A}})$ is  a presentation of the $\infty$-category of functors $\fun(\catname{I}, \ch_{\catname{A}})$. 

To summarize, we have the following well-known statement, which is however crucial to the whole paper.
\begin{lem}\label{lem:key isomorphism}
Let $\catname{I}$ be a small $1$-category, and $\catname{A}$ an abelian $1$-category with a set of projective generators. Let $F, G\colon\catname{I}\to \catname{A}$ be functors. We have an isomorphism 
\begin{equation}\label{eq:key isomorphism}
\ext^*_{\fun(\catname{I}, \catname{A})}(F, G)\cong\pi_{-*}\left(\spectralNat(F, G)\right).
\end{equation}
\end{lem}
\begin{proof}
The result follows by combining Lemma~\ref{lem:ext-pi} and equivalence~\eqref{eq:functors into chain complexes presentation}
\end{proof}

\subsection{Sifted colimits}\label{section:sifted} 
Sifted categories are defined in~\cite[Definition 5.5.8.1]{LurieHTT}. A sifted colimit in a category $\catname{C}$ is a colimit of a functor from a sifted category to~$\catname{C}$. We do not need to repeat the definition of a sifted category, but for example filtered colimits are sifted (\cite[Proposition 5.3.1.22]{LurieHTT}), and simplicial colimits (also known as ``geometric realisations'') are sifted~\cite[Lemma 5.5.8.4]{LurieHTT}. Furthermore, these examples generate all sifted colimits in a sense that can be made precise as follows: suppose $\catname{C}, \catname{D}$ are categories that have small colimits. Then a functor $f\colon \catname{C}\to\catname{D}$ preserves sifted colimits if and only if it preserves filtered colimits and simplicial colimits~\cite[Corollary 5.5.8.17]{LurieHTT}. We let $\fun_{\Sigma}(\catname{C}, \catname{D})\subset \fun(\catname{C}, \catname{D})$ denote the full subcategory consisting of functors that preserve sifted colimits. We say more about sifted colimits in Section~\ref{section:sifted completion}, and for more details see~\cite[Section 5.5.8]{LurieHTT}.

\section{From free groups to connected spaces}\label{section: fr to connected}
Throughout this section, let $\catname{D}$ be a fixed $\infty$-category, that has all small colimits. Recall that $\fr$ is the category of finitely generated free groups. Let $\connected$ denote the category of \emph{connected} pointed spaces. In this section we compare the functor categories $\fun(\fr, \catname{D})$ and $\fun(\connected, \catname{D})$. The main result of this section is Theorem~\ref{theorem: fr to connected} below. 

Let $\classifying\colon \fr \to \connected$ be the classifying space functor. It is well-defined up to natural equivalence, and we may consider it a functor between $\infty$-categories. It induces a restriction functor
\[
\rho_{\classifying}\colon \fun(\connected, \catname{D}) \to \fun(\fr, \catname{D}).
\]
Let $\fun_{\Sigma}(\connected, \catname{D})\subset \fun(\connected, \catname{D})$ be the full subcategory spanned by functors that preserve sifted colimits (Section~\ref{section:sifted}). The following is the main result of this section
\begin{thm}\label{theorem: fr to connected}
Composition with the classifying space functor induces an equivalence of categories
\[
\rho_B\colon \fun_{\Sigma}(\connected, \catname{D}) \xrightarrow{\simeq} \fun(\fr, \catname{D}).
\]
\end{thm}
This theorem is essentially well-known. For example, it follows more or less immediately from~\cite[Corollary 5.2.6.21]{LurieHA}. We will give the proof both for the sake of being self-contained, and because some of the details that come up in the process of proving Theorem~\ref{theorem: fr to connected} will be used elsewhere in the paper. The theorem is proved at the end of the section. First we have to do some preparatory work.
\subsection{From free groups to wedges of circles} Let $\circles\subset \spaces$ be a full subcategory of pointed spaces spanned by finite wedge sums of $S^1$. We may consider $\circles$ as being spanned by objects of the form $\classifying \free{n}$, where $\free{n}$ is the free group on $n$ generators. It follows that the classifying space functor may be considered as a functor from $\fr$ to $\circles$. Indeed, the classifying space functor induces an equivalence of categories:
\begin{equation}\label{eq: classifying equivalence}
\classifying\colon  \fr \xrightarrow{\simeq} \circles.
\end{equation}
Here when we say that a functor is an equivalence of categories, we mean it in the sense of~\cite[Definition 1.1.3.6]{LurieHTT}: it is essentially surjective on objects, and for any two objects $G_1, G_2$ of $\fr$, $\classifying$ induces an equivalence of spaces (where the source has the discrete topology)
\[
 \hom(G_1, G_2) \xrightarrow{\simeq}\pointedmaps(\classifying G_1, \classifying G_2).
\]
An explicit inverse to equivalence~\eqref{eq: classifying equivalence} is given by the fundamental group functor
\[\label{eq: pi_1 equivalence}
\pi_1\colon \circles \xrightarrow{\simeq} \fr.
\]

It follows that the fundamental group functor and the classifying space functor induce inverse equivalences of functor categories (\cite[Proposition 1.2.7.3 (3)]{LurieHTT})
\begin{equation}\label{eq: functor equivalence}
\rho_{\pi_1} : \begin{tikzcd}
           \fun(\fr, \catname{D})
           \arrow[r, shift left=.75ex]
            \arrow[r, phantom, "\simeq"] 
           & \fun(\circles, \catname{D})
           \arrow[l, shift left=.75ex]      
        \end{tikzcd}  : \rho_{\classifying}
\end{equation}
\subsection{From wedges of circles to connected spaces} We claim that the category $\connected$ is obtained from $\circles$ by {\it freely adjoining sifted colimits}. To make this precise, we have to digress to recall Lurie's construction of a free completion of a category under sifted colimits.

\subsubsection{Sifted completion}\label{section:sifted completion} Let $\catname{C}$ be a category that admits finite coproducts. Let $\presh(\catname{C})=\fun(\catname{C}^{\op}, \spaces)$ be the category of presheaves on $\catname{C}$. The category $\presh_{\Sigma}(\catname{C})$ is defined to be the full subcategory of $\presh(\catname{C})$ consisting of presheaves that take finite coproducts to finite products~\cite[Definition 5.5.8.8]{LurieHTT}.  Proposition 5.5.8.10 in~\cite{LurieHTT} implies, among other things, that $\presh_{\Sigma}(\catname{C})$ has all small colimits, and the inclusion $\presh_{\Sigma}(\catname{C}) \to \presh(\catname{C})$ preserves sifted colimits. It is clear that $\presh_{\Sigma}(\catname{C})$ contains the essential image of $\catname{C}$ in $\presh(\catname{C})$ under the Yoneda embedding. Lemma 5.5.8.14 of \cite{LurieHTT} says that $\presh_{\Sigma}(\catname{C})$ is (equivalent to) the closure of $\catname{C}$ in $\presh(\catname{C})$ spanned by sifted (i.e. filtered and simplicial) 
colimits. Moreover, $\presh_{\Sigma}(\catname{C})$ is the category obtained by \emph{freely} adjoining sifted colimits to $\catname{C}$. This is made precise in the following lemma. Recall that $\fun_{\Sigma}(-, -)$ denotes the category of functors that preserve sifted colimits.
\begin{lem}[\cite{LurieHTT}, Proposition 5.5.8.15]\label{lem:sifted yoneda}
Let $\catname{C}$ be a small $\infty$-category which admits finite
coproducts and let $\catname{D}$ be an $\infty$-category which admits sifted colimits. The Yoneda embedding $C\hookrightarrow \presh(\catname{C})$ factors through a fully faithful embedding $j\colon \catname{C}\hookrightarrow \presh_{\Sigma}(\catname{C})$. Restriction along $j$ induces an
equivalence of categories
\[
\fun_{\Sigma}(\presh_{\Sigma}(\catname{C}), \catname{D}) \to \fun(\catname{C}, \catname{D}).
\]
The inverse of this equivalence is given by left Kan extension.
\end{lem}
\begin{cor}\label{cor:psigma}
Let $\catname{C}$ be a small $\infty$-category which admits finite
coproducts and $\catname{D}$ an $\infty$-category which admits sifted colimits. Any functor $f\colon \catname{C} \to \catname{D}$ factors essentially uniquely as a composition
\[
\catname{C}\xrightarrow{j} \presh_{\Sigma}(\catname{C})\xrightarrow{F} \catname{D},
\]
where $j$ is as defined above, and $F$ preserves sifted colimits. Furthermore, $F$ is a left Kan extension of $f$.
\end{cor}
Furthermore~\cite[Proposition 5.5.8.22]{LurieHTT} gives necessary and sufficient conditions for the functor $F$ above to be  an equivalence of categories
\begin{lem}[\cite{LurieHTT}, Proposition 5.5.8.22]\label{lemma:equivalence conditions}
Let $\catname{C}$ and $\catname{D}$ be as in Corollary~\ref{cor:psigma}, and suppose we have a functor $f\colon\catname{C}\to \catname{D}$. The induced functor $F\colon \presh_{\Sigma}(\catname{C})\to \catname{D}$ is an equivalence if the following conditions hold:
\begin{enumerate}
\item The functor $f$ is fully faithful. \label{faithful}
\item For every object $X$ of $\catname{C}$ the represented functor  \label{sifted}
\[
\map_{\catname{D}}(f(X) , -)\colon \catname{D}\to \catname{Spaces}
\] 
preserves sifted colimits.
\item The category $\catname{D}$ is generated by the essential image of $f$ under sifted colimits. \label{generates}
\end{enumerate}
\end{lem}
\subsubsection*{Back to topological spaces} It follows from Corollary~\ref{cor:psigma} that the inclusion $\circles\to \connected$ factors essentially uniquely as a composition
\begin{equation}\label{eq: inclusion factorisation}
\circles \to \presh_{\Sigma}(\circles) \to \connected
\end{equation}
where the first functor is the Yoneda embedding, and the second functor preserves sifted colimits. Next, we claim that the second functor is in fact an equivalence 
\begin{lem}\label{lemma: connected generated by circles}
The functor $\presh_{\Sigma}(\circles) \to \connected$ in~\eqref{eq: inclusion factorisation} is an equivalence of categories.
\end{lem}
\begin{proof}
The lemma is essentially given by~\cite[Corollary 5.2.6.21]{LurieHA}, which asserts that $\presh_{\Sigma}(\circles)$ and $\connected$ are equivalent categories. It can also be deduced directly from Lemma~\ref{lemma:equivalence conditions}. Indeed, this lemma gives three conditions that one needs to verify. Condition~\eqref{faithful} is obvious. Condition~\eqref{sifted} says that if $X$ is a finite wedge of circles, then the represented functor $\pointedmaps(X, -)$ preserves filtered colimits and geometric realisations of simplicial objects in $\connected$. It preserves filtered colimits because $X$ is compact. That it preserves geometric realisations follows from the well-known fact that the loop space functor $\Omega$ preserves geometric realisations of simplicial objects in $\connected$. This is true because $\Omega$ is an inverse to the classifying space functor, which has a simplicial model that clearly preserves geometric realisations.

Finally condition~\eqref{generates} says that every object of $\connected$ is equivalent to a sifted colimit of objects of $\circles$. This follows (for example) from the well-known fact that for every connected pointed space $X$, $\Omega X$ is weakly equivalent to the geometric realisation of a simplicial group that is free in each degree (this goes back to Kan~\cite{KanLoop}). Thus $X$ is equivalent to the geometric realisation of a simplicial space that is equivalent to a wedge of circles in each degree.
\end{proof}

Since $\circles$ is a full subcategory of $\connected$
\[
\fun(\connected, \catname{D}) \to \fun(\circles, \catname{D})
\]
has a fully faithful left adjoint (also known as the left Kan extension). We can now describe the image of the left adjoint. It consists of functors that preserve sifted colimits. This is stated formally in the following proposition.
\begin{prop}\label{prop:circles to connected}
Let $\catname{D}$ be a category that admits small colimits. The inclusion $\circles\subset \connected$ induces an equivalence of categories
\[
\fun_{\Sigma}(\connected, \catname{D}) \xrightarrow{\simeq} \fun(\circles, \catname{D}).
\]
The inverse is given by left Kan extension, which we denote by 
\[
\lkan\colon \fun(\circles, \catname{D})\xrightarrow{\simeq}  \fun_{\Sigma}(\connected, \catname{D}).
\]
\end{prop}
\begin{proof}
We saw in~\eqref{eq: inclusion factorisation} that the inclusion $\circles\to \connected$ factors essentially uniquely as a composition
\[
\circles \to \presh_{\Sigma}(\circles) \to \connected.
\]
These induce restriction functors
\[
\fun_{\Sigma}(\connected, \catname{D}) \to \fun_{\Sigma}(\presh_{\Sigma}(\circles), \catname{D}) \to \fun(\circles, \catname{D}).
\]
The first of these restrictions is an equivalence by Lemma~\ref{lemma: connected generated by circles} and the second by Lemma~\ref{lem:sifted yoneda}.
\end{proof}

\subsection{Connection with simplicial groups} 
It is well-known that the $1$-category of simplicial groups is equipped with a projective model structure, where fibrations and weak equivalences are those of the underlying simplicial sets~\cite{QuillenHA},~\cite[Chapter II, Example 5.2]{Goerss-Jardine}. We denote by $\sgr$ the $\infty$-category associated with this model structure.

There is a fully faithful embedding $\const\colon \fr\hookrightarrow \sgr$ that sends a group to the associated constant simplicial group. The classifying space functor $\classifying \colon \fr \to \circles$ extends to a functor, which we denote by the same letter $\classifying \colon \sgr \to \connected$. The latter functor is defined by taking levelwise classifying space of a simplicial group and then taking geometric realisation. On the level of model categories it preserves weak equivalences, and therefore it induces a functor between the associated $\infty$-categories. We obtain a diagram of categories that commutes up to a natural equivalence
\[\begin{tikzcd}
	\fr & \circles \\
	\sgr & \connected
	\arrow["\classifying", from=1-1, to=1-2]
	\arrow["\const"', from=1-1, to=2-1]
	\arrow[hook, from=1-2, to=2-2]
	\arrow["\classifying"', from=2-1, to=2-2]
\end{tikzcd}\]
We already saw that because of Corollary~\ref{cor:psigma} the inclusion $\circles\to \connected$ factors essentially uniquely as a composition $\circles \to \presh_{\Sigma}(\circles) \to \connected$. Also by Corollary~\ref{cor:psigma}, the functor $\const\colon \fr\to \sgr$ factors essentially uniquely as $\fr\to \presh_{\Sigma}(\fr) \to \sgr$. Furthermore, the composition functor $\fr \xrightarrow{\classifying} \circles \hookrightarrow \presh_{\Sigma}(\circles)$ factors essentially uniquely through $\presh_{\Sigma}(\fr)$. Because of the uniqueness of factorisations, the following diagram of categories, which refines the previous one, commutes up to a natural equivalence
\begin{equation}\label{diagram:categories}
\begin{tikzcd}
	\fr & \circles \\
	{\presh_{\Sigma}(\fr)} & {\presh_{\Sigma}(\circles)} \\
	\sgr & \connected
	\arrow["\simeq"', "\classifying", from=1-1, to=1-2]
	\arrow[from=1-1, to=2-1]
	\arrow[from=1-2, to=2-2]
	\arrow["\simeq"', from=2-1, to=2-2]
	\arrow["\simeq"', from=2-1, to=3-1]
	\arrow["\simeq", from=2-2, to=3-2]
	\arrow["\simeq"', "\classifying", from=3-1, to=3-2]
\end{tikzcd}
\end{equation}
\begin{lem}\label{lemma:equivalences}
Every arrow marked with $\simeq$ in Diagram~\eqref{diagram:categories} is an equivalence of categories. 
\end{lem}
\begin{proof}
That the middle arrow is an equivalence follows from the fact that the top arrow is an equivalence. That the bottom right arrow is an equivalence is Lemma~\ref{lemma: connected generated by circles}. It is of course well-known that the bottom arrow is an equivalence. This goes back to Kan~\cite{KanLoop}. One can also prove directly that the bottom left arrow is an equivalence, very similarly to the proof of Lemma~\ref{lemma: connected generated by circles}, using Lemma~\ref{lemma:equivalence conditions}. It then follows that the bottom arrow is an equivalence.
\end{proof}
Now we are ready to prove the main result of this section.
\begin{proof}[Proof of Theorem~\ref{theorem: fr to connected}]
Diagram~\eqref{diagram:categories} induces the following diagram of functor categories, where the horizontal arrows are induced by the classifying space functor, and all arrows are equivalences (for the lower vertical arrows this follows from Lemma~\ref{lem:sifted yoneda}):
\begin{equation}\label{diagram:functors}
\begin{tikzcd}
	{\fun_{\Sigma}(\connected, \catname{D})} & {\fun_{\Sigma}( \sgr, \catname{D})} \\
	{\fun_{\Sigma}( \presh_{\Sigma}(\circles), \catname{D})} & {\fun_{\Sigma}( \presh_{\Sigma}(\fr), \catname{D})} \\
	{\fun(\circles, \catname{D})} & {\fun(\fr, \catname{D})}
	\arrow["\simeq"', swap, "\rho_{\classifying}", from=1-1, to=1-2]
	\arrow["\simeq", from=1-1, to=2-1]
	\arrow["\simeq"', from=1-2, to=2-2]
	\arrow["\simeq", from=2-1, to=2-2]
	\arrow["\simeq", from=2-1, to=3-1]
	\arrow["\simeq"', from=2-2, to=3-2]
	\arrow["\simeq"',  swap, "\rho_{\classifying}", from=3-1, to=3-2]
\end{tikzcd}
\end{equation}
Going from top left to bottom right corner gives the desired equivalence.
\end{proof}

\section{Cubical diagrams}\label{section:cubical}
We saw that the inclusion $\fr\to \sgr$ induces an equivalence of functor categories $\fun_{\Sigma}(\sgr, \catname{D})\xrightarrow{\simeq} \fun(\fr, \catname{D})$. Our next goal is to prove that when $\catname{D}$ is stable, this equivalence restricts to an equivalence from the category of \emph{excisive} functors from $\sgr$ to $\catname{D}$ to the category of \emph{polynomial} functors from $\fr$ to $\catname{D}$. 

The notions of polynomial and excisive functors use cubical diagrams and cross-effects. In this section we review some generalities regarding cubical diagrams. This is elementary material, but we need to pinpoint some facts about the relationship between being cocartesian in an $\infty$-category and in its homotopy category.

\begin{defn}
Given a set $U$, let $\power(U)$ denote the category (poset) of \emph{finite} subsets of $U$ and inclusions. Most of the time $U$ will itself be finite, in which case $\power(U)$ is the category of all subsets. We write $\power(n)$ for the category of subsets of the set $\underline{n}:=\{1, \ldots, n\}$. Given an integer $i$, let $\power(U)_{\le i}$ denote the subcategory of $\power(U)$ consisting of sets of cardinality at most $i$. 

Given a category $\catname{D}$, an $n$-dimensional cubical diagram in $\catname{D}$ is a functor $\chi\colon \power(n)\to \catname{D}$. 
\end{defn}
Next, let us recall the well-known definitions of cartesian, cocartesian and strongly cocartesian cubical diagrams.
\begin{defn}\label{def:cocartesian} Let $\chi\colon \power(n)\to \catname{D}$ be a cubical diagram. We say that $\chi$ is {\it cartesian} if it is a pullback cube. I.e., if the following natural map is an equivalence
\[
\chi(\emptyset)\to \underset{\emptyset \ne U\subset \underline{n}}{\lim} \chi(U).
\]
Dually, we say that $\chi$ is {\it cocartesian} if it is a pushout cube. I.e., if the following natural map is an equivalence
\[
\underset{U\subsetneq \underline{n}}{\colim}\chi(U) \to \chi(\underline{n}).
\] 
In yet another formulation: $\chi$ is cocartesian if and only if $\chi$ is equivalent to the left Kan extension of its restriction to $\power(n)_{\le n-1}$. Finally, we say that $\chi$ is {\it strongly cocartesian} if $\chi$ is equivalent to the left Kan extension of its restriction to $\power(n)_{\le 1}$. Specifically it means that for all $U\subset \underline{n}$ the following natural map is an equivalence
\[
\colim \chi|_{\power(U)_{\le 1}} \to \chi(U).
\]
\end{defn}
The following definition of free and split cubical diagrams is perhaps less standard, but these types of cubes come up naturally in the study of polynomial functors. Let $\ob(\power(n))$ be the category that has the same objects as $\power(n)$, and just the identity morphisms. Let us remark that if $\catname{D}$ is a category with finite coproducts (but not necessarily other colimits), then the left Kan extension functor 
\[
\lkan\colon \fun(\ob(\power(n)), \catname{D})\to\fun(\power(n), \catname{D})
\]
exists, because pointwise Kan extension only involves finite coproducts in this case.
\begin{defn}\label{def:split}
Let $\catname{D}$ be a category with finite coproducts. A cubical diagram $\chi\colon\power(n)\to \catname{D}$ is {\it split} if it is a left Kan extension of a functor $\ob(\power(n))\to \catname{D}$. 
%
\end{defn}
\begin{exmp}\label{example:split}
Let us describe explicitly what a split $n$-dimensional cubical diagram looks like. A functor $\ob(\power(n))\to\catname{D}$ is a collection of objects $X(U)$, where $U$ ranges over subsets of $\underline{n}$. The left Kan extension of this functor to $\power(n)$ is defined on objects by the formula
\begin{equation}\label{eq:general split}
U\mapsto \coprod_{V\subset U} X(V),
\end{equation}
and on the level of morphisms it is given by the obvious summand inclusions.
\end{exmp}
\begin{rem} \label{rem:split}
It follows from Definitions~\ref{def:cocartesian} and~\ref{def:split} that a cubical diagram is both split and cocartesian (resp. split and strongly cocartesian) if and only if it is a left Kan extension of a functor $\ob(\power(n)_{\le n-1})\to \catname{D}$ (resp. a functor $\ob(\power(n)_{\le 1})\to \catname{D}$). 
\end{rem}
\begin{exmp}\label{example:splitstrong}
To be really explicit, a split strongly cocartesian $n$-dimensional diagram is determined by $n+1$ objects $X_0, \ldots, X_n$ and is given by the formula
\[
U\mapsto X_0\sqcup\left(\coprod_{i\in U} X_i\right).
\]
with maps being inclusions of wedge summands. More generally, every split strongly cocartesian cube in $\catname{D}$ is equivalent to one of this form.
\end{exmp}
\begin{lem}\label{lemma:additive preserves split}
Let $\catname{C}$, $\catname{D}$ be categories with finite coproducts. Let $F\colon \catname{C}\to \catname{D}$ be a functor that preserves finite coproducts. Then $F$ preserves split cubical diagrams, as well as split cocartesian cubical diagrams and split strongly cocartesian cubical diagrams.
\end{lem}
\begin{proof}
We claim that the following diagram of functor categories commutes up to a natural equivalence
\begin{equation}\label{eq:general square of cubes}
\begin{tikzcd}
	{\fun(\ob(\power(n)), \catname{C})} & {\fun(\power(n), \catname{C})} \\
	{\fun(\ob(\power(n)), \catname{D})} & {\fun(\power(n), \catname{D})}
	\arrow["\lkan", from=1-1, to=1-2]
	\arrow["F", from=1-1, to=2-1]
	\arrow["F", from=1-2, to=2-2]
	\arrow["\lkan"', from=2-1, to=2-2]
\end{tikzcd}
\end{equation}
The reason that the diagram commutes is that, once again, the construction of left Kan extension only involves finite coproducts, and $F$ preserves finite coproducts. The split cubical diagrams in $\catname{C}$ and $\catname{D}$ are precisely the objects in the essential image of the horizontal functors in this diagram. It follows that $F$ preserves split cubical diagrams. The proof for split (strongly) cocartesian cubical diagrams is similar. One just needs to replace $\ob(\power(n))$ with $\ob(\power(n)_{\le n-1})$ (or $\ob(\power(n)_{\le 1})$) in the left column of the diagram.
\end{proof}
\begin{exmps}\label{example: preserves} Let $\catname{D}$ be an $\infty$-category with coproducts. Let $\ho(D)$ denote the homotopy category of $\catname{D}$. The canonical functor $\catname{D} \to \ho(\catname{D})$ from $\catname{D}$ to its homotopy category preserves coproducts, and therefore it preserves split (cocartesian) cubical diagrams. 

If $\catname{D}$ is a stable $\infty$-category with a $t$-structure, then the inclusion functor $\catname{D}^{\heartsuit}\hookrightarrow \catname{D}$ from the heart of $\catname{D}$ to $\catname{D}$ also preserves coproducts, and therefore preserves split (cocartesian) cubical diagrams.

The functor $\const\colon\fr\to \sgr$, that sends a group to a constant simplicial group, preserves coproducts, and therefore preserves split strongly cocartesian cubical diagrams.
\end{exmps}  
\begin{notn}
Let $\splitcube^n(\catname{C}) \subset \fun(\power(n), \catname{C})$ be the full subcategory consisting of split cubical diagrams. 
\end{notn}
Suppose $\catname{I}$ is a small 1-category, and let $\catname{C}$ be an $\infty$-category. There is a natural functor 
\begin{equation}\label{eq:hofunctor}
\ho(\fun(\catname{I}, \catname{C}))\to \fun(\catname{I}, \ho\catname{C}).
\end{equation}
It is well known that in general this functor is not an equivalence of categories. It is neither essentially surjective on objects nor fully faithful. On the other hand, it is conservative. A map between diagrams in $\catname{C}$ an equivalence if and only if it induces an isomorphism between diagrams in $\ho\catname{C}$. 

Special cases when~\eqref{eq:hofunctor} or a related functor is an equivalence are therefore of note. One obvious case when~\eqref{eq:hofunctor} is an equivalence is when $\catname{I}$ is a discrete category. The following easy lemma gives another example when a certain restriction of~\eqref{eq:hofunctor} is an equivalence.
\begin{lem}\label{lemma: homotopy split}
Suppose $\catname{C}$ has finite coproducts. The functor $\ho(\fun(\power(n), \catname{C}))\to \fun(\power(n), \ho\catname{C})$ restricts to an equivalence of categories
\[
\ho(\splitcube^n(\catname{C}))\xrightarrow{\simeq} \splitcube^n(\ho\catname{C})
\]
from the homotopy category of split cubical diagrams in $\catname{C}$ to split cubical diagrams in the homotopy category of $\catname{C}$. A similar result holds for cocartesian split diagrams and strongly cocartesian split diagrams.
\end{lem}
\begin{proof}
The category $\ho\catname{C}$ has finite coproducts, and the functor $\catname{C}\to \ho\catname{C}$ preserves finite coproducts. 
It follows that the functor $\catname{C}\to \ho\catname{C}$ preserves split cubical diagrams (as well as cocartesian and strongly cocartesian split cubical diagrams) by Lemma~\ref{lemma:additive preserves split}. 
It follows that we have a functor $\splitcube^n{\catname{C}}\to \splitcube^n{\ho\catname{C}}$, which passes through a functor $\ho(\splitcube^n(\catname{C}))\to \splitcube^n(\ho\catname{C})$. Furthermore, there is a commuting (up to isomorphism)
diagram of categories, as in ~\eqref{eq:general square of cubes}
\[
\begin{tikzcd}
	\ho(\fun(\ob(\power(n)), \catname{C})) & \ho(\fun(\power(n), \catname{C})) \\
	{\fun(\ob(\power(n)), \ho\catname{C})} & {\fun(\power(n), \ho\catname{C})}
	\arrow["\lkan", from=1-1, to=1-2]
	\arrow[ from=1-1, to=2-1]
	\arrow[ from=1-2, to=2-2]
	\arrow["\lkan"', from=2-1, to=2-2]
\end{tikzcd}
\]
The left vertical functor is an equivalence, because $\ob(\power(n))$ is a discrete category. It follows that the essential image of the top horizontal functor maps surjectively onto the essential image of the bottom horizontal functor. This means that the functor $\ho(\splitcube^n(\catname{C}))\to \splitcube^n(\ho\catname{C})$ is essentially surjective on objects. 

Since $\splitcube^n(\catname{C})$ is the essential image of the left Kan extension functor, we obtain the following diagram
\[
\begin{tikzcd}
	\ho(\fun(\ob(\power(n)), \catname{C})) & \ho(\splitcube^n(\catname{C})) \\
	{\fun(\ob(\power(n)), \ho\catname{C})} & {\splitcube^n(\ho\catname{C})}
	\arrow["\lkan", from=1-1, to=1-2]
	\arrow["\simeq", from=1-1, to=2-1]
	\arrow[ from=1-2, to=2-2]
	\arrow["\lkan"', from=2-1, to=2-2]
\end{tikzcd}
\]
We have shown that the right vertical functor is essentially surjective on objects. It remains to show that it induces a bijection on sets of morphisms. Let us suppose we are given two split cubical diagram $\chi_1, \chi_2\colon \power(n) \to \catname{C}$. Let us denote the images of these cubes in $\fun(\power(n), \ho\catname{C})$ by $\ho\chi_1$ and $\ho\chi_2$. We need to show that the set of morphisms from $\chi_1$ to $\chi_2$ in $\ho(\splitcube^n(\catname{C}))$ maps bijectively to the set of morphisms from $\ho\chi_1$ to $\ho\chi_2$ in $\splitcube^n(\ho\catname{C})$. By assumption $\chi_1$ is split, which means that $\chi_1$ is a left Kan extension of some functor $\chi^0_1\colon \ob(\power(n))\to \catname{C}$. It follows that the set of morphisms from $\chi_1$ to $\chi_2$ in $\ho(\splitcube^n(\catname{C}))$ is in bijective correspondence to the set of morphisms in $\ho(\fun(\ob(\power(n)), \catname{C}))$ from $\chi^0_1$ to the restriction of $\chi_2$ to $\ob(\power(n))$. Since the left vertical functor is an equivalence of categories, this is the same as the set of morphisms in $\fun(\ob(\power(n)), \ho\catname{C})$ from $\ho\chi^0_1$ to the restriction of $\ho\chi_2$ to $\ob(\power(n))$. Again using the fact that the horizontal morphisms are Kan extensions and the diagram commutes, this is the same as the set of morphisms in $\splitcube^n(\ho\catname{C})$ from $\ho\chi_1$ to $\ho\chi_2$.

We have proved that the functor $\ho(\splitcube^n(\catname{C}))\xrightarrow{\simeq} \splitcube^n(\ho\catname{C})$ is an equivalence. The analogous statement for split (strongly) cocartesian diagrams is proved similarly, replacing $\ob(\power(n))$ with $\ob(\power(n)_{\le n-1})$ (or with $\ob(\power(n)_{\le 1})$)
\end{proof}
The following lemma is somewhat similar to the previous one. It is handy for detecting split cubical diagrams in an $\infty$-category.
\begin{lem}\label{lemma:lift}
Let $\catname{C}$ be a category that admits finite coproducts. A cubical diagram in $\catname{C}$ is split (resp. cocartesian split) if and only if its image in $\ho(\catname{C})$ is split (resp. cocartesian split).
\end{lem} 
\begin{proof}
We saw already that the functor $\catname{C}\to \ho(\catname{C})$ preserves (cocartesian) split cubical diagrams, so the ``only if'' direction follows. For the ``if'' direction, consider again the commutative diagram
\[\label{eq:square of cubes}
\begin{tikzcd}
\fun(\ob(\power(n)), \catname{C}) & \fun(\power(n), \catname{C}) \\
	\ho(\fun(\ob(\power(n)), \catname{C})) & \ho(\fun(\power(n), \catname{C})) \\
	{\fun(\ob(\power(n)), \ho(\catname{C}))} & {\fun(\power(n), \ho(\catname{C}))}
	\arrow["\lkan", from=1-1, to=1-2]
	\arrow[from=1-1, to=2-1]
	\arrow[from=2-2, to=3-2]
	\arrow["\simeq", from=2-1, to=3-1]
	\arrow[from=1-2, to=2-2]
	\arrow["\lkan"', from=2-1, to=2-2]
	\arrow["\lkan", from=3-1, to=3-2]
\end{tikzcd}
\]
Let $\chi$ be an object of $\fun(\power(n), \catname{C})$, and let $\ho \chi$ denote the image of $\chi$ in $ \fun(\power(n), \ho(\catname{C}))$. Suppose that $\ho\chi$ is split, i.e. is in the image of the bottom horizontal functor. We want to prove that $\chi$ is split. Since the vertical functors on the left of the diagram are essentially surjective on objects, it follows that there exists a split cubical diagram $\bar\chi \colon\power(n)\to \catname{C}$ such that $\ho\bar\chi\cong \ho\chi$, as cubical diagrams in $\ho(C)$. Since $\bar \chi$ is split, it is a left Kan extension of a diagram $\bar\chi_0\colon \ob(\power(n))\to \catname{C}$. By the commutativity of the diagram above, $\ho\bar\chi$ is a left Kan extension of $\ho \bar\chi_0$. It follows that there is a natural map $\ho \bar\chi_0\to \ho \chi|_{\ob(\power(n))}$, 
that left Kan extends to the isomorphism $\ho\bar \chi \xrightarrow{\cong} \ho \chi$. Since the bottom left vertical map in the diagram above is an equivalence of categories, and the top left vertical map is essentially surjective on objects and on morphisms, a natural map $\ho \bar\chi_0\to \ho \chi|_{\ob(\power(n))}$ in $\fun(\ob(\power(n)), \ho(\catname{C}))$ comes from a natural map $\bar\chi_0\to \chi|_{\ob(\power(n))}$ in $\fun(\ob(\power(n)), \catname{C})$. The left Kan extension of this map gives a map $\bar\chi \to \chi$. Using again the commutativity of the diagram above, the map $\bar\chi\to \chi$ becomes an isomorphism of diagrams in $\ho(\catname{C})$, and therefore $\bar\chi \xrightarrow{\simeq} \chi$ is an equivalence. Since $\bar\chi$ is a split cubical diagram by assumption, we have proved that $\chi$ is a split cubical diagram.

A similar argument applies to split cocartesian diagrams. One just needs to replace $\ob(\power(n))$ with $\ob(\power(n)_{\le n-1})$ everywhere in the argument.
\end{proof}
\section{Polynomial functors}\label{sec:polynomial}
The reason that we spent so much space on split cubical diagrams is that they come up when studying the cross-effects of a functor, which in turn are important in the study of polynomial functors. We will next review the notion of a cross-effect. We will follow the original treatment from~\cite{Eilenberg1954}, which constructed the cross-effects as images of certain idempotents. Let us briefly review from~\cite[Section 4.4.5]{LurieHTT} the meaning of an idempotent in an $\infty$-category.
\begin{defn}
Let $\idem$ be the category consisting of a single object, whose monoid of endomorphisms consist of two elements $\{\id, e\}$, where the composition is determined by the relation $e^2=e$. Given a category $\catname{D}$, an idempotent in $\catname{D}$ is a functor $\idem \to \catname{D}$.
\end{defn}
When $\catname{D}$ is a $1$-category, an idempotent in $\catname{D}$ is an endomorphism $f$ of some object in $\catname{D}$ that satisfies $f\circ f=f$. In $1$-categories being an idempotent is a property of an endomorphism. When $\catname{D}$ is an $\infty$-category, an idempotent in $\catname{D}$ is an endomorphism $f$ together with compatible homotopies between $f$ and its iterates. In $\infty$-categories, being an idempotent is a structure on an endomorphism.




We will avoid subtleties concerning idempotents in $\infty$-categories by carrying out this part of the discussion in the context of $1$-categories. Then in the end we will lift the result that we need to $\infty$-categories using Lemma~\ref{lemma:lift}. But let us also remark that it might not be very difficult to carry out the entire forthcoming discussion on the level of \emph{stable} $\infty$-categories, since idempotents in the homotopy category of a stable $\infty$-category have lifts to $\infty$-idempotents~\cite[Lemma 1.2.4.6]{LurieHA}.

Until further notice, let $\catname{C}$ be a pointed $1$-category with finite coproducts, possibly enriched over topological spaces, and let $\catname{D}$ be an idempotent complete additive $1$-category (also known as a Karoubian category). We denote the coproduct in $\catname{C}$ by $\vee$ and the coproduct in $\catname{D}$ by $\oplus$. We denote the zero object of $\catname{C}$ by $*$ and the zero object of $\catname{D}$ by $0$. Let $X_1, \ldots, X_n$ be objects of $\catname{C}$.
\begin{defn}
For every subset $U\subset \underline{n}$, let 
\[
\psi_U(X_1, \ldots, X_n) \colon X_1 \vee\ldots\vee X_n\to X_1 \vee\ldots\vee X_n
\] 
be the map that is the identity on $X_i$ for each $i\in U$, and collapses $X_i$ to $*$ for each $i\not \in U$. 
\end{defn}
Sometimes we will just write $\psi_U$ instead of $\psi_U(X_1, \ldots, X_n)$, if the $X_i$s are clear from the context, or are irrelevant. 
It is clear that $\{\psi_U\mid U\subset \underline{n}\}$ are pairwise commuting idempotents, which satisfy the relation $\psi_U\circ \psi_V=\psi_{U\cap V}$.

Now let $F\colon\catname{C}\to \catname{D}$ be a functor (if $\catname{C}$ is topologically enriched while $\catname{D}$ is an ``ordinary'' category, then we assume $F$ to be topologically enriched, in the sense of inducing locally constant maps on morphism spaces). For each $U\subset \underline{n}$ we have the idempotent
\[
F(\psi_U)(X_1, \ldots, X_n)\colon F(X_1 \vee\ldots\vee X_n)\to F(X_1 \vee\ldots\vee X_n).
\]
By functoriality of $F$, the maps $F(\psi_U)$ are pairwise commuting idempotents, satisfying $F(\psi_U)\circ F(\psi_V)=F(\psi_{U\cap V})$. Because $\catname{D}$ is an additive category, we may add and subtract maps of the form $F(\psi_U)$. Therefore the following definition makes sense
\begin{defn}
For every subset $U\subset \underline{n}$, define the map 
\begin{equation}\label{eq:D_U}
D_UF(X_1, \ldots, X_n)\colon F(X_1 \vee\ldots\vee X_n)\to F(X_1 \vee\ldots\vee X_n)
\end{equation}
by the formula
\[
D_UF=\sum_{V\subset U} (-1)^{|U|-|V|} F(\psi_V).
\]
\end{defn}
\begin{lem}\label{lem:orthogonal}
The maps $\{D_UF\mid U\subset \underline{n}\}$ are pairwise orthogonal idempotents whose sum is the identity.
\end{lem}
\begin{proof}
The proof is straightforward. The lemma is proved as part of the proof of~\cite[Theorem 9.1]{Eilenberg1954}. It also is~\cite[Lemma 3.21]{Arone-Barthel-Heard-Sanders}, with the following  two minor caveats: In~\cite{Eilenberg1954} the lemma is proved for functors between abelian categories, and in~\cite{Arone-Barthel-Heard-Sanders} it is proved for functors from spectra to spectra. But it is clear that the proof works equally well for functors from a pointed category to an additive category. Second, in~\cite{Eilenberg1954} it is assumed that $F$ is a reduced functor, i.e., that it preserves the zero object. But there is no need to assume this. This means that we have to include $U=\emptyset$ in the decomposition. Note that $D_{\emptyset}F(X_1, \ldots, X_n)$ is the composition
\[
F(X_1 \vee\ldots\vee X_n) \to F(*) \to F(X_1 \vee\ldots\vee X_n).
\]
\end{proof}
\begin{defn}\label{def:cross-effects}
Let $\ce_UF(X_1, \ldots, X_n)$ denote the image of the idempotent map~\eqref{eq:D_U}. When $U=\underline{n}$, we denote it simply by $\ce_nF(X_1, \ldots, X_n)$. We call the functor of $n$ variables $\ce_nF$ the $n$-th cross-effect of $F$.
\end{defn}
\begin{lem}\label{lemma:split}
Let $F\colon \catname{C}\to \catname{D}$ be a functor, and $X_1, \ldots, X_n$ be objects of $\catname{C}$. The idempotents $D_UF$ induce an isomorphism
\[
F(X_1\vee\ldots\vee X_n)\xrightarrow{\cong} \bigoplus_{U\subset\underline{n}} \ce_UF(X_1, \ldots, X_n).
\]
Moreover, the functor $\ce_UF(X_1, \ldots, X_n)$ depends only on $X_i$ for $i\in U$. More precisely, suppose we have morphisms $h_i\colon X_i\to Y_i$ for $i=1, \ldots, n$ and a set $U$ such that $h_i$ is an isomorphism for all $i\in U$. Then the induced map
$\ce_UF(X_1, \ldots, X_n)\to \ce_UF(Y_1, \ldots, Y_n)$ is an isomorphism. In particular, if we have an injective map  $\underline{m}\hookrightarrow\underline{n}$, and $U\subset \underline{m}$, then it induces an isomorphism $\ce_UF(X_1, \ldots, X_m)\xrightarrow{\cong} \ce_UF(X_1, \ldots, X_n)$. 
\end{lem}
\begin{proof}
The first statement is an immediate consequence of Lemma~\ref{lem:orthogonal}, and also is~\cite[Theorem 9.1]{Eilenberg1954}. The second statement is very easy to prove from the definition of the idempotents $D_U$.
\end{proof} 
The following easy and well-known lemma is left as an exercise for the reader
\begin{lem}
Let $F$ be a functor as above. There is a natural isomorphism
\begin{multline*}
\ce_nF(X_1\vee X_2, X_3, \ldots, X_{n+1})\cong \\ \cong\ce_nF(X_1, X_3, \ldots, X_{n+1}) \oplus \ce_nF(X_2, X_3, \ldots, X_{n+1}) \oplus \ce_{n+1}F(X_1, X_2, \ldots, X_{n+1}).
\end{multline*}
\end{lem}
The following is an immediate corollary
\begin{cor}\label{cor:implies}
Let $F$ be a functor as above. Suppose $\ce_nF(X_1, \ldots, X_n)=0$ for all $X_1, \ldots, X_n\in \catname{C}$. Then $\ce_{n+1}F(X_1, \ldots, X_{n+1})=0$ for all $X_1, \ldots, X_{n+1}$. More generally, $\ce_mF$ is the trivial functor for all $m\ge n$.
\end{cor}
Now we can define polynomial functors
\begin{defn}\label{def: polynomial}
Suppose $\catname{C}$ is a pointed category with coproducts, and $\catname{D}$ be either an idempotent complete stable $\infty$-category or an idempotent complete additive $1$-category. We say that a functor $F\colon \catname{C}\to \catname{D}$ is polynomial of degree $n$ if $\ce_{n+1}\ho F$ is the trivial functor of $n+1$ variables. Here by $\ho F$ we mean $F$ if $\catname{D}$ is a $1$-category, and the composition $\catname{C}\xrightarrow{F} \catname{D} \to \ho\catname{D}$ if $\catname{D}$ is a stable $\infty$-category.
\end{defn}
It follows from Corollary~\ref{cor:implies} that if $F$ is polynomial of degree $n$ then $F$ is polynomial of degree $m$ for all $m\ge n$. We say that $F$ is polynomial if it is polynomial of some degree. The next proposition will help establish a connection between polynomial functors and excisive functors in the sense of Goodwillie.
\begin{prop}\label{prop:polynomials}
Let $\catname{C}$ be a pointed $1$-category with finite coproducts. Let $\catname{D}$ be either an idempotent complete stable $\infty$-category or an idempotent complete additive $1$-category. Let $F\colon \catname{C}\to \catname{D}$ be a functor and let $\chi\colon \power(n+1)\to \catname{C}$ be a split strongly cocartesian cubical diagram. Then $F\circ \chi$ is always a split cubical diagram. If $F$ is polynomial of degree $n$, then $F\circ \chi$ is cocartesian. Conversely, if $F\circ \chi$ is cocartesian for all split strongly cocartesian $n+1$-dimensional diagrams $\chi$, then $F$ is polynomial of degree $n$.
\end{prop}
\begin{proof}
Let us first suppose that $\catname{D}$ is an additive $1$-category. We saw in Example~\ref{example:splitstrong} that there exist objects $X_0, X_1, \ldots X_{n+1}$ such that $\chi$ is equivalent to the diagram given by the formula
\[
\chi(U)=X_0\vee \bigvee_{i\in U} X_i.
\]
It follows that there are isomorphisms, by Lemma~\ref{lemma:split}, of the following form: suppose $U=\{i_1, \ldots, i_s\}\subset\underline{n+1}$
\[
F\circ \chi(U)\xrightarrow{\cong} \bigoplus_{V\subset U\cup \{0\}} \ce_V F(X_0, X_{i_1}, \ldots, X_{i_s}).
\]
Furthermore, by the second part of the same lemma, the maps in the cube $F\circ \chi$ are as follows: for $V\subset U\subset U'$ the summands $\ce_V F(X_0, X_{i_1}, \ldots, X_{i_s})$ and $\ce_{V\cup \{0\}} F(X_0, X_{i_1}, \ldots, X_{i_s})$ of $F\circ \chi(U)$ are mapped isomorphically to the summands  $\ce_VF(\ldots)$ and $\ce_{V\cup \{0\}} F(\ldots)$ of $F\circ \chi(U')$. Here we assumed that $U=\{i_1, \ldots, i_s\}$ and we avoided naming the elements of $U'$.

It follows that $F\circ \chi$ is isomorphic to the left Kan extension of the functor $\Theta\colon \ob(\power(n+1))\to \catname{D}$ that is defined by the formula
\[
\Theta(U)= \ce_UF(X_{i_1}, \ldots, X_{i_s})\oplus \ce_{U\cup \{0\}}F(X_0, X_{i_1}, \ldots, X_{i_s}).
\]
Compare with Example~\ref{example:split}. Note that since we assume that $F$ takes values in a $1$-category, it is enough that we have verified that $F\circ \chi$ is isomorphic to the left Kan extension of $\Theta$ on the level of objects and morphisms. We do not need to worry about higher simplices. It follows that $F\circ \chi$ is always a split cube, regardless of whether $F$ is polynomial.

Suppose $F$ is polynomial of degree $n$. Then we have the isomorphism
\[
\Theta(\underline{n+1})=\ce_{n+1}F(X_{1}, \ldots, X_{n+1})\oplus \ce_{n+2}F(X_{0}, \ldots, X_{n+1})\cong 0,
\]
and it follows that $F\circ \chi$ is a left Kan extension of the functor 
\[
\Theta|_{\power(n+1)_{\le n}}\colon \power(n+1)_{\le n}\to \catname{D}.
\] 
Thus $F\circ \chi$ is split cocartesian (see Remark~\ref{rem:split}). Conversely, if $F\circ \chi$ is always cocartesian, then it implies that $\chi_{n+1}F(X_{1}, \ldots, X_{n+1})=0$ for all $X_1, \ldots, X_{n+1}$, which implies that $F$ is polynomial of degree $n$.

Now suppose that $\catname{D}$ is an idempotent complete stable $\infty$-category.  Then $\ho \catname{D}$ is an idempotent complete additive $1$-category. If $\chi$ is a split strongly cocartesian cubical diagram in $\catname{C}$ then the image of $F\circ \chi$ in $\ho\catname{D}$ is a split diagram by the first part of the proof, and therefore $F\circ \chi$ is a split diagram by Lemma~\ref{lemma:lift}. If $F$ is polynomial of degree $n$, then by definition $\ho F$ is polynomial, so the image of $F\circ \chi$ in $\ho\catname{D}$ is split strongly cocartesian, again by the first part of the proof, and therefore $F\circ \chi$ is split strongly cocartesian, again by Lemma~\ref{lemma:lift}. Finally, if $\ho F\circ \chi$ is cocartesian for all split strongly cocartesian $\chi$ then we know that $\ho F \circ \chi$ is split cocartesian for all such $\chi$, so $\ho F$ is polynomial, and thus $F$ is polynomial.
\end{proof}
The next lemma says that if a functor preserves finite coproducts then composition with this functor preserves polynomial functors.
\begin{lem}\label{lemma:additive preserves polynomial}
Suppose $F\colon \catname{C}\to \catname{D}$ is polynomial of degree $n$, and $L\colon \catname{D}\to \catname{D}_1$ is a functor that preserves finite coproducts. Then $L\circ F$ is polynomial of degree $n$.
\end{lem}
\begin{proof}
By Proposition~\ref{prop:polynomials}, $F$ is polynomial of degree $n$ if and only if for every strongly cocartesian split cubical diagram $\chi\colon \power(n+1)\to \catname{C}$, $F\circ \chi$ is cocartesian split. By Lemma~\ref{lemma:additive preserves split}, it follows that $L\circ F\circ \chi$ is cocartesian split for all strongly cocartesian split $\chi$, and therefore $L\circ F$ is polynomial of degree $n$.
\end{proof}
\begin{exmp}\label{example:heart}
Let $\catname{D}$ be a stable $\infty$-category with a $t$-structure. Suppose $\catname{C}$ is, as usual, a pointed category with finite coproducts, and $F\colon \catname{C}\to \catname{D}^{\heartsuit}$ is a polynomial functor of degree $n$. Then the functor $\catname{C} \to \catname{D}$ obtained by composing $F$ with the inclusion functor $\catname{D}^{\heartsuit}\to \catname{D}$ is polynomial of degree $n$, since the inclusion functor preserves coproducts.
\end{exmp}

\section{Excisive functors}\label{section:excisive}
Excisive functors are defined similarly to polynomial functors, but with a stricter hypothesis. They turned out to be very useful in the study of homotopical functors. In this section we review excisive functors, and establish a precise relationship between polynomial functors on the category of (finitely generated free) groups, and excisive functors on the category of simplicial groups.

Let us recall a condition (\cite[Definition 6.1.1.6]{LurieHA}) on a category $\catname{D}$ that helps ensure that Goodwillie calculus works as it should for functors with values in $\catname{D}$
\begin{defn}
An $\infty$-category $\catname{D}$ is {\it differentiable} if it has small colimits and finite limits, and sequential colimits in $\catname{D}$ commute with finite limits. 
\end{defn}
\begin{rem}
Let us remark that if $\catname{D}$ is any one of the following:
\begin{itemize}
\item a stable $\infty$-category with countable coproducts
\item an $\infty$-topos
\item a compactly generated $\infty$-category
\end{itemize}
then $\catname{D}$ is differentiable~\cite[Examples 6.1.1.7-9]{LurieHA}.
\end{rem}
Let us recall what excisive functors are.
\begin{defn}\label{def: excisive}
Suppose $\catname{C}$ is a category with small colimits and $\catname{D}$ differentiable. We say that a functor $F\colon \catname{C}\to \catname{D}$ is $n$-excisive if (a) $F$ preserves filtered colimits, and (b) for every strongly cocartesian diagram $\chi\colon \power(n+1)\to \catname{C}$, the diagram $F\circ \chi\colon \power(n+1)\to \catname{D}$ is cartesian.
We say that a functor is excisive if it is $n$-excisive for some $n$.
\end{defn}
\begin{rems}
The definition of an excisive functor is due to Goodwillie~\cite{Goodwillie2003}, except that Goodwillie did not include the requirement that it preserves filtered colimits. But we are only interested in functors that preserve filtered colimits, so we made it part of the definition. Functors that preserve filtered colimits are called {\it finitary} in~\cite{Goodwillie2003}.

If $\catname{D}$ is a stable $\infty$-category with small colimits then it automatically has finite limits, and a cubical diagram in $\catname{D}$ is cocartesian if and only if it is cartesian~\cite[Proposition 1.2.4.13]{LurieHA}. Thus for functors with values in a stable $\infty$-category we may equivalently state condition (b) as that $F$ takes strongly cocartesian $n+1$-dimensional diagrams to cocartesian diagrams.
\end{rems}
We will mostly be interested in excisive functors with values in a stable $\infty$-category. It turns out that such functors preserve all sifted colimits.
\begin{lem}\label{lemma:preserve sifted}
A functor with values in a stable $\infty$-category that is excisive in the sense of Definition~\ref{def: excisive} preserves sifted colimits.
\end{lem}
\begin{proof}
Let $F\colon \catname{C}\to \catname{D}$ be an excisive functor. By~\cite[Corollary 5.5.8.17]{LurieHTT}, it is enough to prove that $F$ preserves filtered colimits and geometric realisations. $F$ preserves filtered colimits by definition. Therefore, it is enough to prove that $F$ preserves geometric realisations. In fact, it is enough to know that $F$ preserves \emph{finite} geometric realisations, because every simplicial object is equivalent to a filtered colimit of skeletal simplicial objects. 

Preservation of finite geometric realisations in known for functors $F\colon \catname{C}\to \catname{D}$ (where $\catname{D}$ is stable) that take strongly cocartesian $n$-dimensional cubes to cocartesian $n$-dimensional cubes. This was proved by Barwick-Glasman-Mathew-Nikolaus as part of~\cite[Proposition 2.15]{Barwick-Glasman-Mathew-Nikolaus}. Barwick-Glasman-Mathew-Nikolaus only state the claim in the case when $\catname{C}$ and $\catname{D}$ are both stable $\infty$-categories, but the relevant part of the proof does not require $\catname{C}$ to be stable. The main ingredient of the proof is Goodwillie's equivalence between homogeneous functors and symmetric multi-linear functors, and this equivalence holds for functors from an arbitrary $\infty$-category to a stable $\infty$-category. See~\cite{LurieHA} Section 6.1.4, in particular Proposition 6.1.4.14.
\end{proof}
\begin{notns}
Let $\poly_n(\catname{C}, \catname{D})\subset \fun(\catname{C}, \catname{D})$ denote the category of polynomial functors of degree $n$, and $\exc_n(\catname{C}, \catname{D})$ the category of $n$-excisive functors. 
\end{notns}
\begin{rem}\label{rem:excisive-sifted}
It follows from Lemma~\ref{lemma:preserve sifted} that if $\catname{C}$ has sifted colimits and $\catname{D}$ is stable, then $\exc_n(\catname{C}, \catname{D})$ us a (full) subcategory of $\fun_{\Sigma}(\catname{C}, \catname{D})$. Notice furthermore that $\fun_{\Sigma}(\catname{C}, \catname{D})$ is a full subcategory of $\fun_{\operatorname{ind}}(\catname{C}, \catname{D})$ of functors that preserve filtered colimits. By definition, Goodwillie's functor $P_n\colon \fun_{\operatorname{ind}}(\catname{C}, \catname{D})\to \exc_n(\catname{C}, \catname{D})$ is left adjoint to inclusion. Here we have restricted the domain of $P_n$ to consist of filtered-colimit-preserving functors because we have defined $\exc_n(\catname{C}, \catname{D})$ to consists just of filtered-colimit-preserving excisive functors (the assumption that $\catname{D}$ is differentiable guarantees that if $F$ preserves filtered colimits then so does $P_nF$). It follows that the restriction of $P_n$ to $\fun_{\Sigma}(\catname{C}, \catname{D})$, which we will also denote simply by $P_n$, is also a left adjoint to the inclusion $\exc_n(\catname{C}, \catname{D}) \hookrightarrow \fun_{\Sigma}(\catname{C}, \catname{D})$.
\end{rem}
Now we are ready to compare polynomial functors on $\fr$ with excisive functors on $\sgr$ (or, equivalently, $\connected$).
\begin{prop}\label{prop:poly to exc}
Let $\catname{D}$ be a stable $\infty$-category with all small colimits. For each $n\ge 0$, the equivalence of categories (seen in Diagram~\eqref{diagram:functors})
\[
\fun_{\Sigma}(\sgr, \catname{D}) \xrightarrow{\simeq} \fun(\fr, \catname{D})
\]
restricts to an equivalence
\[
\exc_n(\sgr, \catname{D}) \xrightarrow{\simeq} \poly_n(\fr, \catname{D}).
\]
\end{prop}
\begin{proof}
Let $F\in \fun_{\Sigma}(\sgr, \catname{D})$ be a functor from $\sgr$ to $\catname{D}$ that preserves sifted colimits. To prove the proposition it is enough to show that the restriction of $F$ to $\fr$ is polynomial of degree $n$ if and only if $F$ is $n$-excisive.

First let us prove the ``if'' direction, which is easier. Suppose $F\colon \sgr \to \catname{D}$ is $n$-excisive. We want to show that the restriction of $F$ to $\fr$ is polynomial of degree $n$. By Proposition~\ref{prop:polynomials}, this is equivalent to showing that $F$ takes split strongly cocartesian $n+1$-dimensional diagrams in $\fr$ to cocartesian diagrams in $\catname{D}$. By one of the examples in~\ref{example: preserves}, a split strongly cocartesian diagram in $\fr$ is also strongly cocartesian in $\sgr$. Since $F$ is $n$-excisive on $\sgr$, the desired claim follows.

Now suppose that $F\in \fun_{\Sigma}(\sgr, \catname{D}) $ is a sifted colimits preserving functor that restricts to a polynomial functor of degree $n$ on $\fr$. We need to prove that $F$ is $n$-excisive. First of all, we claim that $F$ restricts to a polynomial functor on the category $\allfr$ of all free groups (as opposed to just the finitely generated ones). To prove this claim we have to show that $F$ takes split strongly cocartesian $n+1$-dimensional cubical diagrams in $\allfr$ to cocartesian diagrams in $\catname{D}$. Let $\chi\colon \power(n+1)\to \allfr$ be a split strongly cocartesian diagram. It is easy to show that $\chi$ is equivalent to a filtered colimit of split strongly cocartesian diagrams in $\fr$. Let us note that filtered colimits in the $1$-category $\allfr$ are also filtered colimits in the $\infty$-category $\sgr$. We assume that $F$ preserves sifted colimits, so in particular preserves filtered colimits. It follows that $F\circ \chi$ is equivalent to a filtered colimit of cocartesian diagrams, and thus it is itself cocartesian.

Finally we have to show that if $\chi\colon \power(n+1)\to\sgr$ is a strongly cocartesian diagram, then $F\circ \chi$ is cocartesian. For this we will use that $\sgr$ is the underlying $\infty$-category of a Quillen model structure on the $1$-category of simplicial groups. Let us recall a few facts about the Quillen model structure on $\sgr$ (now considered as a $1$-category). The model structure on $\sgr$ is cofibrantly generated, where generating cofibrations are {\it almost-free maps}. The definition of an almost free map can be found in~\cite[Page 257]{Goerss-Jardine}. For our purposes, the relevant property of almost free maps is the following: Suppose $\theta_\bullet \colon G_\bullet\to H_\bullet$ is an almost free map. Then in each simplicial degree $n$, $\theta_n$ is isomorphic to an inclusion of a summand into a free product with a free group. The small-object argument can be used to replace every morphism $ G_\bullet\to H_\bullet$ in the $1$-category $\sgr$ with an equivalent almost free  map. Furtheremore, one can replace $G_\bullet$ with an equivalent almost free group. Thus any morphism in $\sgr$ is equivalent to a cofibration which in each simplicial degree is an inclusion of the form $F_0\hookrightarrow F_0 * F_1$, where $F_0$ and $F_1$ are free groups.

It follows that a strongly cocartesian diagram in the $\infty$-category $\sgr$ is equivalent to a strongly cocartesian diagram in the $1$-category of simplicial groups, where the initial maps are cofibrations that satisfy the condition above. Therefore we may assume that a strongly cocartesian diagram $\chi\colon \power(n+1)\to\sgr$ is a diagram of simplicial groups that in each simplicial degree restricts to a split strongly cocartesian diagram in the $1$-category $\allfr$. We have seen that $F$ restricts to a polynomial functor on $\allfr$, so evaluating $F$ in each simplicial degree of $\chi$ gives a cocartesian diagram in $\catname{D}$. Next, we assume that $F$ preserves sifted colimits. It follows that $F\circ \chi$ can be calculated as the colimit (over $\Delta^{\op}$) of the restriction of $F$ to each simplicial degree of $\chi$. Thus $F\circ \chi$ is a colimit of cocartesian diagrams and is therefore itself a cocartesian diagram.
\end{proof}
\begin{rem}
Since the inverse to the equivalence of categories
\[
\fun_{\Sigma}(\sgr, \catname{D}) \xrightarrow{\simeq} \fun(\fr, \catname{D})
\]
is given by left Kan extension, it follows that the inverse of the restricted equivalence
\[
\exc_n(\sgr, \catname{D}) \xrightarrow{\simeq} \poly_n(\fr, \catname{D}).
\]
is also given by left Kan extension. In other words, if $F\colon \fr \to \catname{D}$ is a polynomial functor, then its left Kan extension to $\sgr$ is an excisive functor $\lkan F \colon \sgr\to \catname{D}$.
\end{rem}
Since the classifying space functor $\classifying\colon \sgr \to \connected$ is an equivalence, the following is a corollary
\begin{cor}\label{cor:connected_to_fr}
Let $\catname{D}$ be a stable $\infty$-category with small colimits. Consider the classifying space functor $\classifying\colon \fr\to \connected$. For each $n\ge 0$ it induces an equivalence of categories
\[
\exc_n(\connected, \catname{D}) \xrightarrow{\simeq} \poly_n(\fr, \catname{D}).
\]
\end{cor}
\section{From connected spaces to all spaces}\label{section:connected to all}
Our next step is to compare the functor category $\exc_n(\connected, \catname{D})$ with $\exc_n(\spaces,\catname{D})$. The restriction functor \[\fun(\spaces, \catname{D}) \to \fun(\connected, \catname{D})\] is not an equivalence. However, we will see that it restricts to an equivalence from $\exc_n(\spaces, \catname{D})$ to $\exc_n(\connected, \catname{D})$. Moreover, the right Kan extension functor $\rkan\colon\fun(\connected, \catname{D})\to \fun(\spaces, \catname{D})$ restricts to an equivalence between categories of excisive functors, which is an inverse to the restriction functor.

To show this, let us begin by noting that the inclusion functor $\connected \hookrightarrow \spaces$ preserves colimits, and therefore it preserves strongly cocartesian diagrams. Thus restriction along this inclusion really does induce a functor 
\[
\exc_n(\spaces, \catname{D})\to \exc_n(\connected, \catname{D}).
\]
To construct an inverse to the restriction functor we will use Goodwillie's operator $T_n$, which he introduced as an ingredient in the definition of the better known operator $P_n$ --- the universal $n$ excisive approximation. Roughly speaking, $P_n$ is an infinite iterate of $T_n$. 

The definition of $T_n$ involves a strongly cocartesian cubical diagram of a particular type. Let us first recall the construction of this cubical diagram in the classical setting of topological spaces. Suppose $X$ is a topological space and $U$ is a finite set. Recall that $X*U$ denotes the join of $X$ and $U$. The space $X*U$ is homeomorphic to the union of $|U|$ copies of the cone on $X$, glued along $X$. In particular, $X*\emptyset\cong X$, $X*\{1\}\cong CX$ (where $CX$ denotes the cone on $X$) and $X*\{1, 2\}\cong \Sigma X$ (where $\Sigma X$ is the suspension of $X$).

Suppose we let $U$ range over subsets of $\{1, \ldots, n\}$. We obtain a cubical diagram
\[
\chi(U)=X*U.
\]
The diagram $\chi$ has the following properties
\begin{enumerate}
\item $\chi(\emptyset)\cong X$.
\item $\chi(\{i\})\simeq *$
\item $\chi$ is strongly cocartesian.
\end{enumerate}
These properties determine $\chi$ up to equivalence. One can also characterise $\chi$ as follows: $\chi$ is a final object in the $\infty$-category of strongly cocartesian $n$-dimensional cubical diagrams whose initial object is $X$.

It is clear that the construction of a cubical diagram with properties (1)-(3) can be abstracted and extended to other categories. Kuhn showed how the construction can be carried out in a general model category~\cite{Kuhn_2007}, and Lurie generalised it to $\infty$-categories. More precisely, Lurie showed the following
\begin{lem}[\cite{LurieHA}, Construction 6.1.1.18] \label{lemma: abstract join}
Let $\catname{C}$ be a category with finite colimits and a final object $*$. Let $\injfin$ be the category of finite sets and injective functions between them. There is an essentially unique functor $\catname{C}\times \injfin\to \catname{C}$, which we will denote by $(X, U)\mapsto C_U(X)$ with the following properties
\begin{enumerate}
\item $C_{\emptyset}(X)\simeq X$. More precisely $C_{\emptyset}(-)$ is equivalent to the identity functor on $\catname{C}$.
\item $C_{\{i\}}(X)\simeq *$
\item \label{eq:like join} For all finite sets $U$ the following natural map is an equivalence
\[
\underset{S\in \power(U)_{\le 1}}{\colim} C_S(X) \xrightarrow{\simeq} C_U(X).
\]
\end{enumerate}
\end{lem}
\begin{rems}
The construction $C_U(X)$ is an abstraction of $X*U$. 

In particular, it is clear from the lemma that $C_{\{1,2\}}(X)\simeq *\coprod_X *$ is a suspension of $X$.

It is also clear from the lemma that if we let $U$ range over subsets of $\underline{n}$ then the assignment $U\mapsto C_U(X)$ gives a strongly cocartesian cubical diagram in $\catname{C}$.
\end{rems}
We will need an easy lemma to the effect that $C_U$ preserves certain colimits. 
\begin{lem}\label{lemma:colimits}
Suppose $\catname{C}$ is a category with finite colimits and a final object $*$. Let $\catname{I}$ be a finite category, or more generally a small category such that $\catname{C}$ has $\catname{I}$-shaped colimits. Assume that the following natural map is an equivalence
\[
\underset{\catname{I}}{\colim} * \to*.
\]
Then for any fixed finite set $U$ the functor $C_U(-)\colon \catname{C}\to \catname{C}$ preserves $\catname{I}$-shaped colimits.
\end{lem}
\begin{proof}
Suppose we have a functor $X\colon \catname{I} \to \catname{C}$, which we will denote by $\alpha\mapsto X_\alpha$. We need to show that the natural map
\[
\underset{\alpha\in \catname{I}}{\colim} C_U(X_\alpha)  \to C_U(\underset{\alpha\in \catname{I}}{\colim} X_\alpha)
\]
is an equivalence. By definition of $C_U(X)$, we have a natural equivalence
\[
\underset{S\in \power(U)_{\le 1}}\colim C_S(X) \xrightarrow{\simeq} C_U(X).
\]
Therefore, we need to show that the following map is an equivalence
\[
\underset{\alpha\in \catname{I}}{\colim} \underset{S\in \power(U)_{\le 1}}{\colim} C_S(X_\alpha)  \to \underset{S\in \power(U)_{\le 1}}{\colim} C_S(\underset{\alpha\in \catname{I}}{\colim}X_\alpha).
\]
Exchanging the order of colimits, this is equivalent to showing that the following is an equivalence
\[
\underset{S\in \power(U)_{\le 1}}{\colim}\underset{\alpha\in \catname{I}}{\colim}  C_S(X_\alpha)  \to \underset{S\in \power(U)_{\le 1}}{\colim} C_S(\underset{\alpha\in \catname{I}}{\colim}X_\alpha).
\]
For this it is enough to show that the following is an equivalence whenever $S$ has cardinality at most $1$:
\[
\underset{\alpha\in \catname{I}}{\colim}  C_S(X_\alpha)  \to C_S(\underset{\alpha\in \catname{I}}{\colim}X_\alpha).
\]
When $S=\emptyset$, $C_S(X)\simeq X$, and the above map is equivalent to the identity map on $\underset{\alpha\in \catname{I}}{\colim}  X_\alpha $. When $S$ is a singleton, $C_S(X)\simeq *$, and the above map becomes $\colim_{\catname{I}} * \to *$. This map is an equivalence by hypothesis.
\end{proof}
\begin{exmps}\label{ex:C_U preserves cocartesian}
If $\catname{C}$ is a pointed category then $*$ is both an initial and a final object, and then the map $\colim_{\catname{I}} * \to *$ is an equivalence for all diagrams $\catname{I}$. So if $\catname{C}$ is pointed then $C_U(-)$ preserves all colimits. In general, the map $\colim_{\catname{I}} * \to *$ is an equivalence whenever the geometric realisation of $\catname{I}$ is contractible. It follows that $C_U(-)$ always preserves cocartesian and strongly cocartesian cubical diagrams.
\end{exmps}

Now we can recall the definition of Goodwillie's operator $T_n$, adapted to the setting of $\infty$ categories. The following is taken from~\cite[Construction 6.1.1.22]{LurieHA}
\begin{defn}
Let $\catname{C}$ be a category with finite colimits and a final object, and let $\catname{D}$ be a category with finite limits. We define the functor $T_n\colon \fun(\catname{C}, \catname{D})\to \fun(\catname{C}, \catname{D})$ as follows. Let $F\colon \catname{C}\to \catname{D}$ be a functor. We define a new functor $T_nF\colon \catname{C}\to \catname{D}$ by the formula
\[
T_nF(X)=\lim_{\emptyset \ne U\subset \underline{n+1}} F(C_U(X)).
\]
When we need to make the category $\catname{C}$ visible, we write $T_n^{\catname{C}}$ instead of $T_n$.
\end{defn}
The following is immediate
\begin{lem}
There is a natural transformation $F\to T_nF$, which is an equivalence if $F$ is $n$-excisive.
\end{lem}
\begin{proof}
The diagram $U \mapsto C_U(X)$ is strongly cocartesian. If $F$ is $n$-excisive it takes strongly cocartesian diagrams to cartesian diagrams.
\end{proof}
\begin{rem}
Under mild assumptions, the ``if'' in the last lemma is in fact an ``if and only if''. The only if direction follows because if $F\to T_nF$ is an equivalence, then $F\to P_nF$ is an equivalence, and $P_nF$ is an $n$-excisive functor for all $F$ by a theorem of Goodwillie~\cite[Lemma 6.1.1.33]{LurieHA}.
\end{rem}
Next let us notice that if $F$ is a functor from $\connected$ to $\catname{D}$, then $T_nF$ can be considered a functor from $\spaces$ to $\catname{D}$, because $T_nF(X)$ depends only on $X*U$ where $U$ is non-empty, and this space is always connected. Let us formalise this observation in a more general setting. 
\begin{hyp}\label{hypothesis:suspension}
Let $\catname{C}$ be a category with small colimits and a final object $*$. Let $\catname{C}_1\subset \catname{C}$ be a full subcategory with the following properties: 
\begin{enumerate}
\item if $X\in\ob(\catname{C}_1)$ and $Y\simeq X$, then $Y\in \ob(\catname{C}_1)$ (this assumption is made for convenience and is probably not essential),
\item $\catname{C}_1$ is closed under finite colimits in $\catname{C}$ and contains the final objects of $\catname{C}$, 
\item For every object $X$ of $\catname{C}$, $\Sigma X\in \catname{C}_1$. Here $\Sigma X=\colim(* \leftarrow X \rightarrow *)$ is the suspension of $X$.
\end{enumerate}
\end{hyp}
\begin{exmp}
The subcategory $\connected\subset \spaces$ satisfies Hypothesis~\ref{hypothesis:suspension}.
\end{exmp}
We have the following easy but important lemma
\begin{lem}\label{lemma:join connected}
Suppose $\catname{C}_1\subset \catname{C}$ satisfy Hypothesis~\ref{hypothesis:suspension}, and let $X$ be an object of $\catname{C}$. Then $C_U(X)$ is an object of $\catname{C}_1$ for every \emph{non-empty} finite set $U$.
\end{lem}
\begin{proof}
When $U$ is a singleton, $C_U(X)$ is a final object, which is in $\catname{C}_1$ by hypothesis. When $U$ has two points, $C_U(X)$ is a suspension of $X$, which is in $\catname{C}_1$, also by hypothesis. For larger sets $U$, notice that there is an equivalence
\[
C_{\underline{n+1}}(X) \simeq \colim\left(C_{\underline{n}}(X) \leftarrow C_{\underline{1}}(X) \rightarrow C_{\underline{2}}(X)\right).
\]
The equivalence is easily verified using property~\eqref{eq:like join} in the statement of Lemma~\ref{lemma: abstract join}. Since $\catname{C}_1$ is closed under finite colimits, it follows by induction on the number of elements in $U$ that $C_U(X)$ is an object $\catname{C}_1$ for all non-empty finite $U$.
\end{proof}
Suppose that $\catname{C}_1\subset \catname{C}$ are categories that satisfy Hypothesis~\ref{hypothesis:suspension}. We saw in Lemma~\ref{lemma: abstract join} that there is a functor $\catname{C}\times \injfin\to \catname{C}$, which sends $(X, U)$ to $C_U(X)$, that satisfies certain properties. Let $\injfin_{\ge 1}$ be the full subcategory of $\injfin$ consisting of non-empty sets. By Lemma~\ref{lemma:join connected}, for every $X\in \catname{C}_1$ and $U\in \injfin_{\ge 1}$, $C_U(X)$ is an object of $\catname{C}_1$. Since $\catname{C}_1$ is a full subcategory of $\catname{C}$ it follows that the functor $(X, U)\mapsto C_U(X)$ factors essentially uniquely as a functor 
\begin{equation}\label{eq: enhanced C_U}
\catname{C}\times \injfin_{\ge 1}\to \catname{C}_1.
\end{equation}
We will continue denoting this functor by $(X, U)\mapsto C_U(X)$.

Suppose $\catname{D}$ is a category with finite limits, and we have a functor $F\colon \catname{C}_1\to \catname{D}$. We have defined the functor $T_nF\colon \catname{C}_1\to \catname{D}$ by the formula $T_nF(X)=\underset{\emptyset\ne U\subset \underline{n+1}}{\lim}F(C_U(X))$. Now we know that for any fixed $X\in \catname{C}$, the functor $U\mapsto C_U(X)$ takes values in $\catname{C}_1$. Thus $T_nF(X)$ is well-defined for all $X\in \catname{C}$, and indeed we may consider $T_nF$ as a functor from $\catname{C}$ to $\catname{D}$. To emphasise that we have extended the domain, we denote this extension of $T_nF$ by $\widetilde T_nF$.
\begin{defn}
Let $\catname{C}_1, \catname{C}$, $\catname{D}$, and $F\colon\catname{C}_1\to D$ be as above. Define the functor $\widetilde T_n\colon \catname{C}\to \catname{D}$ by the formula
\[
\widetilde T_n F(X)= \underset{\emptyset\ne U\subset \underline{n+1}}{\lim}F(C_U(X)).
\]
\end{defn}
With a bit of straightforward diagram chasing, one can prove the following
\begin{lem}\label{lemma:inverse}
The above formula for $\widetilde T_nF$ gives rise to a well-defined functor
\[
\widetilde T_n\colon \fun(\catname{C}_1, \catname{D})\to \fun(\catname{C}, \catname{D}).
\]
Furthermore, let $\rho$ be the restriction functor
\[
\rho\colon \fun(\catname{C}, \catname{D})\to \fun(\catname{C}_1, \catname{D}).
\]
Then $\rho\circ \widetilde T_n\simeq T_n^{\catname{C}_1}$ and $\widetilde T_n\circ \rho\simeq T_n^{\catname{C}}$.
\end{lem}
We have the following lemma
\begin{lem}\label{lemma:excisive}
Suppose $F\colon \catname{C}_1\to\catname{D}$ is $n$-excisive. Then $\widetilde T_n F\colon \catname{C}\to \catname{D}$ is $n$-excisive.
\end{lem}
\begin{proof}
Suppose that we have a strongly cocartesian diagram $\chi\colon \power(n+1)\to \catname{C}$. To reduce the amount of parentheses, we will write $\chi_V$ instead of $\chi(V)$. We need to show that the following map is an equivalence
\[
\widetilde T_n F(\chi_\emptyset)\to \underset{\emptyset \ne V \subset \underline{n+1}}{\lim} \widetilde T_n F(\chi_V).
\]
Substituting the definition of $\widetilde T_n F$, this is equivalent to showing that the following map is an equivalence
\[
\underset{\emptyset \ne U\subset\underline{n+1}}{\lim}F(C_U(\chi_{\emptyset}))\to  \underset{\emptyset \ne V \subset \underline{n+1}}{\lim}\left(\underset{\emptyset \ne U\subset\underline{n+1}}{\lim} F(C_U(\chi_{V}))\right).
\]
Here the diagram $U\mapsto C_U(\chi_V)$ lives in $\catname{C}_1$, as $U$ is always non-empty. Next, we may exchange the order of limits, and rewrite the last map in the following equivalent form
\begin{equation}\label{eq:limit}
\underset{\emptyset \ne U\subset\underline{n+1}}{\lim}F(C_U(\chi_{\emptyset}))\to  \underset{\emptyset \ne U\subset \underline{n+1}}{\lim}\left(\underset{\emptyset \ne V\subset\underline{n+1}}{\lim} F(C_U(\chi_{V}))\right).
\end{equation}
In this form we can think of this map as the limit of of a diagram of maps of the following form, for each fixed non-empty $U\subset{\underline{n+1}}$
\begin{equation}\label{eq:individual}
F(C_U(\chi_{\emptyset}))\to  \underset{\emptyset \ne V\subset\underline{n+1}}{\lim} F(C_U(\chi_{V})).
\end{equation}
To prove that the map~\eqref{eq:limit} is an equivalence, it is enough to prove that the map~\eqref{eq:individual} is an equivalence for each $U$. For each non-empty $U$, the diagram $V\mapsto C_U(\chi_V)$ is a strongly cocartesian diagram in $\catname{C}_1$ by Lemma~\ref{lemma:colimits} and Examples~\ref{ex:C_U preserves cocartesian}, as well as Lemma~\ref{lemma:join connected} and the fact that $\catname{C}_1$ is closed under colimits in $\catname{C}$. Since $F$ is assumed to be an $n$-excisive functor on $\catname{C}_1$, it follows that~\eqref{eq:individual} is an equivalence.
\end{proof}
Lemmas~\ref{lemma:inverse} and~\ref{lemma:excisive} together imply the following proposition
\begin{prop}\label{prop:inverse}
Suppose $\catname{C}_1, \catname{C}$ satisfy Hypothesis~\ref{hypothesis:suspension}, and $\catname{D}$ is an $\infty$-category with finite limits. Then the construction $\widetilde T_n$ and the restriction functor $\rho$ define inverse equivalences of categories
\[
\widetilde T_n : \begin{tikzcd}
           \exc_n(\catname{C}_1, \catname{D})
           \arrow[r, shift left=.75ex]
            \arrow[r, phantom, "\simeq"] 
           &  \exc_n(\catname{C}, \catname{D})
           \arrow[l, shift left=.75ex]      
        \end{tikzcd}  : \rho
\]
\end{prop}
Notice that since for every $m\ge n$ an $n$-excisive functor is also $m$-excisive, one can replace $\widetilde T_n$ in the proposition with $\widetilde T_m$ for all $n \le m\le \infty$. 
\begin{cor}\label{cor:connected}
Let $\catname{D}$ be an $\infty$-category with finite limits. Then $\widetilde T_n$ and the restriction induces inverse equivalences of categories
\[
\exc_n(\spaces, \catname{D})\simeq \exc_n(\connected, \catname{D}).
\]
Furthermore, $\widetilde T_n$ can be replaced with $\widetilde T_m$ for all $n\le m \le \infty$.
\end{cor}
There is another interpretation of the functor $\widetilde T_n$ that seems worth mentioning, even though we will not make direct use of it. Namely, $\widetilde T_n$ is right Kan extension, or more precisely, it agrees with right Kan extension on excisive functors.
Suppose $\catname{C}_1$ is a subcategory of $\catname{C}$. Let $\rkan\colon \fun(\catname{C}_1, \catname{D}) \to \fun(\catname{C}, \catname{D})$ denote the right Kan extension.
\begin{prop}\label{prop:right Kan}
Suppose $\catname{C}_1\subset \catname{C}$ satisfy Hypothesis~\ref{hypothesis:suspension}, and $\catname{D}$ is a differentiable category. The following diagram commutes up to a natural equivalence
\[\begin{tikzcd}
	{\exc_n(\catname{C}_1, \catname{D})} & {\fun(\catname{C}_1, \catname{D})} \\
	{\exc_n(\catname{C}, \catname{D})} & {\fun(\catname{C}, \catname{D})}
	\arrow[hook, from=1-1, to=1-2]
	\arrow["{\widetilde T_n}", from=1-1, to=2-1]
	\arrow["R", from=1-2, to=2-2]
	\arrow[hook, from=2-1, to=2-2]
\end{tikzcd}\]
\end{prop}
Recall that $\nat_{\catname{C}}(F, G)$ denotes the space of maps from $F$ to $G$ in the relevant category of functors. In what follows, the source category of $F$ and $G$ may be either $\catname{C}$ or $\catname{C}_1$, but the target category $\catname{D}$ will be fixed, so we do not make it part of the notation. Also, given a functor $F\colon \catname{C}\to \catname{D}$ we denote the restriction of $F$ to $\catname{C}_1$ by $\rho F$.
\begin{proof}
Let $F\in \fun(\catname{C}, \catname{D})$ and $G\in \exc_n(\catname{C}_1, \catname{D})$. To prove the proposition it is enough to show that there is a natural equivalence
\[
\nat_{\catname{C}_1}(\rho F, G) \simeq \nat_{\catname{C}}(F, \widetilde T_n G).
\]
We have a chain of maps
\begin{multline*}
\nat_{\catname{C}_1}(\rho F, G)\xrightarrow{\alpha} \nat_{\catname{C}}(\widetilde T_n\rho F, \widetilde T_nG)\xrightarrow{\beta}  \nat_{\catname{C}}(F, \widetilde T_nG)\xrightarrow{\gamma}  \\ \to \nat_{\catname{C}_1}(\rho F, \rho\widetilde T_n G) \xrightarrow{\delta} \nat_{\catname{C}}(\widetilde T_n \rho F, \widetilde T_n \rho \widetilde T_n G).
\end{multline*}
Let us point out that the map $\beta$ is restriction along the map $F\to T_nF$, followed by the natural equivalence $T_n F\xrightarrow{\simeq}\widetilde T_n\rho F$ (Lemma~\ref{lemma:inverse}). The definition of the maps $\alpha$, $\gamma$ and $\delta$ should be self evident.

Our goal is to prove that the composition $\beta\alpha$ is an equivalence. We will prove that $\gamma\beta\alpha$ and $\delta\gamma$ are equivalences. This implies that each one of the maps $\beta\alpha$, $\gamma$, and $\delta$ is an equivalence.

Let us start with the map
\[
\gamma\beta\alpha\colon \nat_{\catname{C}_1}(\rho F, G)\to \nat_{\catname{C}_1}(\rho F, \rho\widetilde T_n G).
\]
It is straightforward to check that the map is induced by the natural transformation $G\to \rho\widetilde T_n G$, which is the same as the map $G\to T_n G$. Since $G$ is $n$-excisive, this map is an equivalence, and therefore $\gamma\beta\alpha$ is an equivalence.

To prove that $\delta\gamma$ is an equivalence, we extend this chain of maps as follows
\begin{multline*}
\nat_{\catname{C}}(F, \widetilde T_nG)\xrightarrow{\gamma}  \nat_{\catname{C}_1}(\rho F, \rho\widetilde T_n G) \xrightarrow{\delta} \\ \to \nat_{\catname{C}}(\widetilde T_n \rho F, \widetilde T_n \rho \widetilde T_n G) \xrightarrow{\zeta} 
 \nat_{\catname{C}}( F, \widetilde T_n \rho \widetilde T_n G) 
\end{multline*}
Here the map $\zeta$ is induced by the map $F\to T_nF \simeq \widetilde T_n \rho F$. Once again, it is straightforward to check that the composition
\[
\zeta\delta\gamma\colon \nat_{\catname{C}}(F, \widetilde T_nG)\to  \nat_{\catname{C}}( F, \widetilde T_n \rho \widetilde T_n G) 
\]
is induced by the map $\widetilde T_n G \to T_n\widetilde T_n G\simeq  \widetilde T_n \rho \widetilde T_n G$. This map is an equivalence, because $\widetilde T_n G$ is an $n$-excisive functor, by Lemma~\ref{lemma:excisive}. We have shown that $\zeta\delta\gamma$ is an equivalence, and it remains to show that $\zeta$ is an equivalence, which would imply that $\delta\gamma$ is an equivalence. By definition, $\zeta$ is induced by the map $F\to T_n F \simeq \widetilde T_n\rho F$. Notice again that the functor $\widetilde T_n \rho\widetilde T_n G\simeq \widetilde T_n G$ is $n$-excisive. It remains to prove the following claim: Let $Q\colon \catname{C}\to \catname{D}$ be $n$-excisive. Then the induced map $\nat(T_n F, Q) \to \nat(F, Q)$ is an equivalence. This follows by considering the maps $F\to T_n F\to P_nF$, where $P_n$ is an infinite iteration of $T_n$. The functor $P_n$ can be thought of as a functor $P_n\colon\fun(\catname{C}, \catname{D})\to \exc_n(\catname{C}, \catname{D})$ that is left adjoint to the inclusion~\cite[Theorem 6.1.1.10]{LurieHA}. This means that the map $\nat_{\catname{C}}(P_n F, Q) \xrightarrow {\simeq}\nat_{\catname{C}}(F, Q)$ is an equivalence. The map $\nat_{\catname{C}}(P_n F, Q)\to \nat_{\catname{C}}(T_n F, Q)$ is an equivalence as well, because the map $T_n F\to P_nF$ is equivalent to the map $T_nF\to P_nT_nF$. It follows that the map $\nat_{\catname{C}}(T_n F, Q) \to \nat_{\catname{C}}(F, Q)$ is indeed an equivalence.
\end{proof}
\begin{cor}\label{cor:restriction}
Let $\catname{D}$ be a differentiable category. Restriction induces an equivalence of categories.
\[
\exc_n(\spaces, \catname{D})\xrightarrow{\simeq} \exc_n(\connected, \catname{D}).
\]
The inverse equivalence is given by right Kan extension.
\end{cor}
\section{The main theorems}\label{section:main theorems}
Putting all the previous results together, we get the following theorem
\begin{thm}\label{theorem:main general}
Let $\catname{D}$ be a stable $\infty$-category with small colimits. For each $n\ge 0$ there is a diagram of categories and functors, commuting up to a natural equivalence, where all functors are equivalences of categories, and parallel arrows are inverses of each other. 
\[
\begin{tikzcd}
	{\poly_n(\circles, \catname{D})} & {\exc_n(\connected, \catname{D})} & {\exc_n(\spaces, \catname{D})} \\
	{\poly_n(\fr, \catname{D})} & {\exc_n(\sgr, \catname{D})}
	\arrow["\lkan"', shift right, from=1-1, to=1-2]
	\arrow["{\rho_{\classifying}}"', shift right, from=1-1, to=2-1]
	\arrow["\rho"', shift right, from=1-2, to=1-1]
	\arrow["\rkan", shift left, from=1-2, to=1-3]
	\arrow["{\rho_{\classifying}}"', shift left, from=1-2, to=2-2]
	\arrow["\rho", shift left, from=1-3, to=1-2]
	\arrow["{\rho_{\pi_1}}"', shift right, from=2-1, to=1-1]
	\arrow["\lkan"', shift right, from=2-1, to=2-2]
	\arrow["\rho"', shift right, from=2-2, to=2-1]
	\arrow[draw=none, from=2-2, to=2-1]
\end{tikzcd}
\]
Here $\rho_{\classifying}$ denotes restriction along the functor $\classifying\colon \sgr\to \connected$, or its restriction $\classifying\colon \fr\to \circles$, which we denote with the same letter $\classifying$. $\rho_{\pi_1}$ denotes restriction along the functor $\pi_1\colon \circles \xrightarrow{\simeq} \fr$. Other arrows marked $\rho$ denote restriction along the obvious inclusion functor. The arrows marked $\lkan$ denote left Kan extension, or more precisely the restriction of left Kan extension to the category of polynomial functors. Similarly $\rkan$ denotes right Kan extension.
\end{thm}
\begin{proof}
That the pair of horizontal arrows at the bottom of the diagram are inverse equivalences follows Proposition~\ref{prop:poly to exc} and the fact that left Kan extension is the inverse functor of restriction in this case. The arrows marked $\rho_{\classifying}$ are equivalences because the functors $\classifying \colon \fr \xrightarrow{\simeq} \circles$ and $\classifying\colon \sgr\xrightarrow{\simeq} \connected$ are both equivalences of categories (the former was discussed in Section~\ref{section: fr to connected}, the latter is Lemma~\ref{lemma:equivalences}). The pair of maps at the top right of the diagram are inverse equivalences by Corollary~\ref{cor:restriction}.

It is clear that the diagram consisting of restriction functors commutes. Since the two vertical restriction functors marked $\rho_{\classifying}$ are equivalences, and the functor marked $\rho$ at the bottom of the diagram is an equivalence, it follows that the functor marked $\rho$ at the top left of the diagram is an equivalence. 
\end{proof}
\subsection*{Reinterpretation in terms of operad (co)modules}
Recall that $\finset^{\le n}$ denotes the category of pointed sets with at most $n$ non-basepoint elements. Let $\catname{D}$ be a stable $\infty$-category. Let $\rho_n\colon\fun(\spaces, \catname{D})\to \fun(\finset^{\le n}, \catname{D})$ be the restriction functor. 
and $\lkan_n\colon\fun(\finset^{\le n}, \catname{D}) \to \fun(\spaces, \catname{D})$ the left Kan extension. 
The following lemma is well-known. 
\begin{lem}\label{lem: excisive finset}
Let $\catname{D}$ be a stable $\infty$-category with small colimits. The restriction functor   
\[
\rho_n\colon\fun(\spaces, \catname{D})\to \fun(\finset^{\le n}, \catname{D})
\]
restricts to an equivalence, which we still denote by $\rho_n$
\[
\rho_n\colon\exc_n(\spaces, \catname{D})\xrightarrow{\simeq} \fun(\finset^{\le n}, \catname{D}).
\]
The inverse equivalence is given by left Kan extension $\lkan_n$.
\end{lem}
A proof is given in~\cite{Arone-Ching_Cross-effects} Section 3 for functors from $\spaces$ to $\catname{Spectra}$ in the setting of model categories. A more general $\infty$-categorical statement appears as~\cite[Corollary 6.1.5.7]{LurieHA} (see also Corollary~6.1.5.2 for the special case of functors from $\spaces$ to $\catname{Spectra}$) with the additional hypothesis that $\catname{D}$ is presentable. Lemma~\ref{lem: excisive finset} is less general than~\cite[Corollary 6.1.5.7]{LurieHA} in that we only consider domain category $\spaces$, but it is more general in that we do not assume that $\catname{D}$ is presentable, but only cocomplete stable. This is desirable, because it implies, for example, that Lemma~\ref{lem: excisive finset} applies to contravariant functors from $\spaces$ to $\catname{Spectra}$, even though the category ${\catname{Spectra}}^{\op}$ is not presentable.
\begin{proof}[Sketch of proof of Lemma~\ref{lem: excisive finset}]
To begin with, for an object $d\in \catname D$ the functor $X\mapsto d\otimes X^m$ is $m$-excisive. This is essentially~\cite[Example 3.5]{Goodwillie1991}. Any functor in the image of the left Kan extension  
\[
\lkan_n\colon \fun(\finset^{\le n}, \catname{D})\to \fun(\spaces, \catname{D})
\] 
can be written as a colimit of functors of the form $d\otimes X^m$ with $m\le n$. This follows from the coend formula for left Kan extension, about which one can read for example in~\cite{haugseng2021coendsinftycategories}. It follows that the image of $\lkan_n$ lies in $\exc_n(\spaces, \catname{D})$. It follows that $\lkan_n$ factors as a composition
\[
\fun(\finset^{\le n}, \catname{D})\xrightarrow{\lkan_n^0}\exc_n(\spaces, \catname{D})\hookrightarrow \fun(\spaces, \catname{D}).
\]
Since $\exc_n(\spaces, \catname{D})$ is a full subcategory of $\fun(\spaces, \catname{D})$, we have an adjunction $\lkan_n^0\dashv \rho_n$ between $\fun(\finset^{\le n}, \catname{D})$ and $\exc_n(\spaces, \catname{D})$. We claim that for every $F\in \fun(\finset^{\le n}, \catname{D)}$ the map $F\to\rho_n\lkan F$ is an equivalence, and for every $G\in \exc_n(\spaces, \catname{D})$ the natural map $\lkan_n\rho_n G\to G$ is an equivalence. The first claim holds because $\finset^{\le n}$ is a full subcategory of $\spaces$. 
We need to prove the second claim. Thus we need to show that for every CW complex $X$, and every $n$-excisive functor $G\colon \spaces\to \catname{D}$ the natural map $\epsilon_X\colon \lkan_n\rho_nG(X)\to G(X)$ is an equivalence. Note that $\lkan_n\rho_nG(X)$ is an $n$-excisive functor. We argue that $\epsilon_X$ is an equivalence for all CW complexes $X$, using the following steps.
\begin{enumerate}
\item the map $\epsilon_X$ is an equivalence when $X\in \finset^{\le n}$, because $\rho_n\lkan_n$ is equivalent to the identity. 
\item It follows that $\epsilon_X$ is an equivalence for all $X\in \finset$. To see this, suppose $m>n$ and assume by induction that $\epsilon_X$ is an equivalence whenever $X$ is a pointed set with fewer than $m$ nonbasepoint elements. Since $m>n$, one can construct a strongly homotopy cocartesian diagram $\chi\colon\power(n+1)\to \spaces$ such that $\chi(\underline{n+1})$ is a set with $m$ elements and for every proper subset $U\subsetneq\underline{n+1}$, $\chi(U)$ has fewer than $m$ elements. It follows by induction hypothesis that whenever $U\subsetneq\underline{n+1}$ $\epsilon_{\chi(U)}$ is an equivalence, and we want to show that $\epsilon_{\chi(\underline{n+1)}}$ is an equivalence. Since $G$ and $\lkan_n\rho_n G$ are both $n$-excisive, $\epsilon_\chi\colon \lkan_n\rho_n G \chi\to G\chi$ is a map of cocartesian cubical diagrams (this is where we use that $\catname{D}$ is stable). This map is an equivalence when evaluated at a proper subset $U\subsetneq \underline{n+1}$, and thus it is an equivalence at $\underline{n+1}$. This means that  $\epsilon_{\underline{m}}\colon\lkan_n\rho_n G(m_+)\to G(m_+)$ is an equivalence.
\item We have shown that $\epsilon_X$ is an equivalence when $X$ is a zero-dimensional complex. Suppose by induction that $1\le d$ and $\epsilon_X$ is an equivalence when $X$ is a finite complex of dimension less than $d$. Now suppose $X$ is a a finite CW complex of dimension $d$, with $1\le k$ cells  of dimension $d$, and $\epsilon_X$ is an equivalence whenever $X$ is a finite $d$-dimensional complex with fewer than $k$ cells in dimension $d$. By punching $n+1$ holes in a top-dimensional cell of $X$ we can present $X$ as the pushout of a strongly cocartesian $n+1$-dimensional diagram where all the other terms have fewer top-dimensional cells than $X$. Using a similar argument to the one in the previous step, we conclude that $\epsilon_X$ is an equivalence. 
\item We have shown that $\epsilon_X$ is an equivalence for all finite CW complexes $X$. We assumed that $G$ commutes with filtered colimits, as part of our definition of $n$-excisive functors. The functor $\lkan_n\rho_n G$ also commutes with filtered colimits, because objects of $\finset^{\le n}$ are compact objects in $\spaces$. It follows that $\epsilon_X\colon \lkan_n\rho_n G(X)\to G(X)$ is an equivalence for all CW complexes $X$. 
\end{enumerate}
We have shown that $\lkan_n$ and $\rho_n$ induce inverse equivalences between $\fun(\finset^{\le n}, \catname{D})$ and $\exc_n(\spaces, \catname{D})$.
\end{proof}
Furthermore, let $\epi^{\le n}$ be the category of (unpointed) finite sets of cardinality at most $n$ and epimorphisms between them. There is a rather well-known ``Morita'' equivalence between categories
\[
\fun(\finset^{\le n}, \catname{D})\simeq \fun(\epi^{\le n}, \catname{D}).
\]
The equivalence is induced by a ``$\spaces$-$\epi$-bimodule'', i.e., a bifunctor
\[
\begin{array}{ccc}
\spaces\times \epi^{\op} & \to & \spaces\\
(X, i) & \mapsto &  X^{\wedge i}
\end{array}
\]
Note that for each $n$ we can restrict this to a $\finset^{\le n}$-$\epi^{\le n}$-bimodule. This bimodule induces a pair of adjoint functors
\[
-\otimes_{i} X^{\wedge i} : \begin{tikzcd}
           \fun(\epi^{\le n}, \catname{D})
           \arrow[r, shift left=.75ex]
            \arrow[r, phantom, "\simeq"] 
           &  \fun(\finset^{\le n}, \catname{D})
           \arrow[l, shift left=.75ex]      
        \end{tikzcd}  : \nat_X(X^{\wedge i}, -)
\]
The following statement was discovered by Pirashvili for $1$-categories when $\catname{D}$ is an abelian category~\cite{PirashviliDK}. It was proved by Helmstutler (with technical hypotheses) for stable model categories~\cite{Helmstutler}, and then by Walde for  $\infty$-categories in the generality stated here~\cite[Section 4.2]{Walde}
\begin{prop}\label{prop:Pirashvili-Helmstutler-Walde}
Suppose $\catname{D}$ is an idempotent complete additive $\infty$-category. The adjunction above is an equivalence of categories for any $n\le \infty$. 
\end{prop}
Functors from $\epi^{\le n}$ to $\catname{D}$ are, by definition $n$-truncated right comodules over the the commutative operad $\com$ with values in $\catname{D}$~\cite{Arone-Ching_Cross-effects}. So we also denote the category $\fun(\epi^{\le n}, \catname{D})$ by $\com-\comod_{\le n}(\catname{D})$. Combining Theorem~\ref{theorem:main general}, Lemma~\ref{lem: excisive finset} and Proposition~\ref{prop:Pirashvili-Helmstutler-Walde}, we obtain the  following corollary
\begin{cor}\label{cor:com comodule main}
Let $\catname{D}$ be an $\infty$-category with small colimits. There are equivalences of categories
\[
\poly_n(\fr, \catname{D})\xleftarrow[\rho_{\classifying}]{\simeq} \exc_n(\spaces, \catname{D}) \xrightarrow{\simeq} \fun(\finset^{\le n}, \catname{D}) \xrightarrow{\simeq}\com\!-\!\comod_{\le n}(\catname{D}).
\]
\end{cor}
\subsection*{Bifunctors}
Theorem~\ref{theorem:main general} and Corollary~\ref{cor:com comodule main} apply when $\catname{D}=\catname{Spectra}$ or $\ch$, and also when $\catname{D}={\catname{Spectra}}^{\op}$ or ${\ch}^{\op}$. To put it another way, our main results apply both to covariant and contravariant polynomial functors from $\fr$ to $\ch$ or $\catname{Spectra}$. In the second half of the paper we focus on applications to covariant functors. But we hope that our approach will prove useful for contravariant functors as well. Furthermore, recently there has been interest in (Ext groups between) bifunctors, i.e., functors $\fr\times{\fr}^{\op}\to \ab$. We will prove a version of Theorem~\ref{theorem:main general} for bifunctors. For this, we need to review the notions of polynomial/excisive functors of several variables (we will content ourselves with two)
\begin{defn}
Let $\catname{C}_1, \catname{C}_2, \catname{D}$ be categories. A functor $F\colon\catname{C}_1\times \catname{C}_2\to \catname{D}$ is called $(n_1,n_2)$-polynomial if for every object $X$ of 
$\catname{C}_1$ the functor $F(X, -)\colon \catname{C}_2\to \catname{D}$ is polynomial of degree $n_2$, and for every $Y$ in $\catname{C}_2$ the functor $F(-, Y)\colon \catname{C}_1\to \catname{D}$ is polynomial of degree $n_1$. 
In an analogous way we define $(n_1, n_2)$-excisive functors. Let $\poly_{(n_1, n_2)}(\catname{C}_1\times\catname{C}_2, \catname{D})$ denote the category of $(n_1, n_2)$-polynomial functors, and similarly let $\exc_{(n_1, n_2)}(\catname{C}_1\times\catname{C}_2, \catname{D})$ denote the category of $(n_1, n_2)$-excisive functors.
\end{defn}
There are obvious isomorphisms of categories
\[
\poly_{(n_1, n_2)}(\catname{C}_1\times\catname{C}_2, \catname{D}) \cong \poly_{n_1}(\catname{C}_1, \poly_{n_2}(\catname{C}_2, \catname{D}))
\]
and 
\[
\exc_{(n_1, n_2)}(\catname{C}_1\times\catname{C}_2, \catname{D}) \cong \exc_{n_1}(\catname{C}_1, \exc_{n_2}(\catname{C}_2, \catname{D})).
\]
Now we can prove a bifunctor version of Theorem~\ref{theorem:main general}
\begin{prop}\label{prop:bivariant}
Suppose $\catname{D}$ is an $\infty$-category that has small limits and colimits. The classifying space functor $\classifying\times{\classifying}^{\op} \colon \fr\times {\fr}^{\op}\to \spaces \times {\spaces}^{\!\!\!\op}$ induces an equivalence, for every pair $(n_1, n_2)$ of natural numbers
\[
\exc_{(n_1, n_2)}(\spaces\times {\spaces}^{\!\!\!\op}, \catname{D})\xrightarrow[\rho_{\classifying\times {\classifying}^{\op}}]{\simeq} \poly_{(n_1, n_2)}(\fr\times {\fr}^{\op}, \catname{D}).
\]
\end{prop}
\begin{proof}
It is clear that the restriction functor $\rho_{\classifying\times {\classifying}^{\op}}$ is equivalent to the following composition, where each factor is an equivalence either by adjunction, or by an application of Theorem~\ref{theorem:main general}
\begin{multline*}
\exc_{(n_1, n_2)}(\spaces\times {\spaces}^{\!\!\!\op}, \catname{D}) \to \exc_{n_1}(\spaces, \exc_{n_2}({\spaces}^{\!\!\!\op}, \catname{D}))\to \\ \to \exc_{n_1}(\spaces, \poly_{n_2}({\fr}^{\op}, \catname{D}))\to \poly_{n_1}(\fr, \poly_{n_2}({\fr}^{\op}, \catname{D}))\to \\ \to \poly_{(n_1, n_2)}(\fr\times {\fr}^{\op}, \catname{D}).
\end{multline*}
\end{proof}
\subsection*{When the target category is abelian}
Next, we want to apply Theorem~\ref{theorem:main general} to polynomial functors from $\fr$ to an abelian category. As before, we let $\catname{A}$ denote an abelian $1$-category with enough projectives, and a set of compact projective generators. Let $\ch_{\catname{A}}$ be the $\infty$-category of chain complexes over $\catname{A}$.
We have the following result about polynomial functors from $\sfr$ to $\catname A$.
\begin{thm}\label{thm:fr-to-top}
Let $\catname{A}$ and $\ch_{\catname{A}}$ be as above. Suppose $F\colon \fr\to \catname{A}$ is a polynomial functor of degree $n$. Then there exists an $n$-excisive functor $\widehat F\colon \spaces\to \ch_{\catname{A}}$, unique up to equivalence, for which there is an isomorphism of functors from $\circles$ to $A$
\[
\HH_0\widehat F|_{\circles} \cong F\circ \pi_1,
\]
and $\HH_i \widehat F|_{\circles} \cong 0$ for $i\ne 0$. 

Furthermore if $G\colon \fr \to \ab$ is another polynomial functor, then there is an isomorphism of graded groups
\[
\ext^*_{\fun(\fr, \catname{A})}(F, G)\cong \pi_{-*}\left(\spectralNat\left(\widehat F, \widehat G\right)\right)
\]
\end{thm}
\begin{proof}
As we mentioned already, by~\cite[Proposition 1.3.2.19]{LurieHA} there is an equivalence of categories $\catname{A}\simeq \ch_{\catname{A}}^\heartsuit$, where $\ch_{\catname{A}}^\heartsuit$ is the heart of $\ch_{\catname{A}}$. More concretely, $\ch_{\catname{A}}^\heartsuit$ is the full subcategory of $\ch_{\catname{A}}$ consisting of chain complexes whose homology is concentrated in degree zero. Thus we have a canonical (up to equivalence) fully faithful ``inclusion'' functor $\catname{A}\hookrightarrow \ch_{\catname{A}}$. It is clear that the inclusion functor preserves coproducts (as mentioned already in Example~\ref{example:heart}).

Let $F\colon \fr\to \catname{A}$ be a polynomial functor of degree $n$. Let $F'\colon\fr\to \ch_{\catname{A}}$ be the post-composition of $F$ with the inclusion functor. Note that $F'$ is characterised by the property that there is an isomorphism $\HH_0 F'\cong F$, and $\HH_n F'\cong 0$ for $n\ne 0$. Since the inclusion preserves coproducts, $F'$ is polynomial of degree $n$ by Lemma~\ref{lemma:additive preserves polynomial} (see also Example~\ref{example:heart}). By Theorem~\ref{theorem:main general} there are equivalences of categories
\[
\begin{tikzcd}
\poly_n(\fr, \ch_{\catname{A}}) \arrow[r, "\rho_{\pi_1}", "\simeq"'] &\poly_n(\circles, \ch_{\catname{A}}) & 
\exc_n(\spaces, \ch_{\catname{A}}) \arrow[l, "\rho"', "\simeq"].
\end{tikzcd}
\]
It follows that there is an essentially unique $n$-excicive functor $\widehat F\colon \spaces \to \ch_{\catname{A}}$ such that there is an equivalence $\widehat F|_{\circles} \simeq F'\circ \pi_1$. Since by construction $F'$ takes values in $\ch_{\catname{A}}^{\heartsuit}$, it follows that $\widehat F$ is determined by the requirement $\HH_* \widehat F|_{\circles}\cong \HH_* F'\circ \pi_1$, which is equivalent to saying that $\HH_0 \widehat F|_{\circles} \cong F\circ \pi_1$ and $\HH_i \widehat F|_{\circles} \cong 0$ for $i\ne 0$. This proves the first half of the theorem.

Let $G\colon \fr\to \catname{A}$ be another polynomial functor of degree $n$. Let $G'\colon \fr \to \ch_{\catname{A}}$ be the post-composition of $G$ with the inclusion functor. By the standard connection between homological algebra and stable homotopy theory (see Lemma~\ref{lem:key isomorphism} and the discussion preceding it), there is a natural isomorphism
\[
\ext^*_{\fun(\fr, \catname{A})}(F, G)\simeq \pi_{-*}\left(\spectralNat\left(F', G'\right)\right).
\]
Here by $\spectralNat\left(F', G'\right)$ we mean the spectral mapping object from $F'$ to $G'$ in the full category of functors $\fun(\fr,\ch_\catname{A})$. But in fact $F'$ and $G'$  are objects of the full subcategory $\poly_n(\fr, \ch_\catname{A})$. It follows that there is an isomorphism of graded groups
\[
 \pi_{*}\left(\spectralNat\left(F', G'\right)\right)\cong  \pi_{*}\left(\spectralmaps\hphantom{.}_{ \!\!\poly_n(\fr, \ch_\catname{A})}\!\left(F', G'\right)\right).
\]
By Theorem~\ref{theorem:main general} there is an equivalence of categories 
\[
\exc_n(\spaces, \ch_{\catname{A}})\simeq  \poly_n(\fr, \ch_\catname{A})
\]
induced by the functor $\classifying \fr \to \spaces$, and by definition $\widehat F$ and $\widehat G$ are the images of $F'$ and $G'$ under the inverse of this equivalence. It follows that there is a natural isomorphism 
\[
 \pi_{*}\left(\spectralmaps\hphantom{.}_{ \!\!\poly_n(\fr, \ch_\catname{A})}\!\left(F', G'\right)\right)\cong  \pi_{*}\left(\spectralmaps\hphantom{.}_{ \!\!\exc_n(\spaces, \ch_\catname{A})}\!\left(\widehat F, \widehat G\right)\right).
\]
Finally, since $\exc_n(\spaces, \ch_{\catname{A}})$ is a full subcategory of $\fun(\spaces, \ch_{\catname{A}})$, we have an isomorphism
\[
 \pi_{*}\left(\spectralmaps\hphantom{.}_{ \!\!\exc_n(\spaces, \ch_\catname{A})}\!\left(\widehat F, \widehat G\right)\right)\cong  \pi_{*}\left(\spectralNat\left(\widehat F, \widehat G\right)\right).
\]
Putting all these isomorphisms together, we obtain the isomorphism
\[
\ext^*_{\fun(\fr, \catname{A})}(F, G)\simeq \pi_{-*}\left(\spectralNat\left(\widehat F, \widehat G\right)\right).
\]
This is what we wanted to prove.
\end{proof}

\part{Calculations} \label{part:calculations}

\section{Functors that factor through abelianization}\label{section:dictionary}
Unlike in much of the first part, in this section ``categories'' and ``functors'' are meant in the $1$-categorical sense. Let $\catname{A}$ be an abelian category, and let $F\colon \fr \to \catname{A}$ be a polynomial functor. We continue to denote by $\widehat F\colon \spaces\to \ch_{\catname{A}}$ the extension of $F$, as characterized by Theorem~\ref{thm:fr-to-top}. In this section we calculate $\widehat F$ for a number of commonly occurring functors $F\colon \fr\to\ab$. In particular, we will prove the correctness of Table~\ref{tab:functors}. This will enable us to use Theorem~\ref{thm:fr-to-top} to perform calculations. We start with a fundamental example, on which many subsequent calculations will build. Recall that $\tensor^n\circ \abelianization$ is the $n$-fold tensor power of abelianization.
\begin{lem}\label{lem:abelianization transform}
There is a natural equivalence \[\widehat{\tensor^n\circ \abelianization}(X)=\Omega^n \widetilde\Z[X^{\wedge n}].\]
\end{lem}
\begin{proof}
By Theorem~\ref{theorem:main general} there is an equivalence of $\infty$-categories
\[
\exc_n(\spaces, \ch)\xrightarrow{\simeq} \poly_n(\fr, \ch)
\]
which is induced by restriction along the classifying space functor. By definition, $\widehat{\tensor^n\circ \abelianization}$ is the preimage of $\tensor^n\circ \abelianization$ under this equivalence. This means that we have to show that the following two functors from $\fr$ to $\ch$ are equivalent: one is $G\mapsto (G^{\abelianization})^{\otimes n}$ and the other is $G\mapsto \Omega^n\widetilde\Z[\classifying G^{\wedge n}].$ We claim that both functors take value in the heart of $\ch$, i.e., in chain complexes whose homology is concentrated in degree zero. For the functor $\tensor^n\circ \abelianization$ this holds by definition. Regarding the second functor, we know that $\widetilde\Z[\classifying G^{\wedge n}]$ is a simplicial abelian group whose homotopy groups are naturally isomorphic to the reduced homology group of $\classifying G^{\wedge n}$. Note that if $G\in \fr$, then $\classifying G$ is equivalent to a finite wedge of circles. It follows that the reduced homology of $\classifying G^{\wedge n}$ is concentrated in degree $n$. It follows that $\Omega^n\widetilde\Z[\classifying G^{\wedge n}]$ is a chain complex whose homology is concentrated in degree zero. 

To prove that two functors $\fr\to \ch$ that take values in $\ch^{\heartsuit}$ are equivalent, it is enough to show that they become isomorphic after applying $\pi_0/\HH_0$, as $1$-functors $\fr \to \ab$. To this end, we have isomorphisms, natural in $G$
\[
\pi_0(\Omega^n\widetilde\Z[\classifying G^{\wedge n}])\cong \pi_n(\widetilde\Z[\classifying G^{\wedge n}])\cong \widetilde \HH_n(\classifying G^{\wedge n})\cong [\widetilde\HH_1(\classifying G)]^{\otimes n}\cong (G^{\abelianization})^{\otimes n}.
\]


\end{proof}
\begin{rem}\label{rem:action}
The functor $\tensor^n \circ \abelianization$ has a natural action of $\Sigma_n$, permuting the factors. This action will be important when we consider applications to stable cohomology of automorphisms of free groups. It follows that $\widehat{\tensor^n \circ \ab}$ has a corresponding action of $\Sigma_n$. It is clear that the action of $\Sigma_n$ on $\Omega^n\widetilde \Z[X^{\wedge n}]$ corresponding to the action on $\tensor^n \circ \abelianization$ is the action by conjugation, where $\Sigma_n$ acts both on $\Omega^n$ and on $X^{\wedge n}$ by permuting coordinates. This is the correct action because with it the isomorphism of abelian groups
\[
\pi_0(\Omega^n\widetilde\Z[\classifying G^{\wedge n}])\cong (G^{\abelianization})^{\otimes n}
\]
respects the action of $\Sigma_n$.
\end{rem}
The next two examples concern the functors $\Lambda^n \circ \abelianization$ and $\Gamma^n \circ \abelianization$. Let us review these functors. For an abelian group $A$, $\Lambda^n(A)$ is defined to be the quotient of $A^{\otimes n}$ by the subgroup generated by all pure tensors $a_1\otimes \cdots \otimes a_n$ where at least two of the factors are the same. At least when $A$ is a free abelian group, there is an equivalent definition of $\Lambda^n(A)$ as the group of skew-symmetric tensors. Let $A^{\otimes n}_\pm$ denote the $n$-fold tensor power of $A$, equipped with the \emph{sign action} of $\Sigma_n$. Let $(A^{\otimes n}_\pm)^{\Sigma_n}$ denote the invariants of the sign action. There is a well-known homomorphism
\begin{equation}\label{eq:exterior}
\Lambda^n(A) \to {A^{\otimes n}_\pm}^{\Sigma_n}
\end{equation}
which is defined on pure tensors by the formula 
\[ 
a_1\otimes \cdots \otimes a_n\mapsto\sum_{\sigma\in \Sigma_n} \operatorname{sign}(\sigma) a_{\sigma(1)}\otimes \cdots \otimes a_{\sigma(n)}.
\]
The following lemma is elementary and well-known.
\begin{lem}\label{lemma:exterior}
The homomorphism~\eqref{eq:exterior} is well-defined and natural in $A$. When $A$ is a free abelian group it is an isomorphism.
\end{lem}
Subsequently $A$ will be $\widetilde \HH_1(X)$, where $X$ is a finite wedge of circles, so $A$ will be a finitely generated free abelian group. We are therefore justified in identifying the group ${A^{\otimes n}_\pm}^{\Sigma_n}$ with $\Lambda^n(A)$. Similarly, $\Gamma^n(A)$ is identified with ${A^{\otimes n}}^{\Sigma_n}$ - the invariants of the unsigned action of $\Sigma_n$ on $A^{\otimes n}$.
\begin{prop}\label{prop:exterior and divided powers}
\begin{enumerate}
\item If $F(G)=\Lambda^n (G^\abelianization) $ then $\widehat F(X)=\Sigma^{-n} \widetilde\Z[X^{\wedge n}_{\Sigma_n}]$ \label{eq:propexterior}
\item If $F(G)=\Gamma^n (G^\abelianization)$ then $\widehat F(X)=\Sigma^{-2n}\widetilde\Z[(SX)^{\wedge n}_{\Sigma_n}]$ \label{eq:dividedpower} 
\end{enumerate}
\end{prop}
\begin{proof}
\eqref{eq:propexterior} Using the uniqueness part of Theorem~\ref{thm:fr-to-top}, it is not hard to see that our task amounts to showing that when $X$ is restricted to the category $\circles$ of finite wedges of circles, there is an isomorphism of graded groups, natural in $X$
\[
\widetilde \HH_{*} (X^{\wedge n}_{\Sigma_n})\cong \Lambda^n(\widetilde \HH_*(X)).
\]
We will make use of the homology transfer for strict orbits. Let $K$ be a pointed regular CW-complex with a cellular action of a finite group $G$. It is well-known that there is a natural transfer homomorphism
\[
\widetilde\HH_*(X/G)\to \widetilde \HH_*(X)^G.
\]
One of the defining properties of the transfer homomorphism is that the composition
\begin{equation*}\label{eq:transfer property}
\widetilde\HH_*(X) \to \widetilde\HH_*(X/G)\to \widetilde \HH_*(X)^G\hookrightarrow \widetilde\HH_*(X)
\end{equation*}
is equal to $\sum_{g\in G} g_*$~\cite[Section III.2]{Bredon72}.

When $X$ is a wedge of circles, $X^{\wedge n}$ has the structure of a regular CW-complex with an action of $\Sigma_n$. Let us fix such an $X$. Then $\widetilde \HH_*(X)$ is a finitely generated free abelian group concentrated in degree 1, and $\widetilde \HH_*(X^{\wedge n})$ is concentrated in degree $n$. In degree $n$ $\widetilde \HH_*(X^{\wedge n})$ is canonically isomorphic, as a group with an action of $\Sigma_n$, to $\widetilde \HH_1(X)^{\otimes n}_{\pm}$.  It follows that the transfer induces a homomorphism
\[
\widetilde \HH_{*} (X^{\wedge n}_{\Sigma_n})\xrightarrow{\cong} \widetilde\HH_*(X^{\wedge n})^{\Sigma_n}.
\]
The target of this homomorphism is, by Lemma~\ref{lemma:exterior}, naturally isomorphic to $\Lambda^n(\widetilde \HH_*(X))$. So our goal is to show that the homomorphism above is an isomorphism, whenever $X$ is a wedge of circles.

For the purpose of this proof let us denote $F_n(X)=\widetilde \HH_{*} (X^{\wedge n}_{\Sigma_n})$ and $G_n(X)=\widetilde \HH_*(X^{\wedge n})^{\Sigma_n}$. We have a natural transformation $F_n\to G_n$, and we want to show that it is an isomorphism for all $n>0$ when evaluated on a finite wedge of circles.

Next we claim that $F_n$ and $G_n$ satisfy the same recursive formula. Namely, we claim that there is a natural isomorphism
\begin{equation}\label{eq:recursive}
F_n(X\vee Y)\cong \bigoplus_{i=0}^n F_i(X) \otimes F_{n-i}(Y)
\end{equation}
and $G_n$ satisfies an exactly analogous formula. Furthermore, the natural transformation $F_n\to G_n$ respects this decomposition. All this follows easily from the ``binomial formula'' for $(X\vee Y)^{\wedge n}$, and the naturality of the transfer. By induction on the number of circles, it follows that it is enough to show that the homomorphism $F_n(S^1)\to G_n(S^1)$ is an isomorphism for all $n$. When $n=1$ we get the transfer map for $\Sigma_1$, i.e., the trivial group, which certainly is an isomorphism. For $n>1$, it is well-known that the orbit space $S^n_{\Sigma_n}$ is contractible, so $F_n(S^1)=0$, and it is easy to see that $G_n(S^1)=0$ as well. So in this case the homomorphism $F_n(S^1) \to G_n(S^1)$ can only be an isomorphism.

\eqref{eq:dividedpower} Using the uniqueness part of Theorem~\ref{thm:fr-to-top} once again, we see that our task is to prove that the transfer induces an isomorphism, whenever $X$ is a finite wedge of circles.
\[
\widetilde \HH_{*} ((SX)^{\wedge n}_{\Sigma_n})\xrightarrow{\cong} \widetilde \HH_*((SX)^{\wedge n})^{\Sigma_n}.
\]
This is equivalent to showing that the tranfer induces an isomorphism, whenever $Y$ is a finite wedge sum of copies of $S^2$
\[
\widetilde \HH_{*} (Y^{\wedge n}_{\Sigma_n})\xrightarrow{\cong} \widetilde \HH_*(Y^{\wedge n})^{\Sigma_n}.
\]
Once again, it is easy to see that both the source and the target of this homomorphism satisfy the same recursive relation as in~\eqref{eq:recursive}. It follows that it is enough to prove that this is an isomorphism when $Y=S^2$. In this case we obtain a homomorphism of the following form
\begin{equation}\label{eq:evensphere}
\widetilde \HH_{*} (S^{2n}_{\Sigma_n})\to \widetilde \HH_*(S^{2 n})^{\Sigma_n}.
\end{equation}
We want to show that the transfer homomorphism~\eqref{eq:evensphere} is an isomorphism. It is well-known that the orbit space $S^{2n}_{\Sigma_n}$ is in fact homeomorphic to $S^{2n}$. It follows that the source and the target of~\eqref{eq:evensphere} are in fact isomorphic graded abelian groups, and are both isomorphic to $\Z$, concentrated in degree $2n$. Furthermore, the quotient map $S^{2n} \to S^{2n}_{\Sigma_n}$ induces multiplication by $n!$ on $\widetilde \HH_{2n}$. To see this, note that it is a map between oriented manifolds of dimension $2n$, and it restricts to an oriented $n!$-fold covering over an open subset of $S^{2n}_{\Sigma_n}$, so it has degree $n!$. 

On the other hand, the following composition of homomorphisms is also multiplication by $n!$
\[
\widetilde \HH_{*}(S^{2n}) \to \widetilde \HH_{*}(S^{2n}_{\Sigma_n}) \to \widetilde \HH_{*}(S^{2n})^{\Sigma_n}\xrightarrow{\cong} \widetilde \HH_{*}(S^{2n}).
\]
Indeed, by the defining property of the transfer map, the composition is equal to $\sum_{\sigma\in \Sigma_n}\sigma_*$. But $\Sigma_n$ acts trivially on $\widetilde \HH_{*}(S^{2n})$, so the composition is just multiplication by $n!$. Since the first homomorphism, and the composition are both multiplication by $n!$, it follows that the transfer homomorphism $\widetilde \HH_{*}(S^{2n}_{\Sigma_n}) \to \widetilde \HH_{*}(S^{2n})^{\Sigma_n}$ is an isomorphism. 
\end{proof}
\subsection{A general formula}\label{section:general formula}
So far we calculated $\widehat F$ by guessing it. But it is good to also have a general procedure for calculating it. We will now develop a kind of formula for $\widehat F$ in case when $F$ factors through abelianization. 

Let $F$ be a polynomial functor from $\fr$ to $\ab$. Recall that $\abelianization$ denotes the abelianization functor from $\fr$ to the category of finitely generated free abelian groups. We say that a functor $F\colon \fr \to \ab$ factors through abelianization if there is a functor $F^{\abelianization}$ from the category of finitely generated free abelian groups to abelian groups, such that there is an isomorphism of functors $F\cong F^{\abelianization}\circ \abelianization$. 
\begin{rem}
It is easy to see that if a factorization $F\cong F^{\abelianization}\circ \abelianization$ exists then $F^{\abelianization}$ is unique up to isomorphism. This is so because the abelianization functor (from free groups to free abelian groups) is full and essentially surjective on objects. Furthermore, if $F$ is polynomial of degree $n$ then so is $F^{\abelianization}$, since the abelianization functor preserves coproducts. 
\end{rem}
Suppose $X$ is a pointed (simplicial) set. Normally we interpret $\widetilde \Z[-]$ as an $\infty$-categorical functor, but now we need to commit ourselves to a strict (i.e., 1-categorical) model for this functor. So let us define $\widetilde \Z[-]$ to be a $1$-categorical functor from pointed (simplicial) sets to (simplicial) abelian groups by the formula $\widetilde \Z[X]:=\Z[X]/\Z[*]$. If $K$ is another pointed finite simplicial set, in this section we let $\pointedmaps(K, \widetilde \Z[X])$ denote the cosimplicial simplicial abelian group which in bi-degree $([i], [j])$ is given by $\hom(\widetilde \Z[K_i], \widetilde\Z[X_j])$. Again, note that in this section we deviate slightly from our customary notation. Normally $\pointedmaps(K, \widetilde \Z[X])$ would denote the cotensoring of the chain complex $\widetilde \Z[X]$ by $K$ in the $\infty$-category $\ch$, but here we use the $1$-categorical model for it as a cosimplicial-simplicial group. Since we assume that $K$ is finite, it follows that $\pointedmaps(K, \widetilde\Z[X])$ is free abelian in each bidegree. If $S^1$ is a simplicial circle, we use the notation $\Omega \widetilde \Z[X]:= \pointedmaps(S^1, \widetilde \Z[X])$.
\begin{defn}
Suppose $F^{\abelianization}$ is a functor from finitely generated free abelian groups to abelian groups. Then we denote by $F^{\abelianization}\left(\pointedmaps(K, \widetilde\Z[X])\right)$ the chain complex obtained in the following three steps 
\begin{enumerate}
\item \label{cosimplicial-simplicial} apply $F^\abelianization$ levelwise to the cosimplicial simplicial abelian group $\pointedmaps(K, \widetilde\Z[X])$, to get another cosimplicial simplicial abelian group
\item \label{Dold-Kan} apply the Dold-Kan normalized chains functor in both directions to the cosimplicial simplicial abelian group obtained in step~\eqref{cosimplicial-simplicial}, to get a second quadrant bicomplex, 
\item take the total complex of the bicomplex obtained in step~\eqref{Dold-Kan}.
\end{enumerate}
\end{defn}
Now we are ready to give our formula for $\widehat F$
\begin{thm}\label{thm:formula}
Let $F\colon \fr \to \ab$ be a polynomial functor. Suppose that $F$ factors through abelianization, so there is an isomorphism of functors $F\cong F^{\abelianization}\circ \abelianization$. Then on the category of pointed simplicial sets, there is an equivalence, natural in $X$
\[
\widehat F(X) \simeq F^{\abelianization}\left(\Omega \widetilde\Z [X]\right).
\]
\end{thm}
Before proving the theorem, we need to review the simplicial model for the circle that we will use. Recall that $\Delta$ is the category with objects $[0], [1], \ldots, [n], \ldots$. Here $[n]=\{0, \ldots, n\}$, and the morphisms in $\Delta$ are order-preserving functions between these sets. The standard simplicial model for the circle is the following contravariant functor $\Delta^{\op}\to \finset$:
\[
[n]\mapsto \Delta([n], [1])/\{c_0, c_1\}
\]
where $c_0, c_1$ denote the two constant maps from $[n]$ to $[1]$. We view it as a pointed simplicial set. Thus the following simplicial abelian group is a model for the reduced simplicial chains on $S^1$
\[
[n]\mapsto \widetilde \Z[\Delta([n], [1])/\{c_0, c_1\}]
\]
where $\widetilde \Z[-]$ denotes the pointed free abelian group functor. 

On the other hand, there is a simplicial abelian group
\[
[n]\mapsto \Z^{n+1}/\mathbb Z
\]
where $\Z^{n+1}:=\Z^{[n]}$, and the group is quotiened by the diagonal copy of $\Z$, that is by the group of constant function from $[n]$ to $\Z$. We denote this simplicial abelian group by $\Z^{\bullet +1}/\Z$.

The inclusion of sets $[1]\hookrightarrow \Z$ induces a function $\Delta([n], [1])\to \Z^{n+1}$, which extends to a group homomorphism $\Z[\Delta([n], [1])]\to \Z^{n+1}$. It is clear that it takes constant functions to constant functions, to it induces a homomorphism
\[
\widetilde \Z[\Delta([n], [1])/\{c_0, c_1\}]\to \Z^{n+1}/\Z.
\]
This homomorphism is obviously natural in $[n]$, and it is easy to check that it is an isomorphism (both the source and the target are isomorphic to $\Z^n$). We have proved the following lemma
\begin{lem}\label{lem:circle}
The simplicial abelian group $\Z^{\bullet+1}/\Z$ is isomorphic to $\widetilde \Z[S^1_{\bullet}]$, where $S^1_{\bullet}$ is the standard simplicial circle.   
\end{lem}
\begin{proof}[Proof of Theorem~\ref{thm:formula}]
By Theorem~\ref{thm:fr-to-top_intro}, we know the value of $\widehat F$ on wedges of circles. The key step is to find the value of $\widehat F$ on pointed finite sets. From there, one can evaluate $\widehat F$ on general spaces by left Kan extension from $\finset$ to $\spaces$, which, if we take spaces to mean simplicial sets, is equivalent to evaluating $\widehat F$ level-wise and then taking geometric realization.  

Let $X$ be a pointed finite set. By Corollary~\ref{cor:connected}, we can calculate 
$\widehat F(X)$ as $\widetilde T_\infty \widehat F(X).$ 
Here $\widetilde T_\infty$ is essentially Goodwillie's operator $T_\infty$ defined in~\cite{Goodwillie2003}, and we will now omit $\sim$ from the notation. The reason $T_\infty \widehat F(X)$ is defined when $X$ is a pointed finite set is that $T_\infty \widehat F(X)$ is a homotopy limit of objects of the form $F(\pi_1(X * U))$, where $U$ is a non-empty finite set. Here $X*U$ denotes the pointed joint of $X$ and $U$, i.e. the mapping cone of the map $X\wedge U_+ \to X$. Note that whenever $X$ is a pointed set and $U$ is a non-empty set, $X*U$ is homotopy equivalent to a wedge of circles, and thus there is a natural equivalence $\widehat F(X * U)\simeq F(\pi_1(X * U)).$

We will use a cosimplicial model for $T_\infty \widehat F(X)$. For all $X$, $T_\infty \widehat F(X)$ is equivalent to the totalization of the following cosimplicial object
\[\begin{tikzcd}
	{\widehat F(X*[0])} & {\widehat F(X*[1])} & {\widehat F(X*[2])} & \cdots
	\arrow[shift left, Rightarrow, from=1-1, to=1-2]
	\arrow[shift left, from=1-2, to=1-1]
	\arrow[shift left, Rightarrow, scaling nfold=3, from=1-2, to=1-3]
	\arrow[shift left, Rightarrow, from=1-3, to=1-2]
	\arrow[shift left, Rightarrow, scaling nfold=4, from=1-3, to=1-4]
	\arrow[shift left=2, Rightarrow, scaling nfold=3, from=1-4, to=1-3]
\end{tikzcd}\]
This cosimplicial model for $T_\infty$ appeared in a preprint version of~\cite{Goodwillie2003}, where it was credited to Waldhausen. The cosimplicial model did not make it into the published version of~\cite{Goodwillie2003}. There is an account of it in~\cite{Eldred2013}.

Since we assume $X$ is a pointed finite set, $X*[n]$ is equivalent to a wedge of circles for every $n$. By the defining property of $\widehat F$ given in Theorem~\ref{thm:fr-to-top_intro}, there is a natural equivalence $\widehat F(X*[n])\simeq F(\pi_1(X *[n]))$. We assume that $F$ factors as $F^{\abelianization}\circ \abelianization$, so there is furthere a natural equivalence $\widehat F(X*[n])\simeq F^{\abelianization}(\widetilde \HH_1(X *[n]))$. Recall that $X*[n]$ is the homotopy cofiber of the map $X\wedge [n]_+ \to X$. It follows that there are isomorphisms, natural in $X$ and $[n]$
\[
\widetilde \HH_1(X *[n]))\cong \widetilde \Z[X]\otimes \ker\left(\bigoplus_{n+1}\Z \to \Z\right)\cong \hom(\Z^{n+1}/\Z, \widetilde \Z[X]).
\]
It follows $\widehat F(X)$ is the totalisation of the cosimplicial object obtained by applying $F^{\abelianization}$ levelwise to $\hom(\Z^{\bullet +1}/\Z, \widetilde \Z[X])$. By Lemma~\ref{lem:circle} this cosimplicial abelian group is isomorphic to $\hom(\widetilde \Z[S^1_\bullet], \widetilde \Z[X])\cong \Omega \widetilde \Z[X].$

We have proved the theorem in case when $X$ is a pointed finite set. Since $\widehat F$ is an excisive functor, it preserves geometric realizations (Lemma~\ref{lemma:preserve sifted}). Therefore, if $X$ is a pointed finite simplicial set, one can evaluate $\widehat F$ on $X$ by applying it level-wise and then taking geometric realization. The geometric realization of a simplicial cosimplicial abelian group is equivalent to the total complex of the associated bi-complex, and therefore we can conclude that $\widehat F(X)\simeq F^{\abelianization}(\Omega \widetilde \Z[X])$ when $X$ is a finite simplicial set.
\end{proof}
By way of application we can now give a kind of formula for $\widehat F$ when $F$ is the symmetric powers functor. 
\begin{cor}\label{cor:symmetric-power}
Let $\symm^n$ be the $n$-th symmetric power functor, and consider the functor $\symm^n\circ\abelianization(G)={G^{\abelianization}}^{\otimes n}_{\Sigma_n}$. Then $\widehat{\symm^n\circ \abelianization}(X)=\left(\Omega^n \widetilde \Z [X^{\wedge n}]\right)_{\Sigma_n}.$
\end{cor}
\begin{proof}
    The corollary follows from Theorem~\ref{thm:formula}, and the isomorphism
    \[
    \hom(\widetilde \Z[A],\widetilde \Z[X])^{\otimes n}\xrightarrow{\cong} \hom(\widetilde \Z[A^{\wedge n}], \widetilde \Z[X^{\wedge n}].
    \]
    Here $A$ and $X$ are pointed finite sets, and the isomorphism is natural both in $A$ and in $X$.
\end{proof}
\subsubsection*{Divided power functor, revisited} For another example, let us calculate $\widehat F$ for $F(G)=\Gamma^n(G^{\abelianization})$ for the second time. We know from Proposition~\ref{prop:exterior and divided powers}~\eqref{eq:dividedpower} that $\widehat F(X)=\Sigma^{-2n}\widetilde\Z[(SX)^{\wedge n}_{\Sigma_n}]$. 
On the other hand, it follows from Theorem~\ref{thm:formula} that $\widehat F(X)=\left(\Omega^n \widetilde\Z [X^{\wedge n}]\right)^{\Sigma_n}$. For sanity check, let us prove directly that there is an equivalence of functors
\begin{equation}\label{eq: two functors}
\Sigma^{-2n}\widetilde \Z[(SX)^{\wedge n}_{\Sigma_n}]\simeq \left(\Omega^n \widetilde\Z [X^{\wedge n}]\right)^{\Sigma_n}.
\end{equation}
Both functors are excisive of degree $n$, so both are left Kan extended from $\finset^{\le n}$ by Lemma~\ref{lem: excisive finset}, and thus it is enough to show that the restrictions of these functors to $\finset^{\le n}$ are equivalent. Next, by Proposition~\ref{prop:Pirashvili-Helmstutler-Walde}, it is enough to show that the cross-effects of these functors are equivalent as functors from $\epi^{\le n}$ to $\ch$. So let us calculate the cross-effects of both functors. Let $F_1$ and $F_2$ denote the functors on the left hand side and the right hand side of~\eqref{eq: two functors}. We have
\[
\ce_iF_1(S^0, \ldots, S^0)=\Sigma^{-2n}\widetilde \Z\left[\bigvee_{n_1+\cdots+n_i=n}S^{n}_{\Sigma_{n_1}\times \cdots \times \Sigma_{n_i}}\right]
\]
The wedge sum is indexed by ordered $n$-tuples $(n_1, \ldots, n_i)$ of positive integers that add up to $n$. Note that if $i<n$ then at least one of the numbers $n_1, \ldots, n_i$ is greater than one, and then $S^n_{\Sigma_{n_1}\times\cdots \times \Sigma_{n_i}}\simeq *$. Thus $\ce_iF_1(S^0, \ldots, S^0)$ is homotopically trivial for $i<n$. On the other hand, $\ce_n F_1(S^0, \ldots, S^0)\simeq \Sigma^{-2n}\widetilde \Z[S^n]$. As a chain complex with an action of $\Sigma_n$, $\ce_n F_1(S^0, \ldots, S^0)$ is equivalent to the chain complex concentrated in degree $-n$, in which it is isomorphic to $\Z$ on which $\Sigma_n$ acts by sign.
On the other hand, an easy calculation shows that 
\[
\ce_i F_2(S^0, \ldots, S^0)=\prod_{n_1+\cdots+n_i=n} \map_*(S^n/_{\Sigma_{n_1}\times\cdots\times \Sigma_{n_i}}, \Z)
\]
where the product is, as before, over $i$-tuples of positive integers that add up to $n$. Once again, it is easy to see that $\ce_i F_2(S^0, \ldots, S^0)$ is trivial for $i<n$, and $\ce_n F_2(S^0, \ldots, S^0)=\Omega^n\Z$, which, again, is equivalent to the complex that has the sign representation in degree $-n$ and is zero otherwise.

We found that $\ce_iF_1$ and $\ce_iF_2$ are both trivial for $i<n$, and are equivalent as complexes with an action of $\Sigma_n$ for $i=n$. It follows that they are equivalent as functors from $\epi^{\le n}$ to $\ch$.

\section{Passi functors}\label{section:Passi} Let $G$ be a group. Recall that we have defined $\ideal(G):=\ker(\Z[G]\to \Z)$, and $\Passi_n(G):=\ideal(G)/\ideal(G)^{n+1}$. $\Passi_n$ is called the $n$-th Passi functor. As usual, we will only evaluate it on finitely generated free groups. It is well-known that $\Passi_1$ is the abelianization functor, but for $n>1$, $\Passi_n$ does not factor through abelianization, and this is our main example of functors that do not factor through abelianization. In this section we will describe $\widehat \Passi_n$ (Proposition~\ref{prop:passi}). In Sections~\ref{section:ab to Passi} and~\ref{section: ext from Passi} we will use the result of this section to calculate Ext into the Passi functor and out of the Passi functor respectively.

One of the reasons that the Passi functors are interesting is that $\Passi_n$ provides a universal $n$-polynomial approximation to the free functor $G\mapsto \widetilde \Z[G]$. Since $\widetilde Z[G]\simeq \widetilde \Z[\Omega BG]$, one may expect the functor $\widehat \Passi_n$ to be the universal $n$-excisive approximation to the functor $X\mapsto \widetilde Z[\Omega X]$. The purpose of this section is to make this intuition precise.

Recall that $P_n F$ denotes Goodwillie's $n$-excisive approximation of $F$. Consider the functor $X\mapsto \widetilde\Z[\Omega X]$. It will be useful to consider also the unreduced version of the functor $\Z[\Omega X]=\widetilde\Z[\Omega X_+]$. There are equivalences
\[
\widetilde\Z[\Omega X]\simeq \operatorname{cofiber}(\Z\to \Z[\Omega X] \simeq \operatorname{fiber}(\Z[\Omega X]\to \Z).
\]
Note that $\operatorname{fiber}(\Z[\Omega X]\to Z)=I(\Omega X)$. So we really have an equivalence
$I(\Omega X) \xrightarrow{\simeq} \widetilde \Z[\Omega X].$

The functor $\widehat \Passi_n$ will be described in terms of $P_n\widetilde\Z[\Omega -]$, so we will review what is known about the Taylor tower of the functor $\widetilde\Z[\Omega -]$. One way to think about this functor is as the functor $X\mapsto H\Z \otimes \Omega X$ from pointed spaces to $H\Z$-modules. Let $\epi^{\le n}_0$ denote the full subcategory of $\epi^{\le n}$ consisting of non-empty sets. In~\cite{Arone-thesis} it was shown that the $n$-th Taylor approximation of the functor $X\mapsto E\otimes \pointedmaps(K, X)$, where $E$ is a spectrum and $K$ is a finite CW-complex has the following form:
\[
P_n(E\otimes \pointedmaps(K, -))(X)=\underset{i\in \epi^{\le n}_0}{\spectralNat}( K^{\wedge i}, E\otimes X^{\wedge i}).
\]
Note that the assignments $i\mapsto 
K^{\wedge i}, E\otimes X^{\wedge i}$ give contravariant functors from $\epi$ to $\spaces$ and $\catname{Spectra}$ respectively. It is perhaps worth noting that the functor $i\mapsto K^{\wedge i}$ is cofibrant in the projective model structure on $\fun(\epi^{\op}, \spaces)$. Therefore with minor assumptions on cofibrancy of $\Sigma^\infty$ and fibrancy of $E\otimes X^{\wedge i}$ one can take $\spectralNat$ in the formula above to be a $1$-categorical mapping object.

The natural map $E\otimes \pointedmaps(K, X)\to \underset{i\in \epi^{\le n}_0}{\spectralNat}( K^{\wedge i}, E\otimes X^{\wedge i})$ is defined to be adjoint to the natural map 
\[
K^{\wedge i}\otimes E\otimes \pointedmaps(K, X)\to  K^{\wedge i}\otimes E\otimes \pointedmaps(K^{\wedge i}, X^{\wedge i})\to E\otimes X^{\wedge i}.
\]

Specialising to $K=S^1$ and $E=H\Z$, we can write the following model for the Taylor tower of $\widetilde\Z[\Omega -]$
\[
P_n\widetilde\Z[\Omega -](X)=\underset{i\in \epi^{\le n}_{0}}{\spectralNat}(S^i, \widetilde\Z[X^{\wedge i}]).
\]
\begin{prop}\label{prop:passi}
There is an equivalence, natural in $X$
\[
\widehat \Passi_n(X)\simeq P_n(\widetilde\Z[\Omega -])(X)\simeq \underset{i\in \epi^{\le n}_0}{\nat}(S^i, \widetilde\Z[X^{\wedge i}]).
\]
\end{prop}
\begin{proof}
It follows from Proposition~\ref{prop:poly to exc}, together with Corollary~\ref{cor:connected_to_fr} that there is a homotopy commutative diagram of $\infty$-categories
\begin{equation}\label{eq:connected}
\begin{tikzcd}
	{\fun_{\Sigma}(\connected, \ch)} & {\fun(\fr, \ch)} \\
	{\exc_n(\connected, \ch)} & {\poly_n(\fr, \ch)}
	\arrow["\rho_{\classifying}", "\simeq"', from=1-1, to=1-2]
	\arrow["{P_n}"', from=1-1, to=2-1]
	\arrow["{\tilde q_n}", dashed, from=1-2, to=2-2]
	\arrow["\rho_{\classifying}", "\simeq"', from=2-1, to=2-2]
\end{tikzcd}
\end{equation}
Here the horizontal functors are induced by the classifying space functor. The left vertical functor is Goodwillie's $P_n$. By definition, $P_n$ is the ($\infty$-categorical) left adjoint to the inclusion functor (see Remark~\ref{rem:excisive-sifted}). It follows that there is an essentially unique functor $\tilde q_n$ that makes the diagram commute. Furthermore, $\tilde q_n$ is the $\infty$-categorical left adjoint of the inclusion $\poly_n(\fr, \ch)\hookrightarrow \fun_{\Sigma}(\fr, \ch)$. 

Consider the functor $\widetilde \Z[\Omega -]\colon\connected \to \ch$. Recall that $\widetilde Z[Y]=\Z[Y]/Z[*]$, and thus $\widetilde Z[Y]$ is also equivalent to $\ker(\Z[Y]\to \Z)$. Suppose $G$ is a finitely generated free group. We have equivalences, natural in $G$
\[
\widetilde \Z[\Omega \classifying G]\simeq \widetilde \Z[G] \cong \ideal(G).
\]
It follows that the equivalence of categories at the top of diagram~\eqref{eq:connected} takes $\widetilde \Z[\Omega -]$ to $\ideal(-)$. Thus to prove the proposition we need to show that there is a natural equivalence $\Passi_n\simeq \tilde q_n(\ideal)$, where $\ideal$ is the augmentation ideal functor. 

It is in fact well-known that $\Passi_n\cong q_n(\ideal)$, where \[q_n\colon \fun(\fr, \ab)\to \poly_n(\fr, \ab)\] is the \emph{one-categorical} left adjoint of the inclusion functor. See for example~\cite[Proposition 3.7]{Djament-Pirashvili-Vespa_2016}. So what we want is equivalent to $\tilde q_n(\ideal)\simeq q_n(\ideal)$. This, too, follows from results in~\cite{Djament-Pirashvili-Vespa_2016}. Indeed, to show that $\Passi_n\simeq \tilde q_n(\ideal)$, it is enough to show that for every $n$-polynomial functor $T\colon \fr\to \ch$, the quotient map $\ideal\to \Passi_n$ induces an equivalence of $\infty$-categorical spectral objects of natural transformations
\begin{equation}\label{eq:concentrated}
\spectralNat(\Passi_n, T)\xrightarrow{\simeq} \\ \spectralNat(\ideal, T).
\end{equation}
It is enough to prove this equivalence when $T$ is a functor that takes values in chain complexes homologically concentrated in degree zero. In other words, if $T$ is in the essential image of the ``inclusion'' functor $\poly_n(\fr, \ab)\to \poly_n(\fr, \ch)$. Indeed, if~\eqref{eq:concentrated} is an equivalence for $T$ concentrated in degree zero, then it is true for all $T$ that take values in bounded chain complexes, and then it follows that it is true for all $T$, because every chain complex is naturally equivalent to the inverse homotopy limit of bounded complexes.

Note that when $T$ is a functor taking values in chain complexes concentrated in degree zero, then both the domain and the range of~\eqref{eq:concentrated} are spectra whose homotopy groups vanish in positive degrees. So we only need to show that the map induces isomorphism on homotopy groups in degrees $\le 0$. Thus, to prove that~\eqref{eq:concentrated} is an equivalence whenever $T$ is concentrated in degree zero, is equivalent to showing that the quotient homomorphism $\ideal\to \Passi_n$ induces isomorphisms for all $i\ge 0$
\[
\ext_{\fun(\fr, \ch)}^i(\Passi_n, T)\xrightarrow{\cong}\ext_{\fun(\fr, \ch)}^i(\ideal, T).
\]
For $i=0$ this is an isomorphism because $\Passi_n=q_n(\ideal).$ For $i>0$ note that $\ideal(G)$ is a direct summand of $\Z[G]\cong \Z[\hom(\Z, G)]$, naturally in $G$. Thus $\ideal$ is a projective object in $\fun(\fr, \ch)$, and therefore $\ext_{\fun(\fr, \ch)}^i(\ideal, T)=0$ for $i>0$. So it remains to show that $\ext_{\fun(\fr, \ch)}^i(\Passi_n, T)=0$ for $i>0$. This follows from the following two resuts of~\cite{Djament-Pirashvili-Vespa_2016}. First, there is~\cite[Corollary 4.3]{Djament-Pirashvili-Vespa_2016}, which says that $\Passi_n$ has homological dimension zero in $\poly_n(\fr, \ab)$, and therefore $\ext_{\poly_n(\fr, \ch)}^i(\Passi_n, T)=0$ for $i>0$. Second, Theorem 1 of [loc. cit] says that there are isomorphisms for all $i$
\[
\ext_{\poly_n(\fr, \ch)}^i(\Passi_n, T)\cong \ext_{\fun(\fr, \ch)}^i(\Passi_n, T).
\]
This implies the result we want.
\end{proof}
\section{$\ext^*(\tensor^n\circ \abelianization, -)$}\label{section:ext from tensor}
In the next few sections we will use our general results to calculate $\ext$ groups between polynomial functors from $\fr$ to $\ab$. We begin with investigating $\ext^*(\tensor^m\circ \abelianization, G)$, where $G\in \poly(\fr, \ab)$. We remind the reader that $\tensor^m$ is the $n$-fold tensor power functor, and $\abelianization$ is abelianization. The calculation will be in terms of the cross-effects of $\widehat G$. In this section we use the term cross-effect as defined by Goodwillie~\cite{Goodwillie2003}. Let us recall the definition
\begin{defn}[Cross-effects, in the sense of Goodwillie]
    Let $H\colon \catname{C}\to \catname{D}$ be a functor, where $\catname{C}$ is pointed and $\catname{D}$ is a stable $\infty$-category. The $n$-th cross-effect of $H$ is a functor $\ce_nH\colon \catname{C}^n\to \catname{D}$. The value $\ce_nH(X_1, \ldots, X_n)$ is the total homotopy fiber of the $n$-dimensional cubical diagram that sends a subset $U\subset \{1, \ldots, n\} $ to 
    $H\left(\bigvee_{i\notin U} X_i\right),$ 
    and whose maps are induced by collapsing the appropriate $X_i$s to the zero object.
\end{defn}
It is well-known, and not difficult to prove that this definition of cross-effect is equivalent to the one in Definition~\ref{def:cross-effects}. See, for example~\cite[Lemma 3.29]{Arone-Barthel-Heard-Sanders}, which establishes this equivalence for functors from Spectra to Spectra.

Here is a general result regarding $\ext^*(\tensor^m\circ \abelianization, G)$.
\begin{prop}\label{prop:tensor-general}
Let $G\in \poly(\fr, \ab)$, and let $\widehat G\in \exc(\spaces, \ch)$ be the extension of $G$, as characterised in Theorem~\ref{thm:fr-to-top_intro}. Then there is an equivalence, respecting the action of $\Sigma_m$
\[
\ext^*(\tensor^m\circ \abelianization, G)\cong  \pi_{-*}\left( S^m\otimes \ce_m\widehat G(S^0, \ldots, S^0)\right) \cong \pi_{-*-m}\ce_m\widehat G(S^0, \ldots, S^0)^{\pm}.
\]
Here the $\pm$ indicates that the action of $\Sigma_m$ on the right hand side is twisted by the sign representation.
\end{prop}
\begin{proof}
By Lemma~\ref{lem:abelianization transform},
$\widehat{\tensor^m \circ \abelianization}\colon \spaces \to \ch$ is the functor \[\widehat{\tensor^m \circ \abelianization}(X)=\Omega^m\widetilde \Z[X^{\wedge m}].\] 
By Theorem~\ref{thm:fr-to-top_intro}, there is an isomorphism   
\[
\ext^*(\tensor^m\circ \abelianization, G)\cong \pi_{-*}\spectralNat_X(\Omega^m\widetilde \Z[X^{\wedge m}], \widehat G(X))
\]
and the right hand side is isomorphic to $\pi_{-*}\spectralNat_X(\widetilde \Z[X^{\wedge m}], S^m\otimes\widehat G(X))$.

Let us recall the functor $X^{\wedge m}$ represents the $m$-th cross effect. The following lemma is elementary and well-known.
\begin{lem}\label{lem: smash power represents ce}
For any functor $H\colon\spaces\to \catname{D}$ where $\catname{D}$ is an $\infty$-category there is an equivalence
\begin{equation}\label{eq: nat from smash power is cross-effect-general version}
    \nat_X(X^{\wedge m}, H(X))\simeq \ce_m H(S^0, \ldots, S^0).
\end{equation}
\end{lem}
\begin{proof}
Notice that if we replace $X^{\wedge m}$ with $X^m$, $X^m\cong\pointedmaps(m_+, X)$ is a representable functor, represented by $m_+$, and we obtain an equivalence by the Yoneda lemma
\[
\nat(X^{m}, H)\simeq H(\underbrace{S^0\vee\ldots\vee S^0}_{m}).
\]
Now write $X^{\wedge m}$ as a total cofiber of a cube of representable functors, indexed by subsets of $\{1, \ldots, m\}$, and apply $\nat(-, H)$. For more details of this argument, see~\cite[Lemma 3.13]{Arone-Ching_Cross-effects}, where the claim is proved in the context of model categories for functors from spaces to spectra.
\end{proof}
Using the free-forgetful adjunction between $\spaces$ and $\ch$, it follows that if $H$ is a functor from $\spaces$ to $\ch$ there is an equivalence
\begin{equation}\label{eq: nat from smash power is cross-effect}
    \spectralNat_X(\widetilde \Z[X^{\wedge m}], H(X))\simeq \ce_m H(S^0, \ldots, S^0).
\end{equation}

Taking cross-effects commutes with suspension. It follows that there is a natural isomorphism
\[
\ext^*(\tensor^m\circ \abelianization, G)\cong \pi_{-*}(S^m\otimes \ce_m\widehat G(S^0, \ldots, S^0)).
\]
Note that the action of $\Sigma_m$ on $S^m$ induces the sign action on its non-trivial reduced homology group. It follows that the right hand side is isomorphic to
$\pi_{-*-m}(\ce_m\widehat G(S^0, \ldots, S^0))$, where the action of $\Sigma_m$ is twisted by the sign.
\end{proof}
As a consequence, we can calculate $\ext$ from $\tensor^m \circ \abelianization$ to most of the functors in Table~\ref{tab:functors}. We begin with the easy ones. Recall that if $A$ is a graded abelian group, or a chain complex, $\Sigma^n A$ is the shift of $A$ by $n$. In the absence of grading, a group is considered concentrated in degree zero. Recall that $\sur(n, m)$ denotes the set of surjections from $\{1, \ldots, n\}$ to $\{1, \ldots, m\}$. 
Let $\Comp(n, m)={_{\Sigma_n}}\!\backslash\sur(n, m)$ denote the set of compositions of $n$ with $m$ components.
\begin{cor}\label{cor: tensor-others}
There are isomorphisms, that respect the action of the appropriate groups of automorphisms.
\begin{enumerate}
    \item \[\ext^i(\tensor^m \circ \abelianization, \tensor^n \circ \abelianization)\cong \left\{\begin{array}{cl} \Z[\sur(n, m)] & i=n-m \\ 0 & \mbox{otherwise}\end{array} \right.,\] \label{item:tensor-tensor}
    where the natural action of the group $\Sigma_m\times\Sigma_n$ on $\Z[\sur(n,m)]$ on the right hand side is twisted by the tensor product of the sign representations of $\Sigma_m$ and of $\Sigma_n$.
    \item \[
    \ext^i(\tensor^m\circ \abelianization, \Lambda^n \circ \abelianization)\cong \left\{ \begin{array}{cl} \Z[\Comp(n, m)] & i=n-m\\ 0 & \mbox{otherwise} \end{array}\right.,\]\label{item:tensor-exterior}
where the action of $\Sigma_m$ on $\Z[\Comp(n, m)]$ is twisted by the sign.
       \item \[\ext^i(\tensor^m\circ \abelianization, \Gamma^n\circ \abelianization) \cong \left\{\begin{array}{cl} \Z & i=m=n=0\\ 0 & \mbox{otherwise}\end{array}\right.. 
       \] \label{item:tensor-dividedP} 
       Here $\Sigma_m$ acts trivially on the right hand side.
\end{enumerate}
\end{cor}
\begin{rem}
Part~\eqref{item:tensor-tensor} is~\cite[Theorem 1]{Vespa2018}. See also [loc. cit., Proposition 2.5] for a discussion of signs.
\end{rem}
\begin{proof}
\eqref{item:tensor-tensor}    By Proposition~\ref{prop:tensor-general} we need to calculate \[S^m\otimes\ce_m\widehat{\tensor^n \circ \abelianization}(S^0, \ldots, S^0).\] By Lemma~\ref{lem:abelianization transform}, $\widehat{\tensor^n \circ \abelianization}$ is the functor $X\mapsto \Omega^n\widetilde Z[X^{\wedge n}]$, so we need to calculate the cross-effect of this functor. Shift commutes with cross-effects, so we suppress $\Omega^n$ at first. Consider the case $X=U_+=\underset{U}{\bigvee}S^0$ where $U$ is a finite set. We have $\widetilde Z[U_+^{\wedge n}]\cong Z[U^n]$. It follows that $\ce_m\widetilde \Z[-^{\wedge}](S^0, \ldots, S^0)$ is equivalent to the total fiber of the contravariant cubical diagram $U\mapsto \Z[U^n]$, were $U$ ranges over subsets of $\{1, \ldots, m\}$. Furthermore, the maps in the cubical diagram are defined as follows: if $U_1\subset U_2$ then the corresponding map $ \Z[U_2^n]\to\Z[U_1^n]$ sends a map $n\to U_2$ to itself if its image is in $U_1$, and to $0$ if its image is not in $U_1$. It is easy to see that taking the total fiber of this diagram amounts to ``killing'' the non-sujective functions from $\{1, \ldots, n\}$ to $\{1, \ldots, m\}$. Thus we obtain that
\[
\ce_m \Omega^n\widetilde \Z[-^{\wedge n}](S^0, \ldots, S^0)\cong \Omega^{n}\Z[\sur(n, m)].
\]
It follows that 
\[
\ext^*(\tensor^m \circ \abelianization, \tensor^n \circ \abelianization)\cong \pi_{-*}\left(S^m\otimes \Omega^n\Z[\sur(n, m)]\right)=\Sigma^{n-m}\Z[\sur(n, m)].
\]
Note that the natural action of $\Sigma_m$ on $S^m$ introduces on homology a twist by the sign representation of $\Sigma_m$ and the action of $\Sigma_n$ on $\Omega^n$ introduces a twist by the sign representation of $\Sigma_n$. This proves part~\eqref{item:tensor-tensor}.

\eqref{item:tensor-exterior} Appealing again to Proposition~\ref{prop:tensor-general}, and to Proposition~\ref{prop:exterior and divided powers}~\eqref{eq:propexterior}, we need to calculate the cross-effects of the functor $X\mapsto \Sigma^{-n} \widetilde\Z [X^{\wedge n}_{\Sigma_n}]$. Once again, note that if $U$ is a finite set then $\widetilde Z[(U_+^{\wedge n})_{\Sigma_n}]\cong Z[(U^n)_{\Sigma_n}]$. It follows that the $m$-th cross-effect of the functor $X\mapsto \widetilde \Z[(X^{\wedge n})_{\Sigma_n}]$ evaluated at $(S^0, \ldots, S^0)$ is equivalent to the total fiber of the cubical diagram $U\mapsto \Z[(U^n)_{\Sigma_n}]$, were $U$ ranges over subsets of $\{1, \ldots, m\}$. It is easy to see, once again, that taking the total fiber of this diagram amounts to ``killing'' the non-sujective functions from $\{1, \ldots, n\}$ to $\{1, \ldots, m\}$. Thus we obtain that
\[
\ce_m \Sigma^{-n}\widetilde \Z[(-^{\wedge n})_{\Sigma_n}](S^0, \ldots, S^0)\cong \Sigma^{-n}\Z[\Comp(n, m)].
\]
It follows that 
\[
\ext^*(\tensor^m \circ \abelianization, \Lambda^n \circ \abelianization)\cong \pi_{-*}S^m\otimes\Sigma^{-n}\Z[\Comp(n, m)]=\pi_{n-m-*}\Z[\Comp(n, m)],
\]
which means that $\ext^*(\tensor^m \circ \abelianization, \Lambda^n \circ \abelianization)\cong\Sigma^{n-m}\Z[\Comp(n,m)]$, as claimed. We proved part~\eqref{item:tensor-exterior}.

\eqref{item:tensor-dividedP} It follows from Proposition~\ref{prop:tensor-general} and Proposition~\ref{prop:exterior and divided powers}~\eqref{eq:dividedpower} that we have an isomorphism
\[
\ext^*(\tensor^m\circ\abelianization, \Gamma^n\circ\abelianization)\cong  \pi_{-*}\left(S^m\otimes \Sigma^{-2n}\ce_m\left(\widetilde\Z[(S^1\wedge -)^{\wedge n}_{\Sigma_n}]\right)(S^0, \ldots, S^0)\right).
\]
As usual, let us first (a) suppress the (de)suspension and (b) consider the value of this functor at a finite pointed set $U_+$. We have 
\[
\widetilde\Z[(S^1\wedge U_+)^{\wedge n}_{\Sigma_n}]\cong\widetilde\Z[(S^n \wedge U^n_+)_{\Sigma_n}].
\]
As before, it follows that the $m$-th cross-effect of the functor $X\mapsto \widetilde\Z[(S^1\wedge X)^{\wedge n}_{\Sigma_n}]$, evaluated at $S^0, \ldots, S^0$, is equivalent to $\widetilde \Z[S^n \wedge_{\Sigma_n}\sur(n, m)_+]$. 

It is elementary to show that $S^n \wedge_{\Sigma_n}\sur(n, m)_+$ is homeomorphic to a wedge sum of terms of the form $S^n_{H}$, where $H$ ranges over a representing set of stabilizers of the action of $\Sigma_n$ on $\sur(n,m)$.
When $n<m$ this is trivial, of course. On the other hand, when $n>m$, the stabilizer of every element of $\sur(n, m)$ is a non-trivial Young subgroup of $\Sigma_n$. We already mentioned that the quotient space $S^n_{\Sigma_n}$ is contractible whenever $n>1$. It follows that the quotient of $S^n$ by the action of a non-trivial Young subgroup of $\Sigma_n$ is contractible as well. We conclude that when $m\ne n$, $\widetilde \Z[S^n \wedge_{\Sigma_n}\sur(n, m)_+]$ is trivial, and therefore $\ext^*(\tensor^m \circ \ab, \Gamma^n \circ \abelianization)=0$. On the other hand, when $m=n$ we have the equivalence
\[
\ce_m\widetilde Z[(S^1\wedge -)^{\wedge m}_{\Sigma_m}](S^0, \ldots, S^0)\simeq \widetilde Z[S^m\wedge_{\Sigma_m}\sur(m,m)_+]\simeq \widetilde\Z[S^m]
\]
where $\Sigma_m$ is acting on the right hand side by permuting the coordinates of $S^m$. 

It follows that
\[
\ext^*(\tensor^m\circ \abelianization, \Gamma^m\circ \abelianization)\cong \pi_{-*} S^{2m}\otimes \Sigma^{-2m} \Z\cong \Z.
\]
Note that $\Sigma_m$ acts by permutation on $S^{2m}$ and acts trivially on $S^{-2m}$. The action of $\Sigma_m$ on $S^{2m}$ is trivial on homology, and therefore the $\Sigma_m$-action on $\Z$ on the right hand side is trivial.
\end{proof}
\subsection{$\ext^*(T^m\circ \abelianization, \symm^n\circ\abelianization)$}
Recall that $\symm^n$ is the $n$-th symmetric power functor $\symm^n(A)=A^{\otimes n}_{\Sigma_n}$. We will now  consider $\ext^*(\tensor^m \circ \abelianization, \symm^n\circ \abelianization)$. By Corollary~\ref{cor:symmetric-power}, 
\begin{equation} \label{eq:symmetric-power}
\widehat{\symm^n\circ \abelianization}(X)=\left(\Omega^n\widetilde\Z[X^{\wedge n}]\right)_{\Sigma_n}.
\end{equation}
This formula is less topological than the one we have for $\tensor^n\circ \abelianization$, $\Lambda^n\circ\abelianization$ and $\Gamma^n\circ \abelianization$, and for this reason calculating $\ext$ into $\symm^n\circ \abelianization$ is harder. We will do some calculations for small $m$ and $n$. One reason for doing it is to illustrate how one can work with the formula in~\eqref{eq:symmetric-power}. Another reason why this calculation is interesting is that it is the first example where we get torsion in the $\ext$ groups.

Let us begin by considering the case $m=1$. Combining the case $m=1$ of Proposition~\ref{prop:tensor-general} with Corollary~\ref{cor:symmetric-power}, we conclude that
\[
\ext^*\left(\abelianization, \symm^n\circ\abelianization\right)\cong \pi_{-*}\left(S^1 \otimes (\Omega^n  \Z)_{\Sigma_n}\right).
\]
It is well-known that there is a $\Sigma_n$-equivariant homeomorphism $S^n\cong S^1 \wedge \bar S^{n-1}$. Here $\Sigma_n$ acts on $S^n$ by permuting the coordinates, while it acts trivially on $S^1$, and acts by reduced standard representation on $\bar S^{n-1}$. It follows that there is an equivalence $S^1\otimes (\Omega^n\widetilde \Z)_{\Sigma_n}\simeq (\bar \Omega^{n-1}  \Z)_{\Sigma_n}$. We obtained the following lemma.
\begin{lem}\label{lem:ab-to-symm}
There is an isomorphism
\[
\ext^*\left(\abelianization, \symm^n\circ\abelianization\right)\cong \pi_{-*}\left((\bar\Omega^{n-1} \Z)_{\Sigma_n}\right).
\]
\end{lem}
This has the following easy consequences
\begin{lem}
\begin{enumerate}
    \item $\ext^i\left(\abelianization, \symm^n\circ\abelianization\right)\otimes \Q=0$ for $n>1$ and all $i$.
    \item $\ext^i\left(\abelianization, \symm^n\circ\abelianization\right)=0$ for $i\ge n$.
\end{enumerate}    
\end{lem}
\begin{proof}
    For the first part, we need to show that $\pi_{*}\left((\bar\Omega^{n-1}  \Z)_{\Sigma_n}\right)\otimes \Q=0$ whenever $n>1$. This is the same as the homology of the chain complex $(\bar \Omega^{n-1} \Q)_{\Sigma_n}$. Since we are in characteristic zero, there is an equivalence 
    \[
    (\bar \Omega^{n-1} \Q)_{\Sigma_n} \simeq (\bar \Omega^{n-1} \Q)_{h\Sigma_n} 
    \]
    where by the right hand side we mean $\bar \Omega^{n-1} \Q\otimes_{\Sigma_n}\Z[E\Sigma_n]$. There is a spectral sequence
 \[
 \HH_p(\Sigma_n; \pi_q(\bar \Omega^{n-1}\Q))\Rightarrow \pi_{p+q}\left((\bar \Omega^{n-1} \Q)_{h\Sigma_n}\right).
 \] 
 $\pi_*(\bar \Omega^{n-1} \Q)$ consists of the rationalized sign representation of $\Sigma_n$ in degree $1-n$, and nothing else. Assuming $n>1$, the rational homology of $\Sigma_n$ with coefficients in the sign representation is zero in all degrees, including degree zero. It follows that $ \HH_p(\Sigma_n; \pi_q(\bar \Omega^{n-1}\Q))=0$ for all $p$ and $q$, and therefore $\pi_{*}\left((\bar\Omega^{n-1}  \Z)_{\Sigma_n}\right)\otimes \Q=0$.

 For the second part, note the $\Sigma_n$-space $\bar S^{n-1}$ is the geometric realization of an $n-1$-dimensional
 simplicial set (see below). It follows that $(\bar\Omega^{n-1}  \Z)_{\Sigma_n}$ is equivalent to a chain complex concentrated in degrees $1-n, \ldots, 0$. It follows that $\pi_i(\bar\Omega^{n-1}  \Z)_{\Sigma_n}=0$ for $i\le -n$ and therefore $\ext^i\left(\abelianization, \symm^n\circ\abelianization\right)=0$ for $i\ge n$.
\end{proof}
The space $\bar S^{n-1}$ is $\Sigma_n$-equivariantly homeomorphic to the quotient space $\Delta^{n-1}/\partial\Delta^{n-1}$, where $\Delta^{n-1}$ is the $n-1$ dimensional simplex, on which $\Sigma_n$ acts by permuting the vertices. The space $\Delta^{n-1}$ is homeomorphic to the geometric realization of the poset of non-empty subsets of $\{1, \ldots, n\}$. The boundary $\partial\Delta^{n-1}$ is the realisation of the poset of non-empty \emph{proper} subsets of $\{1, \ldots, n\}$. It follows that $\bar S^{n-1}$ has a pointed simplicial model, where the set of $k$-dimensional simplices is the set of chains of subsets $U_0\subset U_1\subset \cdots \subset U_k=\{1, \ldots, n\}$, together with a basepoint which corresponds to all chains that don't end at the maximal element. Notice that the stabilizer group of a chain of this form is $\Sigma_{|U_0|}\times \Sigma_{|U_1|-|U_0|}\times \cdots \Sigma_{n-|U_{k-1}|}$. It follows that $\bar S^{n-1}$ has an equivariant pointed simplicial model of the following form
\[\begin{tikzcd}[column sep=small]
	{\frac{\Sigma_n}{\Sigma_n}_+} & {\displaystyle\bigvee_{1\le n_0\le n_1=n}\frac{\Sigma_n}{\Sigma_{n_0}\times\Sigma_{n-n_0}}_+} & \cdots  {\displaystyle\bigvee_{1\le n_0\le \cdots\le n_k=n}\frac{\Sigma_n}{\Sigma_{n_0}\times\cdots\times\Sigma_{n-n_{k-1}}}_+}  \cdots
	\arrow[shift right, from=1-1, to=1-2]
	\arrow[shift right, Rightarrow, from=1-2, to=1-1]
	\arrow[shift right=2, Rightarrow, from=1-2, to=1-3]
	\arrow[shift right, Rightarrow, scaling nfold=3, from=1-3, to=1-2]
\end{tikzcd}\]
Applying $\map(-, \Z)_{\Sigma_n}$ to this simplicial set we obtain a cosimplicial abelian group. Taking normalized chains, one obtains a chain complex of the following form
\begin{equation} \label{eq:coinvariat cochains}
\begin{tikzcd}
	\Z & {\displaystyle\bigoplus_{1\le n_0 < n_1=n} \Z} & \cdots & {\displaystyle\bigoplus_{1\le n_0<n_1<\cdots<n_k=n}\Z} & \cdots
	\arrow["{\delta^0}", from=1-1, to=1-2]
	\arrow["{\delta^1}", from=1-2, to=1-3]
	\arrow["{\delta^{k-1}}", from=1-3, to=1-4]
	\arrow[from=1-4, to=1-5]
\end{tikzcd}
\end{equation}
The number of tuples of integers that satisfy $1\le n_0<\cdots <n_k=n$ is $n-1 \choose k$. Thus the $k$-th coboundary homomorphism in this complex is a homomorphism of the form
\[
\delta^k \colon \Z^{n-1 \choose k} \to \Z^{n-1\choose k+1}.
\]
It follows that $\delta^k$ can be represented with an ${n-1 \choose k}\times {n-1 \choose k+1}$ matrix of integers. Let us denote this matrix by $[\delta^k]$. 
The rows of $[\delta^k]$ are indexed by tuples of the form $0=n_{-1}< n_0<\cdots <n_k=n$ and the columns by tuples of the form $0=m_{-1}< m_0<\cdots <m_{k+1}=n$.
We will now describe the entries of $[\delta^k]$, which one obtains by unraveling the definitions. Suppose that the tuple $(m_0, \ldots, m_{k+1})$ is a refinement of $(n_0, \ldots n_k)$. By this we mean that there exists a $0\le j\le k$ such that there is an identity of tuples
\[(m_0, \ldots,m_{j-1}, m_{j}, m_{j+1}, \ldots m_{k+1})=(n_0, \ldots,  n_{j-1},n_j', n_{j}, \ldots, n_k)\]
where $n_{j-1}<n_j'=m_{j}< n_{j}=m_{j+1}$. Then the entry of $[\delta^k]$ corresponding to row $(n_0, \ldots n_k)$ and column $(m_0, \ldots, m_{k+1})$ is $(-1)^j {n_{j}-n_{j-1} \choose n_j'-n_{j-1}}$. In all other cases, the entry is zero.
We have proved the following proposition
\begin{prop}
The groups $\ext^*(\abelianization,\symm^n\circ\abelianization)$ are isomorphic to the cohomology groups of the cochain complex~\eqref{eq:coinvariat cochains}, whose coboundary homomorphisms are described above.
\end{prop}
\begin{exmps}
Consider the case $n=2$. Then the cochain complex~\eqref{eq:coinvariat cochains} has the following form. Each copy of $\Z$ is labeled with the corresponding increasing sequence $(n_0, n_1\ldots, n_k)$, where $n_k=n=2$.
\[
\underset{(2)}{\Z}\xrightarrow{2}\underset{(1,2)}{\Z}.
\]
It follows that $\ext^1(\abelianization, \symm^2\circ \abelianization)=\Z/2$, and $\ext^i(\abelianization, \symm^2\circ \abelianization)=0$ for $i\ne 1$.

Next, consider the case $n=3$. In this case~\eqref{eq:coinvariat cochains} assumes the following form
\[
\underset{(3)}{\Z}\xrightarrow{[3, 3]}\underset{(1,3)}{\Z}\oplus \underset{(2,3)}{\Z} \xrightarrow{\begin{bmatrix} -2 \\2\end{bmatrix}} \underset{(1,2,3)}{\Z}.
\]
It follows that $\ext^1(\abelianization, \symm^3 \circ \abelianization)=\Z/3$, $\ext^2(\abelianization, \symm^3 \circ \abelianization)=\Z/2$, and $\ext^i(\abelianization, \symm^3 \circ \abelianization)=0$ for $i\ne 1, 2$.

Finally let us consider the case $n=4$. This is the last case we will do by hand. In this case~\eqref{eq:coinvariat cochains} takes the following form
\[
\underset{(4)}{\Z}\xrightarrow{\tiny[4,6,4]} \underset{(1, 4)}{\Z}\oplus \underset{(2, 4)}{\Z}\oplus \underset{(3,4)}{\Z}\xrightarrow{\tiny\begin{bmatrix}
    -3 & -3 & 0\\
    2 & 0 & -2\\
    0 & 3 & 3 
\end{bmatrix}} \underset{(1,2,4)}{\Z}\oplus \underset{(1,3,4)}{\Z}\oplus \underset{(2,3,4)}{\Z} \xrightarrow{\tiny\begin{bmatrix}
    2\\-2\\2
\end{bmatrix}} \underset{(1,2,3,4)}{\Z}
\]
It follows that $\ext^1(\abelianization, \symm^4\circ\abelianization)=\Z/2$, $\ext^2(\abelianization, \symm^4\circ\abelianization)=\Z/3$, $\ext^3(\abelianization, \symm^4\circ\abelianization)=\Z/2$, and $\ext^i(\abelianization, \symm^4\circ\abelianization)=0$ for all $i\ne 1,2,3$.

Table~\ref{tab:ext} describes $\ext^*(\abelianization, \symm^n\circ\abelianization)$ for $n\le 9$. It was obtained by calculating the cohomology of the complex~\eqref{eq:coinvariat cochains} with the help of a Python script. The reader is invited to guess and prove a general pattern.

\begin{table}
    \centering
    \begin{tabular}{c|c||c|c|c|c|c|c|c|c|c|c|}
       \multicolumn{11}{c}{\, \, \quad i} \\
       \cline{2-11}
      \multirow{10}{*}{n}& & 0 & 1 & 2 & 3 & 4 & 5 & 6 & 7 & 8  \\ \cline{2-11} \cline{2-11}
      & 1  & $\Z$ & 0 & 0 & 0 & 0 & 0 & 0 & 0 & 0 \\
      & 2  & 0 & $\Z/2$ & 0 & 0 & 0 & 0 & 0 & 0 & 0 \\
       & 3 & 0 & $\Z/3$ & $\Z/2$ & 0 & 0 & 0 & 0 & 0 & 0 \\
       & 4 & 0 & $\Z/2$ & $\Z/3$ & $\Z/2$ & 0 & 0 & 0 & 0 & 0  \\
       & 5 & 0 & $\Z/5$ & $\Z/2$ & 0 & $\Z/2$ & 0 & 0 & 0 &0 \\
       &6  & 0 & 0 & $\Z/10$ & $\Z/6$ & 0 & $\Z/2$ & 0 & 0 & 0 \\
       & 7 & 0 & $\Z/7$ & 0 & $\Z/2$ & $\Z/6$ & 0 & $\Z/2$ & 0 & 0 \\
       & 8 & 0 & $\Z/2$ & $\Z/7$ & $\Z/2$ & $\Z/2$ & $\Z/2$ & 0 & $\Z/2$ & 0\\
& 9 & 0 & $\Z/3$ & $\Z/2$ & 0 & $\Z/2$ & $\Z/6$ & $\Z/2$ & 0 & $\Z/2$\\ \cline{2-11}
    \end{tabular}
    \caption{$\ext^i(\abelianization, \symm^n\circ\abelianization)$ for $n\le 9$}
    \label{tab:ext}
\end{table}
\end{exmps}
Next let us discuss how one can calculate $\ext^*(\tensor^m\circ \abelianization, \symm^n\circ \abelianization)$. It turns out that it essentially can be reduced to the case $m=1$. We have the following generalization of Lemma~\ref{lem:ab-to-symm}
\begin{lem}\label{lem:tensor-to-symm}
There is an isomorphism
\[
\ext^*\left(\tensor^m\circ \abelianization, \symm^n\circ\abelianization\right)\cong \bigoplus_{\underset{n_1, \ldots, n_m>0}{n_1+\cdots +n_m=n}}\pi_{-*}\left((\bar\Omega^{n_1-1} \Z)_{\Sigma_{n_1}}\otimes\cdots \otimes (\bar\Omega^{n_m-1} \Z)_{\Sigma_{n_m}}\right).
\]
\end{lem}
Note that the right hand side can be calculated with the Kunneth formula. Thus the case of general $m$ can be calculated given knowledge of the case $m=1$.
\begin{exmp}
Let us calculate $\ext^*\left(\tensor^2\circ \abelianization, \symm^4\circ\abelianization\right)$. By Lemma~\ref{lem:tensor-to-symm}, this graded group is isomorphic to the following
\[
\pi_{-*}\left( \Z\otimes(\bar\Omega^{2} \Z)_{\Sigma_{3}}\right)\oplus \pi_{-*}\left( (\bar\Omega \Z)_{\Sigma_{2}}\otimes(\bar\Omega \Z)_{\Sigma_{2}}\right)\oplus \pi_{-*}\left( (\bar\Omega^{2} \Z)_{\Sigma_{3}}\otimes \Z\right).
\]
Using Table~\ref{tab:ext} and the Kunneth formula, we obtain the following result
\[
\ext^i(\tensor^2\circ \abelianization, \symm^4 \circ\abelianization)=\left\{\begin{array}{cc}
    \Z/2\oplus \Z/3\oplus\Z/3 & i=1 \\
    \Z/2\oplus \Z/2 \oplus \Z/2 & i=2 \\
    0 & \mbox{otherwise}
\end{array}\right.
\]
\end{exmp}
\subsection{Ext from abelianization to Passi functors}\label{section:ab to Passi}
In this section we calculate $\ext^*(\abelianization, \Passi_n)$ (see Corollary~\ref{cor:ab to passi} below).
We hope that someone will extend this calculation to $\ext^*(\tensor^m \circ \abelianization, \Passi_n)$.

By case $m=1$ of Proposition~\ref{prop:tensor-general}, there is an isomorphism
\[
\ext^*(\abelianization, \Passi_n) \cong \pi_{-*}S^1\otimes \widehat\Passi_n(S^0).
\]
By Proposition~\ref{prop:passi}
There is an equivalence, 
\[
\widehat \Passi_n(S^0) \simeq \underset{i\in \epi^{\le n}}{\nat}(S^i, \widetilde\Z[{S^0}^{\wedge i}]).
\]
We have a canonical isomorphism $\widetilde\Z[{S^0}^{\wedge i}]\cong \Z$. In particular, this is a constant functor in the variable $i$. Let us introduce the notation
\[
R_n=\colim_{i\in \epi^{\le n}} S^i.
\]
It follows that there is an equivalence
\[
\widehat \Passi_n(S^0) \simeq \pointedmaps(R_n, \widetilde\Z).
\]
We will see shortly that $R_n$ is a suspension of another space, which we denote $\bar R_n$. We have the following consequence
\begin{lem}\label{lem:ab to passi R_n}
There are isomorphisms
\[
\ext^*(\abelianization, \Passi_n) \cong \pi_{-*} \pointedmaps(\bar R_n, \Z)\cong \widetilde\HH^*(\bar R_n).
\]
\end{lem}
Let us analyze the (contravariant!) functor $\epi^{\le n}\to \spaces$ that sends $i$ to $S^i$, and its colimit $R_n$. Let us begin by showing that indeed this functor is a suspension of another functor. Following~\cite{Sphere-operad}, let us define the simplex $\widehat\Delta^{i-1}$ by the following formula
\[
\widehat\Delta^{i-1}=\{(x_1, \ldots, x_i)\in \R^i\mid 0\le x_1, \ldots, x_i,\ \underset{1\le j\le i}{\max}{x_j}=1\}.
\]
It is not hard to show that $\widehat\Delta^{i-1}$ is homeomorphic to the $i-1$-dimensional simplex, and furthermore its boundary consists of $i$-tuples $(x_1, \ldots, x_i)$ where at least one of the coordinates is zero. The reason for using this model of simplices is that it is functorial in $i$. More precisely, there is a contravariant functor
\[
\begin{array}{ccc}
    \epi & \to & \unpointed \\
     i & \mapsto & \widehat \Delta^{i-1}
\end{array}
\]
where functoriality on morphisms is given by repeating coordinates in the obvious way (see~\cite{Sphere-operad} for more details). Let $I=[0, 1]$. Consider the map
\begin{equation}\label{eq: sphere operad}
\begin{array}{ccc}
    \widehat\Delta^{i-1} \times I& \to & I^i \\
     ((x_1, \ldots, x_i), t) & \mapsto & (x_1t, \ldots, x_it).
\end{array}
\end{equation}
It is not hard to check that this map is a natural transformation of contravariant functors from $\epi$ to $\unpointed$, and that for each $i$ it gives a homeomorphism 
\[
(\widehat\Delta^{i-1} \times I, \partial (\widehat\Delta^{i-1} \times I))\xrightarrow{\cong} (I^i, \partial I^i).
\] 
Let us define $\widehat S^{i-1}=\widehat \Delta^{i-1}/\partial \widehat \Delta^{i-1}$. By quotiening the homeomorphism of~\eqref{eq: sphere operad} by the boundary, we obtain a homeomorphism of contravariant functors $\epi\to \spaces$
\[
\widehat S^{i-1} \wedge S^1\cong S^i.
\]
Thus, $S^i$ desuspends functorially in $i$. Let us define $\bar R_n=\underset{i\in \epi^{\le n}}{\colim}\widehat S^{i-1}$. We have a homeomorphism $R_n \cong S^1\wedge \bar R_n$, as promised.

Now we are ready to describe the homotopy type of $\bar R_n$
\begin{lem}\label{lem:R_n}
The space $\bar R_n$ has the structure of a pointed CW-complex, with two zero-dimensional cells (one of which is the basepoint), and a single cell of dimension $i$ for all $1\le i \le n-1$. When $n$ is odd, $\bar R_n\simeq S^{n-1}$. When $n$ is even, $\bar R_n \simeq *$.   
\end{lem}
\begin{proof}
Let us note that $\epi^{\op}$ is a generalized Reedy category in the sense of~\cite{Berger-Moerdijk}, with all maps raising degree, and only isomorphisms lowering degree. If $\catname{D}$ is a cofibrantly generated Quillen model category, then the 1-category of functors $[\epi^{\op}, \catname{D}]$ has a Reedy model structure, which, since the matching objects are trivial, agrees with the projective model structure. 

Consider the functor $\widehat S^{*-1}\colon\epi^{\op}\to \spaces$, $i\mapsto \widehat S^{i-1}$. For each $n$, the latching object of this functor at $n$ is
\[
\underset{n\twoheadrightarrow i\in n\downarrow \epi^{\le n-1}}\colim \widehat S^{i-1}.
\] 
It is easy to see that this colimit is homeomorphic to the fat diagonal in $\widehat S^{n-1}$. More precisely the colimit is homeomorphic to the quotient of the following subspace of $\widehat \Delta^{n-1}$ by its intersection with the boundary of $\widehat \Delta^{n-1}$
\[
\{(x_1, \ldots, x_n) \in \widehat \Delta^{n-1}\mid x_a=x_b \mbox{ for some } a<b\}.
\]
Let us denote the latching object at $n$ by $\Sigma^{-1}\Delta^n S^1$, since it is a desuspension of the fat diagonal in $S^n$. The space $\Sigma^{-1}\Delta^n S^1$ is a $\Sigma_n$-equivariant subspace of $\widehat S^{n-1}$. It consists exactly of the points where $\Sigma_n$ does not act freely. It follows that the inclusion $\Sigma^{-1}\Delta^n S^1\hookrightarrow \widehat S^{n-1}$ is a $\Sigma_n$-cofibration for all $n$. This means that the functor $i\mapsto \widehat S^{i-1}$ is Reedy cofibrant (and therefore also projective cofibrant).

Now we restrict the domain of the functor $\widehat S^{*-1}$ to $\epi^{\le n}$. Let $\lkan_{n-1}^n \widehat S^{*-1}$ denote the functor obtained by restricting $\widehat S^{*-1}$ to $\epi^{\le n-1}$ and then left Kan extending it back to $\epi^{\le n}$. There is a natural transformation $\lkan_{n-1}^n \widehat S^{*-1}\to \widehat S^{*-1}$. It is a homeomorphism for $*<n$. For $*=n$ it is the inclusion $\Sigma^{-1}\Delta^n S^1\hookrightarrow \widehat S^{n-1}$. 

The spaces $\Sigma^{-1}\Delta^n S^1$ and $\widehat S^{n-1}$ are spaces with an action of $\Sigma_n$. We may consider them as contravariant functors from $\epi^{\le n}$ to $\spaces$, by letting the functor to be trivial at all $i<n$. There is a standard square diagram of contravariant functors from $\epi^{\le n}$ to $\spaces$, which is both a pushout and a homotopy pushout
\[\begin{tikzcd}
	{\widehat S^{n-1}} & {\widehat S^{*-1}} \\
	{\Sigma^{-1}\Delta^nS^1} & {\lkan_{n-1}^n\widehat S^{*-1}}
	\arrow[from=1-1, to=1-2]
	\arrow[from=2-1, to=1-1]
	\arrow[from=2-1, to=2-2]
	\arrow[from=2-2, to=1-2]
\end{tikzcd}\]
Passing to colimits over $\epi^{\le n}$ we get a pushout diagram of the following form
\[\begin{tikzcd}
	{\widehat S^{n-1}/_{\Sigma_n}} & {\bar R_n} \\
	{\Sigma^{-1}\Delta^nS^1/_{\Sigma_n}} & {\bar R_{n-1}}
	\arrow[from=1-1, to=1-2]
	\arrow[from=2-1, to=1-1]
	\arrow[from=2-1, to=2-2]
	\arrow[from=2-2, to=1-2]
\end{tikzcd}\]
Let us analyze the spaces in the left half of this square. We mentioned already that $S^n/_{\Sigma_n}$ is a contractible space for $n>1$. The space $\widehat S^{n-1}/_{\Sigma_n}$ is a desuspension of this contractible space, and indeed it is contractible as well. In fact, we can describe it very explicitly. Recall that $\widehat S^{n-1}$ is the quotient of the space of $n$-tuples $(x_1, \ldots, x_n)$ of non-negative real numbers with maximum $1$ by the subspace consisting of $n$-tuples where at least one coordinates is zero. Every such $n$-tuple of real numbers is equivalent, up to permutation, to a unique $n$-tuple that satisfies $0\le x_1 \le \cdots \le x_{n-1}\le x_n=1$. It follows that $\widehat S^{n-1}/_{\Sigma_n}$ is homeomorphic to the quotient of the space of $n$-tuples of the form $0 \le x_1 \le \cdots \le x_{n-1}\le x_n=1$ by the subspace of such $n$-tuples where $x_1=0$. This is the quotient of an $n-1$-dimensional simplex by the face determined by the equation $x_1=0$. In particular $\widehat S^{n-1}/_{\Sigma_n}$ is homeomorphic to an $n-1$-dimensional simplex.

Next let us analyze the space ${\Sigma^{-1}\Delta^nS^1/_{\Sigma_n}}$. By a similar analysis to the above, it can be identified with a quotient of the space of $n$-tuples that satisfy $0\le x_1 \le \ldots \le x_{n-1}\le x_n = 1$, where at least two of the $x_i$s are the same. The space is quotiented by the subspace that satisfies $x_1=0$. This is exactly the boundary of the $n-1$-simplex, quotiened by the face $x_1=0$.

Let $\Delta^{n-1}$ be the simplex defined by the inequalities $0\le x_1\le \cdots \le x_{n-1}\le 1$. Let $\Delta^{n-1}_0$ be the face defined by $x_1=0$. We have shown that there is a pushout square
\[\begin{tikzcd}
	\Delta^{n-1}/\Delta^{n-1}_0 & {\bar R_n} \\
	{\partial\Delta^{n-1}/\Delta^{n-1}_0} & {\bar R_{n-1}}
	\arrow[from=1-1, to=1-2]
	\arrow[from=2-1, to=1-1]
	\arrow[from=2-1, to=2-2]
	\arrow[from=2-2, to=1-2]
\end{tikzcd}\]
It follows that $\bar R_n$ is obtained from $\bar R_{n-1}$ by attaching a single cell of dimension $n-1$. It is clear that $\bar R_0=S^0$. It follows by induction that $\bar R_n$ has a single cell in each dimension between $1$ and $n-1$. Furhtermore, if $i<n$ then $\bar R_i$ is the $i$-th skeleton of $\bar R_n$.

It remains to prove the statement about the homotopy type of $\bar R_n$. We already mentioned that $\bar R_1=S^0$. It is easy to check directly that $\bar R_2$ is homeomorphic to an interval, and therefore is contractible. Assume by induction that the statement we want to prove holds for $\bar R_i$ for $i<n$. We need to show that if $\bar R_{n-1}$ is contractible (and $n$ is odd) then $\bar R_n$ is a sphere and if $\bar R_{n-1}$ is a sphere (and $n$ is even) then $\bar R_n$ is contractible.

Since $\bar R_n$ is obtained from $\bar R_{n-1}$ by attaching a single cell of dimension $n-1$, it is clear that if $\bar R_{n-1}$ is contractible then $\bar R_n\simeq S^{n-1}$. Suppose that $n$ is even, and $\bar R_{n-1}\simeq S^{n-2}$. To understand the homotopy type of $\bar R_n$ we need to understand the degree of the attaching map $\partial\Delta^{n-1}/\Delta^{n-1}_0\to \bar R_{n-1}$. $\bar R_{n-1}$ has a single top-dimensional cell (of dimension $n-2$), that can be identified with the simplex defined by the inequalities $0\le x_1\le \cdots \le x_{n-2} \le x_{n-1}=1$. The simplex $\partial \Delta^{n-1}$ has $n$ facets. The attaching map collapses the facet $x_1=0$ to a point, and sends the interior of each of the remaining $n-1$ facets of $\partial \Delta^{n-1}$ homeomorphically onto the interior of the top dimensional cell of $\bar R_{n-1}$, with alternating signs. Since $n$ is even, $n-1$ is odd, so the total degree of the attaching map is $\pm 1$. It follows that  $\bar R_n\simeq *$.
\end{proof}
From Lemmas~\ref{lem:ab to passi R_n} and~\ref{lem:R_n} we have the following consequence
\begin{cor}\label{cor:ab to passi}
There is an isomorphism of graded abelian groups
\[
\ext^*(\abelianization, \Passi_n)\cong \left\{\begin{array}{cl}
  \Sigma^{n-1}\Z   & n \ \mathrm{ odd}  \\
   0  &  n \ \mathrm{ even}
\end{array}\right\}
\]
\end{cor}
\section{$\ext^*(\Lambda^m\circ\abelianization, -)$}\label{section:ext from exterior} In this section we will investigate $\ext^*(\Lambda^m \circ \abelianization, G)$, where $G$ is a polynomial functor. As usual, our method is to express the answer in terms of $\widehat G$. By Proposition~\ref{prop:exterior and divided powers}~\eqref{eq:propexterior}, the extension of $\Lambda^n \circ \abelianization$ is the functor
\[
\widehat{\Lambda^m \circ \abelianization}(X)=\Sigma^{-m}\widetilde\Z[X^{\wedge m}_{\Sigma_m}].
\]
Which implies by Theorem~\ref{thm:fr-to-top_intro} that 
\begin{equation}\label{eq:ext from lambda}
\ext^*(\Lambda^m\circ \abelianization, G)\simeq \pi_{-*-m}\spectralNat_X(\widetilde\Z[X^{\wedge m}_{\Sigma_m}],\widehat G(X)).
\end{equation}
\subsection{Rational calculations}\label{section:rational} The formula above is especially easy to use over $\Q$, because in this case strict orbits are equivalent to homotopy orbits. We mention this chiefly to illustrate how our method can be used to recover some previous results from, e.g.,~\cite{Vespa2018}. To see this, consider the functor $\Q\otimes ({\Lambda^m \circ \abelianization})$. It is clear that its extension is given by 
\begin{equation}\label{eq:rational}
\Q\otimes\widehat{(\Lambda^m \circ \abelianization)}(X)=\Sigma^{-m}\widetilde\Q[X^{\wedge m}_{\Sigma_m}].
\end{equation}
Suppose $X$ is an object in some $\infty$-category $\catname D$ with an action of $\Sigma_n$. We may identify $X$ with a functor $X\colon B\Sigma_n\to \catname{D}$. Recall that $X^{h\Sigma_n}$ and $X_{h\Sigma_n}$ denote, respectively, the $\infty$-categorical limit and colimit of this functor.
\begin{lem}\label{lem:rational}
Suppose $G\colon \fr\to \ab$ is a polynomial functor that takes values in rational vector spaces. Let $\widehat G$ be, as usual, the corresponding functor from $\spaces$ to $\ch$. Then
\[
\ext^*(\Lambda^m\circ\abelianization, G)\cong \ext^*\left(\Q\otimes(\Lambda^m\circ\abelianization\right), G)\cong \pi_{-*-m}\ce_m\widehat G(S^0, \ldots, S^0)^{h\Sigma_m}.
\]
\end{lem}
\begin{proof}
The first isomorphism is obvious, because $G$ takes values in rational vector spaces, and the category of rational chain complexes is a localization of the category of integral chain complexes.

By rationalized version of~\eqref{eq:ext from lambda} we have an isomorphism
\[
\ext^*(\Q\otimes(\Lambda^m\circ\abelianization), G)\cong \pi_{-*-m}\spectralNat_X(\widetilde \Q[X^{\wedge m}_{\Sigma_m}], \widehat G(X))
\]
The canonical map $X^{\wedge m}_{h\Sigma_m}\to X^{\wedge m}_{\Sigma_m}$ induces isomorphism in rational homology. Thus we have equivalences
\[
\widetilde Q[X^{\wedge m}]_{h\Sigma_m}\simeq \widetilde Q[X^{\wedge m}_{h\Sigma_m}]\xrightarrow{\simeq}\widetilde \Q[X^{\wedge m}_{\Sigma_m}].
\]
Therefore, we obtain an isomorphism
\[
\ext^*(\Q\otimes(\Lambda^m\circ\abelianization), G)\cong \pi_{-*-m}\spectralNat_X(\widetilde \Q[X^{\wedge m}]_{h\Sigma_m}, \widehat G(X)).
\]
The functor $\spectralNat(F, G)$ converts $\infty$-categorical colimits in $F$ to $\infty$-categorical limits. Therefore we have an equivalence
\[
\spectralNat_X(\widetilde \Q[X^{\wedge m}]_{h\Sigma_m}, \widehat G(X))\simeq \spectralNat_X(\widetilde \Q[X^{\wedge m}], \widehat G(X))^{h\Sigma_m}.
\]
By a calculation similar to the one we did in the proof of Proposition~\ref{prop:tensor-general}, the right hand side is equivalent to $\left(\ce_m\widehat G(S^0, \ldots, S^0)^{h\Sigma_m}\right)$.
It follows that there is an isomorphism
\[
\ext^*\left(\Q\otimes(\Lambda^m\circ\abelianization\right), G)\cong \pi_{-*-m}\ce_m\widehat G(S^0, \ldots, S^0)^{h\Sigma_m}.
\]
\end{proof}
We can now recover some rational calculations done in~\cite{Vespa2018}. We remind the reader that $\Part(n, m)={_{\Sigma_n}\backslash}\sur(n, m)/_{\Sigma_m}$ is the set of partitions of the number $n$ into $m$ non-zero summands. Similarly let $\Stirling(n, m)=\sur(n, m)/_{\Sigma_m}$ be the Stirling partition number, i.e., the number of ways to partition a set with $n$ elements into $m$ non-empty parts.
\begin{prop}[\cite{Vespa2018}, Theorem 3] \label{prop:Vespa theorem 3}
There are isomorphisms of graded $\Q$-vector spaces.
\begin{enumerate}
\item\label{eq:rational lambda to lambda}
\[
\ext^*(\Q\otimes(\Lambda^m\circ \abelianization), \Q\otimes(\Lambda^n \circ \abelianization))\cong \Sigma^{n-m}\Q[\Part(n, m)].
\]
\item \label{eq: rational lambda to gamma}
If $m=n\le 1$ then 
\[
\ext^*(\Q\otimes(\Lambda^m\circ \abelianization), \Q\otimes(\symm^n \circ \abelianization))\cong \ext^*(\Q\otimes(\Lambda^m\circ \abelianization), \Q\otimes(\Gamma^n \circ \abelianization))\cong\Q.
\]
If $(m,n)$ is not $(0,0)$ or $(1,1)$ then 
\[
\ext^*(\Q\otimes(\Lambda^m\circ \abelianization), \Q\otimes(\symm^n \circ \abelianization))\cong \ext^*(\Q\otimes(\Lambda^m\circ \abelianization), \Q\otimes(\Gamma^n \circ \abelianization))\cong 0.
\]
\item \label{eq: rational lambda to tensor}
\[
\ext^*(\Q\otimes(\Lambda^m\circ \abelianization), \Q\otimes(\tensor^n \circ \abelianization))\cong \Sigma^{n-m}Q[\Stirling(n, m)].
\]
\end{enumerate}
\end{prop}
\begin{proof}
\eqref{eq:rational lambda to lambda} By a calculation virtually identical to the proof of Corollary~\ref{cor: tensor-others}~\eqref{item:tensor-exterior}, the value of the $m$-th cross-effect of the functor $X\mapsto S^{-n}\widetilde\Q[X^{\wedge n}_{\Sigma_n}]$ at $S^0, \ldots, S^0$ is $\Sigma^{-n}\Q[{_{\Sigma_n}}\backslash\sur(n, m)]$.
By equation~\eqref{eq:rational} and Lemma~\ref{lem:rational} we have an isomorphism
\[
\ext^*(\Q\otimes(\Lambda^m\circ \abelianization), \Q\otimes(\Lambda^n \circ \abelianization))\cong \pi_{-*}\left(S^{m-n}\otimes (\Q[{_{\Sigma_n}}\!\backslash\sur(n, m)])^{h\Sigma_m}.
\right)
\]
There are natural equivalences
\[
\Q[{_{\Sigma_n}}\!\backslash\sur(n, m)/_{\Sigma_m}])\xleftarrow{\simeq}\Q[{_{\Sigma_n}}\!\backslash\sur(n, m)])_{h\Sigma_m}\xrightarrow{\simeq} \Q[{_{\Sigma_n}}\!\backslash\sur(n, m)])^{h\Sigma_m}.
\]
Here the left homomorphism is the canonical map from homotopy orbits to strict orbits, which is rationally an equivalence. The right map is the ``norm'' map from homotopy orbits to homotopy fixed points, which also is a rational equivalence. Reminding ourselves that $\Part(n,m)={_{\Sigma_n}}\!\backslash\sur(n, m)/_{\Sigma_m}$ We conclude that
\[
\ext^*(\Q\otimes(\Lambda^m\circ \abelianization), \Q\otimes(\Lambda^n \circ \abelianization))\cong \pi_{-*}\left(\Sigma^{m-n} \Q[\Part(n,m)]
\right)\cong\Sigma^{n-m}\Q[\Part(n, m)].
\]

\eqref{eq: rational lambda to gamma} First of all, note that there is a natural transformation from co-invariants to invariants
\[
\symm^n\circ\abelianization\to \Gamma^n\circ \abelianization.
\]
This transformation becomes an isomorphism after tensoring with $\Q$. It follows that there is an isomorphism
\[
\ext^*(\Q\otimes(\Lambda^m\circ \abelianization), \Q\otimes(\symm^n \circ \abelianization))\cong \ext^*(\Q\otimes(\Lambda^m\circ \abelianization), \Q\otimes(\Gamma^n \circ \abelianization)).
\]
So it is enough to calculate one of these groups. We will calculate Ext into $\Gamma^n\circ \abelianization$. Using Lemma~\ref{lem:rational} and Proposition~\ref{prop:exterior and divided powers}~\eqref{eq:dividedpower}, the problem is reduced to calculating the $m$-th cross-effect of the functor  $X\mapsto \Sigma^{-2n}\widetilde\Q[(SX)^{\wedge n}_{\Sigma_n}]$.
By the usual calculation, we find that
\[
\ext^*(\Q\otimes(\Lambda^m\circ \abelianization), \Q\otimes(\Gamma^n \circ \abelianization))\cong \pi_{-*+2n-m}\left(\widetilde\Q[S^n\wedge_{\Sigma_n}\sur(n, m)_+]^{h\Sigma_m}\right).
\]
If $m=n=0$ or $1$, then clearly the right hand side is isomorphic to $\Q$. If $(m, n)=(1,0)$ or $(0,1)$, then $\sur(n,m)=\emptyset$ and we get zero. In general, if $n<m$ the $\sur(n,m)=\emptyset$, and we get zero. So, suppose $n>m>0$. Then, as we saw already, $S^n\wedge_{\Sigma_n}\sur(n, m)_+$ is contractible, and we get zero. Finally, suppose $n=m>1$. In this case the expression $\widetilde\Q[S^n\wedge_{\Sigma_n}\sur(n, m)_+]^{h\Sigma_m}$ simplifies to 
\[
\widetilde Q[S^n]^{h\Sigma_n}\simeq \widetilde Q[S^n_{\Sigma_n}].
\]
Once again, $S^n_{\Sigma_n}$ is contractible, so the above chain complex is zero in $\pi_*$, so we get zero.

Part~\eqref{eq: rational lambda to tensor} is proved similarly, and is left to the reader.
\end{proof}
\subsubsection*{Integral calculations} Next we will do some integral calculations of $\ext^*(\Lambda^m \circ \abelianization, G)$. This is where our methods really lead to some new results. 

We will first give a presentation of $\ext^*(\Lambda^m \circ \abelianization, G)$ as the homotopy groups of an explicit homotopy limit. We will then use the general model to give a full calculation of $\ext^*(\Lambda^m \circ \abelianization, \Lambda^n\circ \abelianization)$ when $m=2,3$.  
\begin{rem}
One can adapt our methods to do similar calculations involving the functors $\Gamma^m\circ \abelianization$ and, possibly, $\symm^m\circ \abelianization$, though the latter is likely to be harder. We leave to the interested reader to work this out.
\end{rem}
The starting point for calculating $\ext^*(\Lambda^m \circ \abelianization, G)$ is equation~\eqref{eq:ext from lambda} on page~\pageref{eq:ext from lambda}. Let us remark that $\spectralNat(F, G)$ is an $\infty$-categorical mapping object. As such, it converts $\infty$-categorical (a.k.a homotopy) colimits in $F$ to $\infty$-categorical limits. But it does not, in general, behave well on $1$-categorical limits. So it is not true that $\spectralNat(\widetilde Z[X^{\wedge n}_{\Sigma_n}], \widehat G)$ is equivalent to $\spectralNat(\widetilde Z[X^{\wedge n}], \widehat G)^{\Sigma_n}$. To obtain a ``formula'' for $\spectralNat(\widetilde Z[X^{\wedge n}_{\Sigma_n}], \widehat G)$ we need to present $Z[X^{\wedge n}_{\Sigma_n}]$ as a homotopy colimit. 

There is in fact a well-known way to present the strict orbit space of an action of a finite group as a homotopy colimit indexed by an orbit category. Let us introduce a few definitions
\begin{defn}
Suppose $G$ is a finite group. Let $\orbit_G$ denote the orbit category of $G$. Its objects are transitive $G$-sets, and morphisms are $G$-equivariant maps.

More generally, suppose $\mathcal C$ is a collection of subgroups of $G$, i.e., a set of subgroups closed under conjugation. Then let $\orbit_{\mathcal C}$ denote the full subcategory of $\orbit_G$ consisting of transitive sets whose stabilizers belong to $\mathcal C$.
\end{defn}
We remind the reader that every transitive $G$-set $O$ is (noncanonically) isomorphic to a set of the form $G/H$, where $H$ is a subgroup of $G$. This is in fact a bijection between isomorphism classes of transitive $G$-sets and conjugacy classes of subgroups of $G$. For any $G$-space $X$ there is a canonical homemorphism $\map(G/H, X)^G\cong X^H$. 

Let $X$ be a pointed $G$-space. Then $X$ defines a contravariant functor from $\orbit_G$ to $\spaces$, $O\mapsto \map(O, X)^G$. Note that there is a canonical map $\pointedmaps(O, X)^G\to X/_G$, given by evaluation at any point of $O$. This map induces a canonical map
\[
\underset{O\in \orbit_G}{\hocolim} \map(O, X)^G \to X/_G.
\]
Proposition~\ref{prop:orbits homotopy colimit} below says that this map is an equivalence. This is a well-known consequence of Elmendorf's theorem. As far as we know, it was first observed by Dror Farjoun. The statement for unpointed spaces is proved as Proposition 0.0(ii) in~\cite{Slominska}. A similar argument works for pointed spaces.
\begin{prop}\label{prop:orbits homotopy colimit}
Let $X$ be a pointed $G$-CW complex. Let $\mathcal C$ be a collection of subgroups of $G$ that contains all the non-basepoint isotropy groups of $X$. Then the following canonical map is an equivalence, where the homotopy colimit is taken in the pointed category:
\begin{equation}\label{eq:orbits hocolim}
\underset{O\in \orbit_{\mathcal C}}{\hocolim} \map(O, X)^G \to X/_G.
\end{equation}
\end{prop}
Now let $X$ be a pointed CW-complex, and consider the space $X^{\wedge n}$, with the action of $\Sigma_n$. It is easy to see that the isotropy groups of $X^{\wedge n}$ are the Young subgroups of $\Sigma_n$. This means subgroups conjugate to $\Sigma_{n_1}\times \cdots \times \Sigma_{n_i}$, where $n_1, \ldots, n_i>0$, $n_1+\cdots +n_i=n$, and $\Sigma_{n_1}\times \cdots \times \Sigma_{n_i}$ consists of elements of $\Sigma_n$ that leave invariant the blocks of some partition of $\{1, \ldots, n\}$ with block sizes $n_1, \ldots, n_i$. Let ${\mathcal Y}_n$ be the collection of Young subgroups of $\Sigma_n$, and let $\orbit_{{\mathcal Y}_n}$ be the corresponding orbit category. Proposition~\ref{prop:orbits homotopy colimit} gives a presentation of $X^{\wedge n}_{\Sigma_n}$ as a homotopy colimit indexed on $\orbit_{{\mathcal Y}_n}$.

We find it convenient to use a different description of $\orbit_{{\mathcal Y}_n}$, as a category of multisets.
\begin{defn}
 For us, a multiset is a standard finite set together with a positive integer (multiplicity) associated with each element. A multiset has the form $(n_1, \ldots, n_i)$, where $\{1, \ldots, i\}$ is the underlying set, and $n_1, \ldots, n_i$ are the corresponding multiplicities. We refer to $i$ as the number of elements, and to the sum $n_1+\cdots+n_i$ as the total multiplicity of $(n_1, \ldots, n_i)$. A morphism between multisets $(m_1, \ldots, m_i)\to (n_1,\ldots, n_j)$ is a function $\alpha\colon\{1, \ldots, i\}\to \{1, \ldots, j\}$ such that for every $1\le s \le j$,
 \[
 \sum_{t\in \alpha^{-1}(s)}m_t=n_s.
 \]
Clearly morphisms are, in particular, surjective functions of underlying sets, and there can only exist morphisms between sets of same total multiplicity. Let $\mset_n$ denote the category of multisets of total multiplicity $n$.
\end{defn}
There is a functor $\mset_n\to \orbit_{{\mathcal Y}_n}$ defined as follows. Suppose $(n_1, \ldots, n_i)$ is an object of $\mset_n$. Define $\sur(n_1, \ldots, n_i; i)\subset \sur(n, i)$ to be the set of all surjective functions $\alpha\colon \{1, \ldots, n\}\twoheadrightarrow \{1, \ldots, i\}$ with the property that for each $1\le j\le i$, $\alpha^{-1}(j)$ has exactly $n_j$ elements.

The group $\Sigma_n$ acts naturally on $\sur(n, i)$ by pre-composition. It is clear that for any multi-set $(n_1, \ldots, n_i)$, $\sur(n_1, \ldots, n_i; i)$ is a transitive invariant subset, with stabilizer $\Sigma_{n_1}\times \cdots \times \Sigma_{n_i}$. Thus $\sur(n_1, \ldots, n_i; i)$ is a model for the $\Sigma_n$-orbit $\Sigma_n/\Sigma_{n_1}\times \cdots\times \Sigma_{n_i}$. 

Suppose that we have a morphism in $\mset_n$ $\alpha\colon (m_1, \ldots, m_i)\to (n_1, \ldots, n_j)$. It is clear that $\alpha$ induces a $\Sigma_n$-equivariant map $\sur(m_1, \ldots, m_i; n)\to \sur(n_1, \ldots, n_j; n)$. In this way we obtain a functor $\mset_n\to \orbit_{{\mathcal Y}_n}$. The following lemma is elementary
\begin{lem}\label{lem:multiset}
The assignment $(n_1, \ldots, n_i)\mapsto \sur(n_1, \ldots, n_i; n)$ defines an equivalence of categories
\[
\mset_n\xrightarrow{\simeq} \orbit_{{\mathcal Y}_n}.
\]
The functor sends $(n_1, \ldots, n_i)$ to a set isomorphic to $\Sigma_{n}/\Sigma_{n_1}\times\cdots\times\Sigma_{n_i}$.
\end{lem}
Let us go back to the space $X^{\wedge n}$ with the action of $\Sigma_n$. The contravariant functor $\orbit_{{\mathcal Y}_n}\to \spaces$ that sends $O \mapsto \map(O, X^{\wedge n})^{\Sigma_n}$ sends $\Sigma_n/\Sigma_{n_1}\times\cdots\times \Sigma_{n_i}$ to $(X^{\wedge n})^{\Sigma_{n_1}\times\cdots\times \Sigma_{n_i}}\cong X^{\wedge i}$. Composing this functor with the equivalence of Lemma~\ref{lem:multiset}, we obtain a contravariant functor $\mset_n\to \spaces$ that sends $(n_1, \ldots, n_i)$ to $X^{\wedge i}$. 
Combining Proposition~\ref{prop:orbits homotopy colimit} and Lemma~\ref{lem:multiset} we obtain the following lemma
\begin{lem}\label{lem:hocolim to orbit}
Let $X$ be a pointed CW complex. There is a natural equivalence
\[
\underset{(n_1, \ldots, n_i)\in \mset_n}{\hocolim}X^{\wedge i}\xrightarrow{\simeq} X^{\wedge n}/_{\Sigma_n}.
\]
\end{lem}
The category $\mset_n$ is an EI category, i.e., a category where every endomorphism is an isomorphism. Using S{\l}omi\'nska's model for homotopy colimits over an EI category~\cite{Slominska}, together with some elementary manipulations of colimits, it is not difficult to derive a presentation of the homotopy colimit that occurrs in Lemma~\ref{lem:hocolim to orbit} as a homotopy pushout of a cubical diagram as follows. 

Suppose $1\le i_0< \cdots < i_p \le n$. Let us denote by $\operatorname{Iso}(i_0, \ldots, i_p)$ the groupoid of chains of the form $\underline{i_0}\leftarrow \cdots \leftarrow \underline{i_p}$ in $\mset_n$, where each $\underline{i_j}$ is a multiset with $i_j$ elements and total multiplicity $n$. Then we have a contravariant cubical diagram, indexed by non-empty subsets of $\{1, \ldots, n\}$. Suppose $i_0, \ldots, i_p$ are as above, with $0\le p\le n-1$. The cubical diagram is defined by the formula
\begin{equation}\label{eq:cubical model}
\{i_0, \ldots, i_p\}\mapsto \underset{\underline{i_0}\leftarrow \cdots \leftarrow \underline{i_p}\in \operatorname{Iso}(i_0, \ldots, i_p)}{\hocolim} X^{\wedge i_0}.
\end{equation}
Note that since $\operatorname{Iso}(i_0, \ldots, i_p)$ is a groupoid, the homotopy colimit above is equivalent to a wedge sum indexed by the connected components of this groupoid. 
\begin{equation}\label{eq: crude cubical model}
\underset{\underline{i_0}\leftarrow \cdots \leftarrow \underline{i_p}\in\operatorname{Iso}(i_0, \ldots, i_p)}{\hocolim} X^{\wedge i_0}\simeq \bigvee_{[\underline{i_0}\leftarrow \cdots \leftarrow \underline{i_p}]} X^{\wedge i_0}_{h\aut(\underline{i_0}\leftarrow \cdots \leftarrow \underline{i_p})}.
\end{equation}
Here the wedge sum is indexed by a set of representatives of isomorphism classes of chains. 
To obtain a commuting cubical diagram one may want to use the left hand side of~\eqref{eq: crude cubical model}, but for understanding the homotopy types of the terms, the right hand side may be preferable.

The point  is that $X^{\wedge n}_{\Sigma_n}$ is equivalent to the (homotopy) pushout of this punctured cubical diagram. The following lemma is obtained by applying the machinery of~\cite{Slominska} to the homotopy colimit in Lemma~\ref{lem:hocolim to orbit}.
\begin{lem}\label{lem: cubical orbits}
Let $X$ be a pointed CW complex. The space $X^{\wedge n}_{\Sigma_n}$ is naturally equivalent to the homotopy colimit of the cubical diagram~\eqref{eq:cubical model} (which up to homotopy is the same as the right hand side of~\eqref{eq: crude cubical model}). \end{lem}
As a corolary, we now have the following general formula for $\ext^*(\Lambda^m \circ \abelianization, G)$.
\begin{thm}\label{thm:ext from exterior}
Let $G\colon\fr\to \ab$ be a polynomial functor. Then $\ext^*(\Lambda^m\circ \abelianization, G)$ is isomorphic to $\pi_{-*-m}$ of 
\[
\underset{(n_1, \ldots, n_i)\in \mset_n}{\operatorname{holim}}
 \ce_i\widehat G(S^0, \ldots, S^0).
\]
This is equivalent to the homotopy limit of a punctured cubical diagram of chain complexes, which on objects is defined by the formula
\[
(i_0, \ldots, i_p)\mapsto \prod_{[\underline{i_0}\leftarrow\cdots\leftarrow \underline{i_p}]}\ce_{i_0} \widehat G(S^0, \ldots, S^0)^{h\aut(\underline{i_0}\leftarrow\cdots\leftarrow \underline{i_p})}.
\]
Here the product is indexed by a set of representatives of isomorphism classes of chains $\underline{i_0}\leftarrow\cdots\leftarrow \underline{i_p}$ in $\mset_n$, where each $\underline{i_j}$ is a multiset with $i_j$ elements.
\end{thm}
\begin{proof}
By~\eqref{eq:ext from lambda}, 
$\ext^*(\Lambda^m\circ \abelianization, G)$ is isomorphic to $\pi_{-*-m} $ of $ \spectralNat_X(\widetilde\Z[X^{\wedge m}_{\Sigma_m}],\widehat G(X))$. By Lemma~\ref{lem:hocolim to orbit}
$X^{\wedge n}/_{\Sigma_n}$ is equivalent to $
\underset{(n_1, \ldots, n_i)\in \mset_n}{\hocolim}X^{\wedge i}$. The functor $\widetilde Z$ preserve homotopy colimit, so there is an equivalence
\[
\widetilde Z[X^{\wedge n}/_{\Sigma_n}]\simeq \underset{(n_1, \ldots, n_i)\in \mset_n}{\hocolim}\widetilde Z[X^{\wedge i}].
\]
Since $\spectralNat(-, \widehat G)$ takes homotopy colimits to homotopy limits, we obtain an equivalence
\[
\spectralNat_X(\widetilde\Z[X^{\wedge m}_{\Sigma_m}],\widehat G(X))\simeq \underset{(n_1, \ldots, n_i)\in \mset_n}{\operatorname{holim}}\spectralNat_X(X^{\wedge i}, \widehat G).
\]
By the equivalence in~\eqref{eq: nat from smash power is cross-effect}, which was established in the proof of Proposition~\ref{prop:tensor-general}, the right hand side is equivalent to
$\underset{(n_1, \ldots, n_i)\in \mset_n}{\operatorname{holim}}\ce_i\widehat G(S^0, \ldots, S^0)$. This proves the first assertion of the theorem. The second assertion follows from the first because of Lemma~\ref{lem: cubical orbits}.
\end{proof}
\begin{exmps}\label{exmp:lambda 2 and 3}
We will now use the general results above to calculate  $\ext^*(\Lambda^m\circ \abelianization, G)$ for $m=2, 3$. We start with $m=2$. By equation~\eqref{eq:ext from lambda}, 
\begin{equation}\label{eq: from Lambda square}
\ext^*(\Lambda^2\circ \abelianization, G)\cong \pi_{-*-2}\spectralNat_X(\widetilde \Z[(X\wedge X)_{\Sigma_2}], \widehat G(X)).
\end{equation}
When $n=2$, Lemma~\ref{lem: cubical orbits} amounts to saying that there is a homotopy pushout square, natural in $X$
\[\begin{tikzcd}
	{X\wedge \RP^\infty_+} & {(X\wedge X)_{h\Sigma_2}} \\
	X & {(X\wedge X)_{\Sigma_2}}
	\arrow[from=1-1, to=1-2]
	\arrow[from=1-1, to=2-1]
	\arrow[from=1-2, to=2-2]
	\arrow[from=2-1, to=2-2]
\end{tikzcd}\]
There the horizontal maps are induced by the diagonal, and the vertical maps are the natural maps from homotopy orbits to strict orbits. The functor $\widetilde \Z[-]$ from $\spaces$ to $\ch$ preserves homotopy colimits. Applying $\widetilde\Z[-]$ to the diagram above, and then applying $\spectralNat(-, \widehat G)$, we obtain the following homotopy pullback square
\[\begin{tikzcd}
	{\spectralNat_X\left(\widetilde\Z[(X\wedge X)_{\Sigma_2}], \widehat G(X)\right)} & \spectralNat_X(\widetilde\Z[X], \widehat G(X)) \\
	{\spectralNat_X\left(\widetilde\Z[(X\wedge X)_{h\Sigma_2}], \widehat G(X)\right)} & {\spectralNat_X(\widetilde\Z[X\wedge \RP^\infty_+], \widehat G(X))}
	\arrow[from=1-1, to=1-2]
	\arrow[from=1-1, to=2-1]
	\arrow[from=1-2, to=2-2]
	\arrow[from=2-1, to=2-2]
\end{tikzcd}\]
This is equivalent to the following pullback diagram, which is (the second statement of) Theorem~\ref{thm:ext from exterior} in the case $m=2$.
\begin{equation}\label{eq:Lambda2 pullback}
\begin{tikzcd}
	{\spectralNat_X\left(\widetilde\Z[(X\wedge X)_{\Sigma_2}], \widehat G(X)\right)} & \widehat G(S^0) \\
	{\ce_2\widehat G(S^0, S^0)^{h\Sigma_2}} & {\pointedmaps(\RP^\infty_+, \widehat G(S^0))}
	\arrow[from=1-1, to=1-2]
	\arrow[from=1-1, to=2-1]
	\arrow[from=1-2, to=2-2]
	\arrow[from=2-1, to=2-2]
\end{tikzcd}
\end{equation}
So, $\ext^*(\Lambda^2\circ\abelianization, G)$ is isomorphic to $\pi_{-*-2}$ of the pullback above.

We will now use this to calculate $\ext^*(\Lambda^2\circ \abelianization, \Lambda^n\circ\abelianization)$ for all $n$.
\begin{lem}\label{lem:Lambda2 to Lambdan}
There is an isomorphism of graded groups:
\[
\ext^*(\Lambda^2\circ \abelianization, \Lambda^n\circ\abelianization)\cong \left\{ \begin{array}{cl}
\Sigma^{n-2}\, \Z^{\Part(n, 2)}    & n\ \mathrm{ even}  \\
\Sigma^{n-2}\, \Z^{\Part(n, 2)} \oplus \Sigma^{n-1}\, \widetilde \HH^*(\RP^\infty)   &  n\ \mathrm{ odd}
\end{array}\right.
\]   
\end{lem}
\begin{proof}
By combining equation~\eqref{eq: from Lambda square} with the pullback square~\eqref{eq:Lambda2 pullback}, together with the formula for the cross-effects of $\widehat{\Lambda^n\circ \abelianization}$ (Corollary~\ref{cor: tensor-others}~\eqref{item:tensor-exterior}) we conclude that $\ext^*(\Lambda^2\circ \abelianization, \Lambda^n\circ\abelianization)$ is isomorphic to $\pi_{-*-2+n}$ of the homotopy pullback of the following diagram
\begin{equation}\label{eq:pullback}
\left(\Z[\Comp(n, 2)]\right)^{h\Sigma_2}\to\pointedmaps(\RP^\infty_+, \Z) \leftarrow \Z.
\end{equation}
Recall that $\Comp(n, 2)$ is the set of ordered pairs of positive integers $(n_1, n_2)$ where $n_1+n_2=n$. The group $\Sigma_2$ acts on $\Comp(n, 2)$ by switching $n_1$ and $n_2$. The left side map in~\eqref{eq:pullback} is induced by taking $\Z[-]^{h\Sigma_2}$ of the map $\Comp(n,2)\to \Comp(n,1)=*$.

Suppose first that $n=2k$ is even. Then $\Comp(n, 2)$ has a fixed point $(k, k)$. The inclusion of this fixed point extends to a map from the following pullback square to~\eqref{eq:pullback}
\[
\pointedmaps(\RP^\infty_+, \Z) \to \pointedmaps(\RP^\infty_+, \Z) \leftarrow *.
\]
It is clear that the pullback of this diagram is $*$. It follows that the pullback of~{eq:pullback} is equivalent to the pullback of the quotient of~\eqref{eq:pullback} by the pullback above. The quotient is the following pullback
\[
\left(\Z[\Comp(n, 2)\setminus *]\right)^{h\Sigma_2}\to * \leftarrow \Z.
\]
Here $\Comp(n, 2)\setminus *$ denotes the set $\Comp(n, 2)$ with the fixed point removed. The group $\Sigma_2$ acts freely on $\Comp(n, 2)\setminus *$, and it follows that $$\left(\Z[\Comp(n, 2)\setminus *]\right)^{h\Sigma_2}\simeq \Z[(\Comp(n, 2)\setminus *)_{\Sigma_2}].$$
It follows that the last pullback is equivalent to 
\[
\Z[(\Comp(n, 2)\setminus *)_{\Sigma_2}]\times \Z\simeq \Z[(\Comp(n, 2))_{\Sigma_2}]=\Z[\Part(n, 2)].
\]
We conclude that when $n$ is even, $\ext^*(\Lambda^2\circ \abelianization, \Lambda^n \circ \abelianization)\cong \pi_{-*-2+n}\Z[\Part(n, 2)]\cong\Sigma^{n-2}\Z[\Part(n,2)].$

Now suppose $n$ is odd. Then $\Sigma_2$ acts freely on $\Comp(n, 2)$. In this case $\Z[\Comp(n, 2)]^{h\Sigma_2}\simeq \Z[\Comp(n, 2)_{\Sigma_2}]=\Z[\Part(n, 2)]$. Thus~\eqref{eq:pullback} takes the following form
\[
\Z[\Part(n, 2)]\to\pointedmaps(\RP^\infty_+, \Z) \leftarrow \Z.
\]
Note that the right hand map in this pullback is equivalent to the inclusion $\Z\to \Z\times \pointedmaps(\mathbb R P^\infty, \Z)$. It follows that the last pullback is equivalent to the pullback of the following diagram
\[
\Z[\Part(n, 2)]\to\pointedmaps(\RP^\infty, \Z) \leftarrow *.
\]
The left map in this diagram is zero for dimensional reasons. It follows that the above pullback is equivalent to $\Z[\Part(n, 2)]\times \Omega\pointedmaps(\RP^\infty, \Z)$. We conclude that when $n$ is odd
\begin{eqnarray*}
\ext^*(\Lambda^2\circ \abelianization, \Lambda^n \circ \abelianization) &\cong &\pi_{-*-2+n}\left(\Z[\Part(n, 2)]\times \Omega\pointedmaps(\RP^\infty, \Z)\right) \\ &\cong& \Sigma^{n-2}\Z[\Part(n, 2)]\times \Sigma^{n-1}\widetilde\HH^*(\RP^\infty).
\end{eqnarray*}
\end{proof}
Now let us consider the case $m=3$. Once again, equation~\eqref{eq:ext from lambda} together with Corollary~\ref{cor: tensor-others}~\eqref{item:tensor-exterior} tell us that there is an isomorphisms
\begin{equation}\label{eq: from Lambda cube}
\ext^*(\Lambda^3\circ \abelianization, \Lambda^n\circ\abelianization)\cong \pi_{-*-3+n}\spectralNat_X(\widetilde \Z[(X^{\wedge 3})_{\Sigma_3}], \widetilde \Z[(X^{\wedge n})_{\Sigma_n}]).
\end{equation}
Thus our next task is to understand $\spectralNat_X(\widetilde \Z[(X^{\wedge 3})_{\Sigma_3}], \widetilde \Z[(X^{\wedge n})_{\Sigma_n}])$.
\begin{prop}\label{prop:nat from sym3x}
$\spectralNat_X(\widetilde \Z[(X^{\wedge 3})_{\Sigma_3}], \widetilde \Z[(X^{\wedge n})_{\Sigma_n}])$ is equivalent to
\[ \begin{array}{lc}
   \Z[\Part(n, 3)]\times \Sigma^{-1}\pointedmaps(\RP^\infty, \Z)^{\lfloor\frac{n}{2}\rfloor} & \mbox{if }3\mid n \\[10pt]
   \Z[\Part(n, 3)]\times \Sigma^{-1}\pointedmaps(\RP^\infty, \Z)^{\lfloor\frac{n}{2}\rfloor} \times \Sigma^{-1}\pointedmaps(\classifying\Sigma_3/\classifying \Sigma_2 , \Z) & \mbox{if }3\nmid n 
   \end{array}
\]
\end{prop}
\begin{proof}
We want to use the case $n=3$ of Lemma~\ref{lem: cubical orbits} to resolve $(X^{\wedge 3})_{\Sigma_3}$ as a homotopy colimit of a cubical diagram. The cube is indexed by isomorphism classes of chains in $\mset_3$. Recall that $\mset_3$ is the category of multisets of total multiplicity $3$. The following diagram indicates the isomorphism classes of chains in $\mset_3$ together with the corresponding groups of automorphisms (when the group is non-trivial).

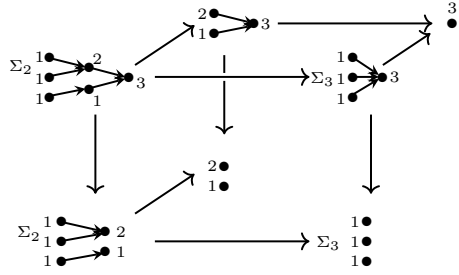
\tikzset{every picture/.style={line width=0.75pt}} 
\[
\begin{tikzpicture}[x=0.75pt,y=0.75pt,yscale=-1,xscale=1]

\draw [-stealth]   (45,91) -- (65,96) ;

\draw  [-stealth]  (45,101) -- (65,97) ;
 
\draw   [-stealth] (45,111) -- (65,107) ;

\draw  [-stealth]  (66,96) -- (85,101) ;

\draw  [-stealth]  (66,107) -- (85,102) ;

\draw (42,87) node [anchor=north west][inner sep=0.75pt]    {$\bullet $};
\draw (42,107) node [anchor=north west][inner sep=0.75pt]    {$\bullet $};
\draw (42,97) node [anchor=north west][inner sep=0.75pt]    {$\bullet $};
\draw (62,92) node [anchor=north west][inner sep=0.75pt]    {$\bullet $};
\draw (62,103) node [anchor=north west][inner sep=0.75pt]    {$\bullet $};
\draw (82,97) node [anchor=north west][inner sep=0.75pt]    {$\bullet $};

\draw (25,91) node [anchor=north west][inner sep=0.75pt]  [font=\tiny] [align=left] {$\Sigma_2$};

\draw (38,87) node [anchor=north west][inner sep=0.75pt]  [font=\tiny] [align=left] {$\displaystyle 1$};
\draw (38,108) node [anchor=north west][inner sep=0.75pt]  [font=\tiny] [align=left] {$\displaystyle 1$};
\draw (38,98) node [anchor=north west][inner sep=0.75pt]  [font=\tiny] [align=left] {$\displaystyle 1$};
\draw (67,109) node [anchor=north west][inner sep=0.75pt]  [font=\tiny] [align=left] {$\displaystyle 1$};
\draw (67,88) node [anchor=north west][inner sep=0.75pt]  [font=\tiny] [align=left] {$\displaystyle 2$};
\draw (89,100) node [anchor=north west][inner sep=0.75pt]  [font=\tiny] [align=left] {$\displaystyle 3$};

\draw [->] (90, 95)-- (120, 75);

\draw (125,65) node [anchor=north west][inner sep=0.75pt]    {$\bullet $};
\draw (125,75) node [anchor=north west][inner sep=0.75pt]    {$\bullet $};
\draw (145,70) node [anchor=north west][inner sep=0.75pt]    {$\bullet $};

\draw (120,65) node [anchor=north west][inner sep=0.75pt]  [font=\tiny] [align=left] {$\displaystyle 2$};
\draw (120,75) node [anchor=north west][inner sep=0.75pt]  [font=\tiny] [align=left] {$\displaystyle 1$};
\draw (152,70) node [anchor=north west][inner sep=0.75pt]  [font=\tiny] [align=left] {$\displaystyle 3$};
\draw[-stealth] (130,69)--(150,74);
\draw[-stealth] (130,79)--(150,75);

\draw [->] (100, 101)-- (178, 101);

\draw (195,87) node [anchor=north west][inner sep=0.75pt]    {$\bullet $};
\draw (195,107) node [anchor=north west][inner sep=0.75pt]    {$\bullet $};
\draw (195,97) node [anchor=north west][inner sep=0.75pt]    {$\bullet $};

\draw (190,87) node [anchor=north west][inner sep=0.75pt]  [font=\tiny] [align=left] {$\displaystyle 1$};
\draw (190,108) node [anchor=north west][inner sep=0.75pt]  [font=\tiny] [align=left] {$\displaystyle 1$};
\draw (190,98) node [anchor=north west][inner sep=0.75pt]  [font=\tiny] [align=left] {$\displaystyle 1$};

\draw (177,97) node [anchor=north west][inner sep=0.75pt]  [font=\tiny] [align=left] {$\Sigma_3$};

\draw (217,97) node [anchor=north west][inner sep=0.75pt]  [font=\tiny] [align=left] {$3$};

\draw (210,97) node [anchor=north west][inner sep=0.75pt]    {$\bullet $};

\draw[-stealth] (198,90)--(213,100);
\draw[-stealth] (198,101)--(213,101);
\draw[-stealth] (198,111)--(213,102);

\draw[->] (162,74)--(240,74);
\draw [->] (215, 95)-- (239, 79);

\draw (245,70) node [anchor=north west][inner sep=0.75pt]    {$\bullet $};

\draw (246,62) node [anchor=north west][inner sep=0.75pt]  [font=\tiny] [align=left] {$\displaystyle 3$};

\draw [->] (70, 120)-- (70, 160);
\draw [->] (209, 120)-- (209, 160);

\draw (48,170) node [anchor=north west][inner sep=0.75pt]    {$\bullet $};
\draw (48,180) node [anchor=north west][inner sep=0.75pt]    {$\bullet $};
\draw (48,190) node [anchor=north west][inner sep=0.75pt]    {$\bullet $};
\draw (70,175) node [anchor=north west][inner sep=0.75pt]    {$\bullet $};
\draw (70,185) node [anchor=north west][inner sep=0.75pt]    {$\bullet $};
\draw[-stealth] (52, 174)--(74,179);
\draw[-stealth] (52, 184)--(74,180);
\draw[-stealth] (52, 194)--(74,190);

\draw (29,175) node [anchor=north west][inner sep=0.75pt]  [font=\tiny] [align=left] {$\Sigma_2$};

\draw (42,170) node [anchor=north west][inner sep=0.75pt]  [font=\tiny] [align=left] {$\displaystyle 1$};
\draw (42,180) node [anchor=north west][inner sep=0.75pt]  [font=\tiny] [align=left] {$\displaystyle 1$};
\draw (42,190) node [anchor=north west][inner sep=0.75pt]  [font=\tiny] [align=left] {$\displaystyle 1$};
\draw (79,185) node [anchor=north west][inner sep=0.75pt]  [font=\tiny] [align=left] {$\displaystyle 1$};
\draw (79,175) node [anchor=north west][inner sep=0.75pt]  [font=\tiny] [align=left] {$\displaystyle 2$};

\draw[->] (100,184)--(178,184);

\draw (202,170) node [anchor=north west][inner sep=0.75pt]    {$\bullet $};
\draw (202,180) node [anchor=north west][inner sep=0.75pt]    {$\bullet $};
\draw (202,190) node [anchor=north west][inner sep=0.75pt]    {$\bullet $};

\draw (197,170) node [anchor=north west][inner sep=0.75pt]  [font=\tiny] [align=left] {$\displaystyle 1$};
\draw (197,180) node [anchor=north west][inner sep=0.75pt]  [font=\tiny] [align=left] {$\displaystyle 1$};
\draw (197,190) node [anchor=north west][inner sep=0.75pt]  [font=\tiny] [align=left] {$\displaystyle 1$};

\draw (180,180) node [anchor=north west][inner sep=0.75pt]  [font=\tiny] [align=left] {$\Sigma_3$};

\draw [->] (90, 170)-- (120, 150);

\draw (125,152) node [anchor=north west][inner sep=0.75pt]  [font=\tiny] [align=left] {$\displaystyle 1$};
\draw (125,142) node [anchor=north west][inner sep=0.75pt]  [font=\tiny] [align=left] {$\displaystyle 2$};
\draw (130,152) node [anchor=north west][inner sep=0.75pt]    {$\bullet $};
\draw (130,142) node [anchor=north west][inner sep=0.75pt]    {$\bullet $};

\draw (135, 90)--(135, 99);
\draw[->](135, 103)--(135, 130);

\end{tikzpicture}
\]
Recall that the cubical diagram in Lemma~\ref{lem: cubical orbits} that resolves $(X^{\wedge n})_{\Sigma_n}$ sends a chain $\underline {i_k}\to \cdots \to \underline{i_0}$ to $X^{\wedge i_0}_{h\aut(\underline {i_k}\to \cdots \to \underline{i_0})}$. It follows that there is a homotopy pushout cubes of functors
\[\begin{tikzcd}
	& X && X \\
	{X\wedge\RP^\infty_+} & {} & {X\wedge{\classifying \Sigma_3}_+} \\
	& {X\wedge X} & {} & {X^{\wedge 3}_{\Sigma_3}} \\
	{X\wedge X\wedge \RP^\infty_+} && {X^{\wedge 3}_{h\Sigma_3}}
	\arrow[from=1-2, to=1-4]
	\arrow[no head, from=1-2, to=2-2]
	\arrow[from=1-4, to=3-4]
	\arrow[from=2-1, to=1-2]
	\arrow[from=2-1, to=2-3]
	\arrow[from=2-1, to=4-1]
	\arrow[from=2-2, to=3-2]
	\arrow[from=2-3, to=1-4]
	\arrow[from=2-3, to=4-3]
	\arrow[no head, from=3-2, to=3-3]
	\arrow[from=3-3, to=3-4]
	\arrow[from=4-1, to=3-2]
	\arrow[from=4-1, to=4-3]
	\arrow[from=4-3, to=3-4]
\end{tikzcd}\]
Let $G\colon \fr\to \ch$ be a polynomial functor. Applying $\spectralNat_X(-, \widehat G)$ to the cubical diagram, we find that there a homotopy pullback diagram of the following form (for typographical reasons we write $\ce_i\widehat G$ to mean $\ce_i\widehat G(S^0, \ldots, S^0)$)
\[\begin{tikzcd}[sep = small]
	& {(\ce_3\widehat G)^{h\Sigma_3}} && {\pointedmaps(\RP^\infty_+, \ce_2\widehat G)} \\
	{\spectralNat_X(X^{\wedge 3}_{\Sigma_3}, \widehat G)} & {} & {\ce_2\widehat G} \\
	& {\pointedmaps({\classifying \Sigma_3}_+, \widehat G(S^0))} & {} & {\pointedmaps(\RP^\infty_+, \widehat G(S^0))} \\
	{\widehat G(S^0)} && {\widehat G(S^0)}
	\arrow[from=1-2, to=1-4]
	\arrow[no head, from=1-2, to=2-2]
	\arrow[from=1-4, to=3-4]
	\arrow[from=2-1, to=1-2]
	\arrow[from=2-1, to=2-3]
	\arrow[from=2-1, to=4-1]
	\arrow[from=2-2, to=3-2]
	\arrow[from=2-3, to=1-4]
	\arrow[""{name=0, anchor=center, inner sep=0}, from=2-3, to=4-3]
	\arrow[from=3-3, to=3-4]
	\arrow[from=4-1, to=3-2]
	\arrow[from=4-1, to=4-3]
	\arrow[from=4-3, to=3-4]
	\arrow[no head, from=3-2, to=3-3]
\end{tikzcd}\]
Now we substitute $\widehat G(X)=\widetilde\Z[X^{\wedge n}_{\Sigma_n}]$. In this case $\widehat G(S^0)=\Z$, $\ce_2\widehat G(S^0,S^0)=\Z[\Comp(n, 2)]$ and $\ce_3\widehat G(S^0, S^0, S^0)=\Z[\Comp(n, 3)]$. We obtain the following homotopy pullback cube.
\[\begin{tikzcd}[sep=small]
	& {(\Z[\Comp(n, 3)])^{h\Sigma_3}} && {\pointedmaps(\RP^\infty_+, \Z[\Comp(n, 2)])} \\
	{\spectralNat_X(X^{\wedge 3}_{\Sigma_3}, X^{\wedge n}_{\Sigma_n})} & {} & {\Z[\Comp(n, 2)]} \\
	& {\pointedmaps({\classifying \Sigma_3}_+, \Z)} & {} & {\pointedmaps(\RP^\infty_+, \Z)} \\
	{ Z} && {\Z}
	\arrow[from=1-2, to=1-4]
	\arrow[no head, from=1-2, to=2-2]
	\arrow[from=1-4, to=3-4]
	\arrow[from=2-1, to=1-2]
	\arrow[from=2-1, to=2-3]
	\arrow[from=2-1, to=4-1]
	\arrow[from=2-2, to=3-2]
	\arrow[swap, "\gamma", from=2-3, to=1-4]
	\arrow[""{name=0, anchor=center, inner sep=0}, from=2-3, to=4-3]
	\arrow[from=3-3, to=3-4]
	\arrow["\alpha", from=4-1, to=3-2]
	\arrow[from=4-1, to=4-3]
	\arrow["\beta", from=4-3, to=3-4]
	\arrow[no head, from=3-2, to=3-3]
\end{tikzcd}\]
Consider the arrows marked $\alpha$, $\beta$ and $\gamma$ in the diagram above. The cofibers of these maps are, respectively $\pointedmaps(\classifying\Sigma_3, \Z)$, $\pointedmaps(\RP^\infty, \Z)$, and $\pointedmaps(\RP^\infty, \Z[\Comp(n, 2)])$. It follows that $\spectralNat_X(X^{\wedge 3}_{\Sigma_3}, X^{\wedge n}_{\Sigma_n})$ is equivalent to the total fiber of the following square diagram
\begin{equation}\label{eq:flattened cube}
\begin{tikzcd}
	{\Z[\Comp(n,3)]^{h\Sigma_3}} & {\pointedmaps(\RP^\infty, \Z[\Comp(n,2)])} & {} \\
	{\pointedmaps({\classifying \Sigma_3}, \Z)} & {\pointedmaps(\RP^\infty, \Z)}
	\arrow[from=1-1, to=1-2]
	\arrow[from=1-1, to=2-1]
	\arrow[from=1-2, to=2-2]
	\arrow[from=2-1, to=2-2]
\end{tikzcd}
\end{equation}
Let us analyze $\Z[C(n, 3)]^{h\Sigma_3}$.
Recall that $\Comp(n, 3)$ is the set of ordered triples $(n_1, n_2, n_3)$ of positive integers whose sum is $n$. $\Sigma_3$ acts on this set by permuting the triples. The set of orbits of this action is $\Part(n, 3)$. For each point $x\in \Comp(n, 3)$ let $[x]\in \Part(n, 3)$ be the orbit of $x$, and let $G_x\subset \Sigma_3$ be the stabilizer of $x$. There is an equivalence
\[
\Z[C(n, 3)]^{h\Sigma_3} \simeq \prod_{[x]\in \Part(n, 3)}\pointedmaps (\classifying {G_x}_+,\Z).
\]
Furthermore, there is a natural equivalence $\pointedmaps(\classifying {G_x}_+, \Z)\simeq \Z\times \pointedmaps(\classifying G_x, \Z).$ Therefore we obtain an equivalence
\[
\Z[C(n, 3)]^{h\Sigma_3} \simeq \Z[\Part(n, 3)]\times \prod_{[x]\in \Part(n, 3)}\pointedmaps (\classifying {G_x},\Z).
\]
Suppose that $3\nmid n$. Then the action of $\Sigma_3$ does not have a fixed point. In this case the action has two types of orbits: the free orbits, consisting of triples where $n_1, n_2, n_3$ are distinct, and the orbits where two of the three are the same, but not all three. The stabilizer of such an orbit is conjugate to the group $\Sigma_2$ permuting the first two coordinates. Each orbit of this type has a unique representative of the form $(m,m,m')$ where $m\ne m'$ and $2m+m'=n$. 
It follows that when $3\nmid n$ there is an equivalence
\[
\Z[C(n, 3)]^{h\Sigma_3} \simeq \Z[\Part(n, 3)]\times \prod_{\begin{array}{c}{\{(m, m, m')\mid m, m'>0,} \\ {  2m+m'=n\}}\end{array}}\pointedmaps (\RP^\infty,\Z).
\]
Consider the top map in diagram~\eqref{eq:flattened cube}. With the splitting above, it takes the form
\[
\Z[\Part(n, 3)]\times \prod_{(m, m, m')}\pointedmaps (\RP^\infty,\Z)\to \prod_{(n_1, n_2)\in \Comp(n, 2)}\pointedmaps(\RP^\infty, \Z).
\]
The map is determined by what it does on each factor. The map $\Z[\Part(n, 3)]\to \pointedmaps(\RP^\infty, \Z)^{\Comp(n, 2)}$ can only be null-homotopic, for dimensional reasons. It is not difficult to check that the map 
\[\prod_{(m, m, m')}\pointedmaps (\RP^\infty,\Z) \to \prod_{(n_1, n_2)\in \Comp(n, 2)} \pointedmaps(\RP^\infty, \Z)
\] 
maps the factor indexed by $(m,m,m')$ isomorphically onto the factor indexed by the pair $(2m, m')\in \Comp(n, 2)$ (it also maps the same factor by multiplication by $2$ to the factor indexed by the pair $(m+m', m)$, but this does not affect the homotopy type of the homotopy fiber). It follows that the fiber of the last map is a product of copies of $\Omega  \pointedmaps(\RP^\infty, \Z)$ indexed by pairs $(n_1, n_2)\in \Comp(n, 2)$ where $n_1$ is odd. The number of such pairs is $\lfloor\frac{n}{2}\rfloor$. We conclude that when $3\nmid n$, the homotopy fiber of the top horizontal map in~\eqref{eq:flattened cube} is equivalent to 
\[
\Z[\Part(n, 3)]\times \Omega \pointedmaps(\RP^\infty, \Z)^{\lfloor\frac{n}{2}\rfloor}.
\]
Now consdier the bottom map in~\eqref{eq:flattened cube}. The fiber of this map is $\pointedmaps(\classifying\Sigma_3/\classifying\Sigma_2, \Z$. So the total fiber of~\eqref{eq:flattened cube} is the fiber of a map
\[
\Z[\Part(n, 3)]\times \Omega \pointedmaps(\RP^\infty, \Z)^{\lfloor\frac{n}{2}\rfloor}\to \pointedmaps(\classifying\Sigma_3/\classifying\Sigma_2, \Z).
\]
Once again, the restriction of the map to $\Z[\Part(n, 3)]$ can only be null, for dimensional reasons. Note that the targe of the map is $3$-local, while the factor $\Omega \pointedmaps(\RP^\infty, \Z)^{\lfloor\frac{n}{2}\rfloor}$ is acyclic at the prime $3$. It follows that the map can only be null, and therefore the total fiber of~\eqref{eq:flattened cube} is equivalent to
\[
\Z[\Part(n, 3)]\times \Omega \pointedmaps(\RP^\infty, \Z)^{\lfloor\frac{n}{2}\rfloor}\times  \Omega\pointedmaps(\classifying\Sigma_3/\classifying\Sigma_2, \Z).
\]
This proves the proposition when $3\nmid n$. 

Now suppose that $3\mid n$, so $n=3l$. Then the action of $\Sigma_3$ on $\Comp(n, 3)$ has a fixed point, namely $(l,l,l)$. In this case $\Z[C(n, 3)]^{h\Sigma_3}$ is equivalent to the following product
\[
 \Z[\Part(n, 3)]\times \prod_{\begin{array}{c}{\{(m, m, m')\mid m, m'>0,} \\ {  m\ne m', 2m+m'=n\}}\end{array}}\pointedmaps (\RP^\infty,\Z)\times \pointedmaps(\classifying\Sigma_3, \Z).
\]
And in this case the square diagram~\eqref{eq:flattened cube}
takes the following form
\[
\begin{tikzcd}
\Z[\Part(n, 3)]\times \underset{(m,m,m')}{\prod}\pointedmaps (\RP^\infty,\Z)\times \pointedmaps(\classifying\Sigma_3, \Z) & {\underset{(n_1, n_2)}{\prod}\pointedmaps(\RP^\infty, \Z)}  \\
	{\pointedmaps({\classifying \Sigma_3}, \Z)} & {\pointedmaps(\RP^\infty, \Z)}
	\arrow[from=1-1, to=1-2]
	\arrow[from=1-1, to=2-1]
	\arrow[from=1-2, to=2-2]
	\arrow[from=2-1, to=2-2]
\end{tikzcd}
\]
Here the triples $(m,m,m')$ satisfy, as before $m, m'>0$, $2m+m'=n$ and $m\ne m'$. The pairs $(n_1, n_2)$ range over all elements of $\Comp(n,2)$. Let us make the following observations about the maps in the last diagram: 
\begin{enumerate}
    \item The maps are zero on $\Z[\rho(n, 3)]$ for dimensional reasons. 
    \item The left vertical map restricts to an equivalence on $\pointedmaps(\classifying\Sigma_3, \Z)$, and it restricts to a transfer map (which is a split injection) on each copy of $\pointedmaps(\RP^\infty, \Z)$. 
    \item The right vertical map restricts to an equivalence on each copy of $\pointedmaps(\RP^\infty, \Z)$ 
    \item The top horizontal map is equivalent to a product of three maps:
    \begin{enumerate}
        \item The trivial map on $\Z[\Part(n,3)]$
        \item The map 
        \[
        \prod_{(m, m, m')}\pointedmaps(\RP^\infty, \Z) \to \prod_{(n_1, n_2)\ne(2l, l)} \pointedmaps(\RP^\infty, \Z)
        \]
        which maps the copy of $\pointedmaps(\RP^\infty, \Z)$ indexed by $(m,m,m')$ by an equivalence to the copy indexed by $(2m, m')$ and by multiplication by $2$ to the copy indexed by $(m+m', m)$.
        \item The restriction map $\pointedmaps(\classifying\Sigma_3, \Z)\to \pointedmaps(\RP^\infty, Z)$ where the latter copy of $\pointedmaps(\RP^\infty, Z)$ is indexed by $(2l, l)$.
    \end{enumerate}
    \item The bottom map is the restriction map.
\end{enumerate}  
It follows that the induced map of horizontal fibers of the latest diagram has the following form
\[
\Z[\Part(n, 3)]\times \Omega\pointedmaps (\RP^\infty,\Z)^{\lfloor\frac{n}{2}\rfloor}\times \pointedmaps(\classifying\Sigma_3/\classifying\Sigma_2, \Z)\to \pointedmaps(\classifying\Sigma_3/\classifying\Sigma_2, \Z).
\]
The map restricts to an equivalence of the factor $\pointedmaps(\classifying\Sigma_3/\classifying\Sigma_2, \Z)$. It follows that the fiber of the latest map, which is the total fiber of~\eqref{eq:flattened cube} is 
\[
\Z[\Part(n, 3)]\times \Omega\pointedmaps (\RP^\infty,\Z)^{\lfloor\frac{n}{2}\rfloor}.
\]
This proves the proposition in case $3\mid n$.
\end{proof}
\begin{cor}\label{cor:ext from lambda three to lambda n}
There is an isomorphism
\[
\ext^*(\Lambda^3\circ \abelianization, \Lambda^n\circ\abelianization)\cong \left\{\begin{array}{cc}
\Sigma^{n-3}\Z[\Part(n, 3)]\oplus\Sigma^{n-2}\widetilde\HH_*(\RP^\infty, \Z)^{\lfloor\frac{n}{2}\rfloor}     & \mbox{if } 3\mid n \\[10pt]
 \Sigma^{n-3}\Z[\Part(n, 3)]\oplus\Sigma^{n-2}\widetilde\HH_*(\RP^\infty, \Z)^{\lfloor\frac{n}{2}\rfloor}\oplus     & \mbox{if }3\nmid n\\ \oplus \Sigma^{n-2}\widetilde\HH^*(\classifying\Sigma_3/\classifying\Sigma_2) & 
\end{array}\right.
\]
\end{cor}
\begin{proof}
    This follows easily from Equation~\eqref{eq: from Lambda cube} and Proposition~\ref{prop:nat from sym3x}.
\end{proof}
\end{exmps}

\section{$\ext^*(\Passi_n, -)$}\label{section: ext from Passi}
Recall that $\Passi_n$ denotes the $n$-th Passi functor. In this section we study $\ext^*(\Passi_n, G)$. Recall from Proposition~\ref{prop:passi} that 
\[
\widehat{\Passi_n}=P_n \widetilde\Z[\Omega -].
\]
Recall from Remark~\ref{rem:excisive-sifted} that $P_n$ denotes Goodwillie's $n$-th Taylor approximation, and it is the left adjoint of the inclusion functor 
\[
i_n\colon \exc_n(\spaces, \ch)\to \fun_{\operatorname{ind}}(\spaces, \ch).
\]
Let us recall that the functor $i_n$ has a right adjoint as well. Furthermore, the right adjoint has a simple formula which we will now recall.

Recall that $\finset^{\le n}$ denotes the category of pointed sets with at most $n$ non-basepoint elements. Let $\catname{D}$ be a stable $\infty$-category. Let $\rho_n\colon\fun(\spaces, \catname{D})\to \fun(\finset^{\le n}, \catname{D})$ be the restriction functor. 
Recall from Lemma~\ref{lem: excisive finset} that
the restriction functor   
\[
\rho_n\colon\fun(\spaces, \catname{D})\to \fun(\finset^{\le n}, \catname{D})
\]
restricts to an equivalence, which we still denote by $\rho_n$
\[
\rho\colon\exc_n(\spaces, \catname{D})\xrightarrow{\simeq} \fun(\finset^{\le n}, \catname{D}),
\]
and the inverse equivalence is given by left Kan extension.


It follows that the left Kan extension functor
\[
\lkan_n\colon\fun(\finset^{\le n}, \catname{D}) \to \fun(\spaces, \catname{D})
\]
factors as a composition, where each functor is a fully faithful embedding
\[
\fun(\finset^{\le n}, \catname{D}) \xrightarrow[\simeq]{\lkan_n^0}\exc_n(\spaces, \catname{D}) \xrightarrow{i_n}\fun_{\operatorname{ind}}(\spaces, \catname{D})\xrightarrow{i} \fun(\spaces, \catname{D}).
\]
Since $\lkan_n=i\circ i_n \circ \lkan_n^0$ is left adjoint 
to the restriction functor $\rho_n\colon\fun(\spaces, \catname{D})\to \fun(\finset^{\le n}, \catname{D})$, and $i\colon \fun_{\operatorname{ind}}(\spaces, \catname{D})\to \fun(\spaces, \catname{D})$ is a fully faithful embedding, it follows that $i_n\circ \lkan_n^0$ is left adjoint to the composition $\rho_n\circ i$, which we will denote simply as $\rho_n\colon\fun_{\operatorname{ind}}(\spaces, \catname{D})\to \fun(\finset^{\le n}, \catname{D})$.

Since $\lkan_n^0$ is an equivalence, it follows that the composition of functors $\lkan_n^0\circ \rho_n$
\[
\fun_{\operatorname{ind}}(\spaces, \catname{D})\xrightarrow{\rho_n}\fun(\finset^{\le n}, \catname{D})\xrightarrow[\simeq]{\lkan_n^0}\exc_n(\spaces, \catname{D)}
\]
is a right adjoint to $i_n$. By slight abuse of notation, we denote the right adjoint of $i_n$ simply by $\lkan_n\rho_n.$ Given a functor $F\colon \spaces\to \catname{D}$, $\lkan_n\rho_n F$ is obtained by restricting $F$ to $\finset^{\le n}$ and then left Kan extending it back to $\spaces$.

The utility of the right adjoint is that it allows us to calculate maps from the left adjoint, as in the following proposition. To set up the notation, let $\catname{D}$ be a stable $\infty$-category, let $d$ be an object of $\catname{D}$ and let $K$ be a finite pointed CW complex. Consider the representable functor $d\otimes\pointedmaps(K, -)\colon \spaces\to \catname{D}$ that sends $X$ to $d\otimes \pointedmaps(K, X)$
\begin{prop}\label{prop: maps from Pn is Ln}
With notation above, let $H\colon \spaces\to \catname{D}$ be a functor. There is an equivalence
\[
\spectralNat(P_n (d\otimes \pointedmaps(K, -)), H)\simeq \spectralmaps_{\catname{D}}(d, \lkan_n\rho_n H(K)).
\]
\end{prop}
\begin{proof}
Note that when we write $P_n$ it really is a shorthand for $i_nP_n$, where $i_n$ is the inclusion of  $\exc_n(\spaces, \catname{D})$ into $\fun_{\operatorname{ind}}(\spaces, \catname{D})$ and $P_n$ is the left adjoint of $i_n$. Since $i_n$ also has a right adjoint $\lkan_n\rho_n$, we have equivalences
\begin{multline*}
\spectralNat(i_nP_n (d\otimes \pointedmaps(K, -)), H)\simeq \map_{\exc_n(\spaces, \catname{D})}(P_n (d\otimes \pointedmaps(K, -)), \lkan_n\rho_n H) \simeq \\ \simeq
\spectralNat(d\otimes\pointedmaps(K, -), i_n\lkan_n\rho_n H)\simeq \spectralmaps(d, \lkan_n\rho_n H(K)).
\end{multline*}
\end{proof}
As a corollary we have the following description of $\ext^*(\Passi_n, G)$.
\begin{prop}\label{prop: ext from Passi}
Let $G\colon \fr\to \ch$ be a polynomial functor. There is an isomorphism
\[
\ext^*(\Passi_n, G)\cong \pi_{-*}\lkan_n\rho_n \widehat G(S^1).
\]
\end{prop}
\begin{proof}
We know that \[\widehat \Passi_n=P_n \widetilde \Z[\Omega-]\simeq\Z\otimes \pointedmaps(S^1, -).\] It follows that
\[
\ext^*(\Passi_n, G)\cong \pi_{-*}\spectralNat(P_n\left(\Z\otimes \pointedmaps(S^1, -)\right), \widehat G(-)).
\]
Apply Proposition~\ref{prop: maps from Pn is Ln} to the right hand side, with $\catname{D}=\ch$. Since $\Z$ is the  unit of $\ch$, it follows that the right hand side is equivalent to $\lkan_n\rho_n \widehat G(S^1).$
\end{proof}
The last proposition motivates us to find a formula for $\lkan_n\rho_n$. Let $H\colon \spaces\to \catname{D}$ be a functor. By definition $\lkan_n \rho_nH$ is obtained by restricting $H$ to $\finset^{\le n}$ and then homotopy left Kan extending it back. There is a standard description of it as an $\infty$-categorical coend
\begin{equation}\label{eq:coend kan extension}
\lkan_n\rho_n H(X)\simeq \Sigma^\infty X^i \otimes_{i_+\in \finset^{\le n}} H(i_+).
\end{equation}
There is also a (well-known) smaller model for $\lkan_n\rho_n$ that uses coend over the category $\epi^{\le n}$ rather than $\finset^{\le n}$. We continue letting $\catname{D}$ denote a stable (or more generally additive and idempotent complete) $\infty$-category. 
Recall from Proposition~\ref{prop:Pirashvili-Helmstutler-Walde} that there is a ``Morita'' equivalence between categories
\[
-\otimes_{i} X^{\wedge i} : \begin{tikzcd}
           \fun(\epi^{\le n}, \catname{D})
           \arrow[r, shift left=.75ex]
            \arrow[r, phantom, "\simeq"] 
           &  \fun(\finset^{\le n}, \catname{D})
           \arrow[l, shift left=.75ex]      
        \end{tikzcd}  : \nat_X(X^{\wedge i}, -)
\]
Notice that the hypothesis on $\catname{D}$ is self-dual. Thus the proposition applies both to covariant and contravariant functors. If $F$ is a contravariant functor from $\finset$ to $\catname{D}$ then the corresponding functor from $\epi$ to $\catname{D}$ is $i\mapsto X^{\wedge i}\otimes_X F(X)$.

The following lemma says that these equivalences of functor categories induce equivalence of coends.
\begin{lem}\label{lem:DK coend}
With same hypotheses on $\catname{D}$, suppose we have a pair of functors $F, G\colon \finset\to \catname{D}$. There is an equivalence
\[
G\otimes_{\finset} F\simeq (G(X)\otimes_{X\in \finset}X^{\wedge i})\otimes_{i\in \epi} \nat_{Y}(Y^{\wedge i}, F(Y)).
\]
The same holds when $\finset$ and $\epi$ are replaced with $\finset^{\le n}$ and $\epi^{\le n}$.
\end{lem}
\begin{proof}
We give a sketch of proof, where we blithely manipulate $\infty$-categorical coends as if they were $1$-categorical coends without stopping to justify every step. 

By Proposition~\ref{prop:Pirashvili-Helmstutler-Walde} there is an equivalence, natural in $X$
\[
F(X)\simeq X^{\wedge i}\otimes_{i\in \epi}\nat_Y(Y^{\wedge i}, F(Y)).
\]
It follows that there is an equivalence
\[
G\otimes_{\finset} F\simeq G(X)\otimes_{X\in \finset}\left(X^{\wedge i}\otimes_{i\in \epi}\nat_Y(Y^{\wedge i}, F(Y))\right).
\]
This gives the desired result by associativity of coend (a.k.a the categorical Fubini theorem).
\end{proof}
\begin{rem}
Let $F\in \fun(\finset, \catname{D})$. We saw in Lemma~\ref{lem: smash power represents ce} that there is an equivalence
\[
\nat_X(X^{\wedge i}, F(X))\simeq \ce_iF(S^0, \ldots, S^0).
\]
In other words, the functor
\[
\nat_X(X^{\wedge -}, F(X))\colon \fun(\finset, \catname{D})\to \fun(\epi, \catname{D})
\]
is the functor that associates to $F$ the cross-effects of $F$ which themselves assemble into a functor $\epi\to\catname{D}$. The same statement in dualized form holds for contravariant functors. 
\end{rem}
We now can describe $\lkan_n\rho_n$ in terms of coend over $\epi$.
\begin{lem}\label{lem:epi model for left Kan}
Let $\catname{D}$ be a stable $\infty$-category, and let $H\colon \spaces\to \catname{D}$ be a functor. There is a natural equivalence
\[
\lkan_n\rho_n H(X) \simeq X^{\wedge j} \otimes_{j\in \epi^{\le n}} \ce_jH(S^0, \ldots, S^0).
\]
\end{lem}
\begin{proof}
    We saw in~\eqref{eq:coend kan extension} that 
 \[
\lkan_n\rho_n H(X) \simeq X^{i} \otimes_{i_+\in \finset^{\le n}} H(i_+).
\]   
By Lemma~\ref{lem:DK coend} it follows that
 \[
\lkan_n\rho_n H(X) \simeq (X^{i} \otimes_{i_+\in \finset^{\le n}} (i_+)^{\wedge j}) \otimes_{j\in \epi^{\le n}} \left(\nat_{Y\in \finset^{\le n}}(Y^{\wedge j}, H(Y))\right).
\]   
(Note that the role of $X$ in Lemma~\ref{lem:DK coend} is played by $i_+$ here, the role of $i$ there is played by $j$ here, and $X$ here does not have a counterpart in Lemma~\ref{lem:DK coend}. We apologize to the reader for this change of variables and hope it does not cause too much confusion).
We just saw that $\nat_{Y\in \finset^{\le n}}(Y^{\wedge j}, H(Y))\simeq \ce_jH(S^0, \ldots,S^0)$. We claim that for $j\le n$ there is an equivalence
\[
X^{i} \otimes_{i_+\in \finset^{\le n}} (i_+)^{\wedge j}\simeq X^{\wedge j}
\]
by an easy direct calculation. The lemma follows.
\end{proof}
\begin{rem}\label{rem:strict coend}
We have a model for $\lkan_n\rho_n H$ as an $\infty$-categorical coend. It is worth remarking that in many cases it can be calculated as a $1$-categorical coend. We claim that that for a fixed pointed simplicial set $X$, the functor $X^{\wedge -}\colon\epi^{\op}\to \sset_*$ is cofibrant in the projective model structure. Indeed, it is easy to check that it is cofibrant in the generalized Reedy model structure of~\cite{Berger-Moerdijk}, and that the Reedy model structure coincides with the projective model structure in this case. Suppose that $\catname{D}$ is an $\infty$-category arising from a combinatorial model category. Then $\fun(\epi, \catname{D})$ has an injective, as well as projective model structure. A functor is cofibrant in the injective model structure if it is objectwise cofibrant. Suppose $\ce_jH(S^0, \ldots, S^0)$ is cofibrant for every $j$. Then it follows from~\cite[Remark A.2.9.27]{LurieHTT} that the coend in Lemma~\ref{lem:epi model for left Kan} can be computed 1-categorically.
\end{rem}
This gives us a combinatorial model for $\ext^*(\Passi_n, G)$.
\begin{lem}\label{lem:ext passi strict coend}
Let $G\colon \fr\to \ab$ be a polynomial functor. Assume that we have a model for $\widehat G\colon \spaces\to \ch$ such that for every pointed set $n_+$, $\widehat G(n_+)$ is a cofibrant chain complex. Then there is an isomorphism
\[
\ext^*(\Passi_n, G)\cong \pi_{-*}\left(S^i \otimes_{i\in \epi^{\le n}} \ce_i\widehat G(S^0, \ldots, S^0)\right).
\]
Here the coend can be calculated either $1$ or $\infty$-categorically.
\end{lem}
\begin{proof}
The statement for $\infty$-categorical coend follows from Proposition~\ref{prop: ext from Passi} and Lemma~\ref{lem:epi model for left Kan}. Note that $\ce_j\widehat G(S^0, \ldots, S^0)$ is a direct summand of $\widehat G(i_+)$, and thus the cross-effects of $\widehat G$ take place in free abelian groups. The statement about $1$-categorical coend follows now from Remark~\ref{rem:strict coend}.
\end{proof}
Now we can calculate some examples. Let us consider $\ext^*(\Passi_m, \tensor^n\circ\abelianization)$. We describe these groups in terms of homology of a certain space that we will now define.
\begin{defn}
    Let $X$ be a pointed space. Suppose we are given positive integers $m, n$. Define
    \[
    \Delta^n_mX=\{x_1\wedge\ldots\wedge x_n\in X^{\wedge n}\mid \mbox{at most } m \mbox{ of the } x_i \mbox{ are distinct}\}
    \]
    We also define $\Delta^n_0 X=*$ when $0<n$ and $\Delta^0_0 X=S^0$.
\end{defn}
\begin{rem}
Note that $\Delta^n_{n-1}X=\Delta^n X$ is the fat diagonal in $X^{\wedge n}$. For $m<n$, $\Delta^n_mX$ is a subspace of $\Delta^n X$. For $m\ge n$, $\Delta^n_mX=X^{\wedge n}$.

Let $n/\epi^{\le m}$ be the slice category of arrows $n\twoheadrightarrow i$ where $n$ is fixed and $i\le m$. Then it is easy to see that there is a homeomorphism
\[
\Delta^n_mX\cong \underset{n\twoheadrightarrow i\in n/\epi^{\le m}}{\colim}X^{\wedge i}.
\]
Here colimit is taken $1$-categorically. The $1$-categorical colimit is homotopy equivalent to the $\infty$-categorical one because the functor $i\mapsto X^{\wedge i}$ is cofibrant.
\end{rem}
\begin{prop}\label{prop: ext from Passi to tensor}
  \[
\ext^*(\Passi_m, \tensor^n \circ\abelianization)\cong \widetilde \HH_{n-*}(\Delta^n_m S^1).
  \]  
\end{prop}
\begin{proof}
 By Lemma~\ref{lem:abelianization transform} $\widehat{\tensor^n\circ \abelianization}(X)=\Omega^n \widetilde \Z[X^{\wedge n}]$. By a straightforward calculation, $\ce_i\widehat{\tensor^n\circ \abelianization}(S^0, \ldots, S^0)\simeq \Sigma^{-n}\Z[\sur(n, i)]$. Note that this is a cofibrant chain complex. Indeed, it is a chain complex concentrated in a single degree, in which it is free abelian. 
 It follows from Lemma~\ref{lem:ext passi strict coend} that 
 \[
 \ext^*(\Passi_m, \tensor^n \circ\abelianization)\cong
 \pi_{-*}\left(S^i\otimes_{i\in\epi^{\le m}}\Sigma^{-n}\Z[\sur(n, i)]\right)
 \]
 and the right hand side is isomorphic to $\pi_{n-*}\left(S^i\otimes_{i\in\epi^{\le m}}\Z[\sur(n, i)]\right)$. The coend can be calculated $1$-categorically. It is easy to see that the coend is equivalent to 
 \[\widetilde Z[\underset{n\twoheadrightarrow i\in n/\epi^{\le m}}{\colim}S^i]\simeq \widetilde Z[\Delta^n_mS^n]
 \]
 which proves the proposition.
\end{proof}
In~\cite{Vespa2018}, Proposition 5 it is proved that if $0< m\le n$ then $\ext^{n-m}(\Passi_m, \tensor^n\circ\abelianization)$ is non-trivial. The following corollary refines this statement.
\begin{cor}\label{cor: ext from Passi to tensor}   
If $m< n$, $\ext^*(\Passi_m, \tensor^n\circ \abelianization)$ is concentrated in degree $*=n-m$. Furthermore, $\ext^{n-m}(\Passi_m, \tensor^n\circ \abelianization)$ is a free abelian group, which is non-trivial for $m>0$. For $m\ge n$ $\ext^0(\Passi_m, \tensor^n\circ \abelianization)\cong\Z$ and $\ext^*(\Passi_m, \tensor^n\circ \abelianization)\cong0$ for $*>0$.
\end{cor}
\begin{proof}
By Proposition~\ref{prop: ext from Passi to tensor}, the graded group $\ext^*(\Passi_m, \tensor^n\circ \abelianization)$ is isomorphic to $\widetilde{\HH}_{n-*}(\Delta^n_mS^1)$. For $m\ge n$ we have $\Delta^n_mS^1\cong S^n$, and we are done. For $0<m<n$,  $\Delta^n_mS^1$ is the one-point compactification of an arrangement of subspaces of $\R^n$. The maximal elements of this arrangement have dimension $m$, and for any subspace $H\subset \R^n$ that belongs to the arrangement, the rank of $H$ in the intersection poset is $m-\dim(H)$.

It follows readily from the Goresky-Macpherson formula, or better still from Ziegler-{\v Z}ivalievi{\'c}'s version of the formula~\cite[Theorem 2.2]{ZigZiv}, that $\Delta^n_m$ is equivalent to a wedge copies of $S^m$. Furthermore, the number of copies is non-zero whenever the arrangment is non-empty, i.e. whenever $0<m$. The case $m=0$ is trivial.
\end{proof}
\section{Stable cohomology of $\aut(\free{n})$}\label{section:stable cohomology}
One of the motivations for studying $\ext$ in the category of functors $\fr\to \ab$ comes from its application to the cohomology of groups of automorphisms of free groups with coefficients in a polynomial functor. The application is made possible thanks to remarkable work of Djament~\cite{Djament2019}, which itself builds on a famous theorem of Galatius~\cite{Galatius}. In this section we show how our methods can be combined with Djament's results. But first let us review the setup that we will work in.

Let $\free{n}$ be the free group on $n$ generators. Suppose $G$ is a functor $G\colon \fr\to \ab$. Then for each $n$, $G(\free{n})$ is a representation of $\aut(\free{n})$. Moreover, there are natural homomorphisms
\[
 \cdots \to \HH^*(\aut(\free{n}); G(\free{n}))\to \HH^*(\aut(\free{n-1}), G(\free{n-1}))\to \cdots
\] 
\begin{defn}
Define the stable cohomology of automorphisms of free groups with coefficients in $G$ to be the inverse limit of the sequence above
\[\HH^*_s(\aut; G):=\lim_{\infty \leftarrow n}\HH^*(\aut(\free{n}); G(\free{n})).
\] 
\end{defn}
The following stability result shows that the stable cohomology agrees with actual cohomology in a range.
\begin{thm}[Randal-Williams -- Wahl \cite{RW-Wahl}]
Suppose $G$ is a polynomial functor of degree $d$. The homomorphism $$\HH^*_s(\aut; G)\to \HH^*(\aut(\free{n}); G(\free{n}))$$ is $\frac{n-3}{2}-d$-connected.
\end{thm} 
The following theorem of Djament shows that one can calculate $\HH^*_s(\aut; G)$ in terms of mapping spaces in the $\infty$-category of polynomial functors $\fr\to \ab$.
\begin{thm}[\cite{Djament2019}, Theorem 3.3]\label{thm:Djament}
Let $G\colon \fr \to \ab$ be a polynomial functor. There is an isomorphism
\[
\HH^*_s(\aut, G)\cong \mathbb H^*\left(\Sigma_\infty; \bigoplus_{d\ge 0} \Sigma^d\mathbb H^*\left(\Sigma_d; \operatorname{R}\hom(\tensor^d \circ \abelianization, G)\otimes \mathbb Z[-1]\right)\right).
\]
Here \begin{enumerate}
\item $\mathbb H^*$ denotes hypercohomology.
\item $\operatorname{R}\hom(\tensor^d \circ \abelianization, G)$ is the derived $\hom$ in the category of functors, equipped with the action $\Sigma_d$ induced by the action on $\tensor^d$. To be really explicit, $\operatorname{R}\hom(\tensor^d \circ \abelianization, G)$ can be constructed as the cochain complex obtained by taking $\hom$ from a projective resolution of $\tensor^d \circ \abelianization$ to $G$.
\item $\Z[-1]$ is the sign representation of $\Sigma_d$.
\item $\Sigma_\infty$ acts trivially on the cochain complex $\bigoplus_{d\ge 0} \Sigma^d\mathbb H^* \cdots$. 
\end{enumerate}
\end{thm}
\begin{rem}
    The case when $G$ is the constant functor is Galatius's theorem~\cite{Galatius}, which says that
    $\HH^*_s(\aut, \Z)\cong \HH^*(\Sigma_\infty;\Z)$. The proof of Theorem~\ref{thm:Djament} in~\cite{Djament2019} uses Galatius's result.
\end{rem}
We now can reinterpret Djament's result in terms of our construction.
\begin{thm}\label{thm:Djament reinterpreted}
Let $G\colon \fr\to \ab$ be a polynomial functor of degree $n$. Let $\widehat G\colon \spaces\to \ch$ be, as usual, the extension of $G$ given by Theorem~\ref{thm:fr-to-top_intro}. There is an isomorphism
\[
\HH^*_s(\aut;G)\cong \pi_{-*}\pointedmaps\left({\classifying\Sigma_\infty}_+, \prod_{d=0}^n \ce_d\widehat G(S^0, \ldots, S^0)^{h\Sigma_d}\right).
\]
\end{thm}
\begin{proof}
Switching from cohomological to homological grading, Theorem~\ref{thm:Djament} says that $\HH^*_s(\aut;G)$ 
is isomorphic to 
\[
\pi_{-*}\pointedmaps\left({\classifying\Sigma_\infty}_+,\bigoplus_{d\ge 0} \Sigma^{-d}\left(\spectralNat_{\poly(\fr, \ch)}(\tensor^d\circ\abelianization, G)\otimes \Z[-1]\right)^{h\Sigma_d} \right)
\]
where $\spectralNat_{\poly(\fr, \ch)}(-,-)$ denotes the mapping object (that may be viewed as a chain complex, or a spectrum) in the $\infty$-catgory of polynomial functors from $\fr$ to $\ch$.

By the equivalence of categories $\poly(\fr, \ch)\simeq \exc(\spaces, \ch)$, we have an equivalence
\[
\spectralNat_{\poly(\fr, \ch)}(\tensor^d\circ\abelianization, G)\simeq \spectralNat_{\exc(\spaces, \ch)}(\widehat{\tensor^d\circ\abelianization}, \widehat G).
\]
By Lemma~\ref{lem:abelianization transform} $\widehat{\tensor^d\circ\abelianization}$ is the functor $X\mapsto \Omega^d\widetilde\Z[X^{\wedge d}]$, where $\Sigma_d$ acts by permuting both the $X$ coordinates and the loop coordinates (see Remark~\ref{rem:action}). This is equivalent to saying that the action of $\Sigma_d$ on $\widetilde\Z[X^{\wedge d}]$ is desuspended by $d$ and twisted by the sign. It follows that there is an equivalence
\[
\Sigma^{-d}\left(\spectralNat_{\exc(\spaces, \ch)}(\widehat{\tensor^d\circ\abelianization}, \widehat G)\otimes \Z[-1]\right)^{h\Sigma_d} \simeq \spectralNat_{X}(\widetilde\Z[X^{\wedge d}], \widehat G(X))
\]
where $\Sigma_d$ acts trivially on the suspension coordinate $\Sigma^d$ on the right hand side. It follows that $\HH^*_s(\aut;G)$ is isomorphic to 
\[
\pi_{-*}\pointedmaps\left({\classifying\Sigma_\infty}_+,\bigoplus_{d\ge 0} \left(\spectralNat_X(\widetilde\Z[X^{\wedge d]}], \widehat G)\right)^{h\Sigma_d} \right)
\]
By Lemma~\ref{lem: smash power represents ce} $\spectralNat_X(\widetilde\Z[X^{\wedge d]}], \widehat G)\simeq \ce_d\widehat G(S^0, \ldots, S^0)$. Since we assume that $G$ is polynomial of degree $n$, it follows that $\widehat G$ is $n$-excisive, and thus $\ce_d\widehat G(S^0, \ldots, S^0)\simeq *$ for $d> n$. We finally conclude that
\[
\HH^*_s(\aut;G)\cong \pi_{-*}\pointedmaps\left({\classifying\Sigma_\infty}_+, \prod_{d=0}^n \ce_d\widehat G(S^0, \ldots, S^0)^{h\Sigma_d}\right).
\]
\end{proof}

Let $R$ be a commutative ring with $1$. Define $\tensor^n_R\colon \fr\to R-\module$ to be the functor
\[
\tensor^n_R(G)=R\otimes \tensor^n(\abelianization(G)).
\]
Similarly define
\[
\Lambda^n_R(G)=R\otimes \Lambda^n (\abelianization(G))
\]
and also similarly define functors $\Gamma^n_R$, $\symm^n_R$, etc. The following proposition describes the stable cohomology of $\aut(\free{n})$ with coefficients in some of these functors in terms of cohomology of symmetric groups with coefficients in $R$. Let us note that while integral cohomology of symmetric groups may be known ``in principle'', cohomology with mod $p$ coefficients is known explicitly, and can be read for example from~\cite{Cohen-Lada-May}.
\begin{prop}\label{prop:stable cohomology calculations}
Let $n>0$. There are isomorphisms
\begin{equation}\label{item:stable tensor}
\HH^*_s(\aut; \tensor^n_R)\cong \HH^{*-n}\left({\classifying\Sigma_\infty} ; R\right)^{\Bell(n)}.
\end{equation}
Here $\Bell(n)$ is the total number of partitions of a set with $n$-elements.
\begin{eqnarray}\label{item:stable exterior}
 \HH^*_s(\aut; \Lambda^n_R) & \cong & \HH^{*-n}\left({\classifying\Sigma_\infty}\times \left(\coprod_{d=1}^n \Comp(n, d)\times_{\Sigma_d} E\Sigma_d \right); R\right) \\ \nonumber &\cong & \bigoplus_{d=1}^n\bigoplus_{\Lambda\in\Part(n, d)}\HH^{*-n}(\classifying\Sigma_\infty \times \classifying\aut(\Lambda);R).  
\end{eqnarray}
Here $\Lambda$ ranges over a set of representatives of orbits of the action of $\Sigma_d$ on $\Comp(n, d)$ and $\aut(\Lambda)$ denotes the stabilizer group of $\Lambda$.
\begin{equation} \label{item:stable Gamma}
  \HH^*_s(\aut; \Gamma^n_R) \cong  \HH^{*-n}\left({\classifying \Sigma_\infty}\times\classifying \Sigma_n;R[-1]\right) .
\end{equation}
Here $R[-1]$ denotes $R$ on which $\Sigma_\infty$ acts trivially, and $\Sigma_n$ acts by sign.
\end{prop}
\begin{proof}
\eqref{item:stable tensor}  By Lemma~\ref{lem:abelianization transform} we have $\widehat{\tensor^n\circ\abelianization}(X)=\Omega^n\widetilde\Z[X^{\wedge n}]$. It is clear that by the same argument one can prove that $\widehat{\tensor^n_R}(X)=\Omega^n\widetilde R[X^{\wedge n}]$. It follows that $\ce_d\widehat{\tensor^n_R}(S^0, \ldots, S^0)\cong \Omega^nR[\sur(n, d)]$. Note that the action of $\Sigma_d$ on $\sur(n, d)$ is free. It follows that $\Omega^nR[\sur(n, d)]^{h\Sigma_d}\simeq \Omega^n R[\sur(n, d)_{\Sigma_d}]$, and thus
\[
\prod_{d=0}^n \Omega^nR[\sur(n, d)]^{h\Sigma_d} \simeq \Omega^n R^{\Bell(n)}.
\]
By Theorem~\ref{thm:Djament reinterpreted} it follows that
\[
\HH^*_s(\aut, \tensor^n_R)\cong \pi_{-*}\pointedmaps\left({\classifying\Sigma_\infty}_+, \Omega^nR^{\Bell(n)}\right)\cong \HH^{*-n}(\classifying\Sigma_\infty; R)^{\Bell(n)}.
\]
\eqref{item:stable exterior} By tensoring with $R$ the result of Proposition~\ref{prop:exterior and divided powers}\eqref{eq:propexterior} we have $\widehat{\Lambda^n_R}(X)\simeq \Sigma^{-n}\widetilde R[X^{\wedge n}_{\Sigma_n}]$. It follows that $\ce_d\widehat{\Lambda^n_R}(S^0, \ldots, S^0)\simeq \Sigma^{-n}R^{\Comp(n, d)}$. Note that $\Comp(n, 0)=\emptyset$. And now by Theorem~\ref{thm:Djament reinterpreted} it follows that
\[
\HH^*_s(\aut, \Lambda^n_R)\cong \pi_{-*}\pointedmaps\left({\classifying\Sigma_\infty}_+, \Sigma^{-n}\prod_{d=1}^n (R^{C(n, d)})^{h\Sigma_d} \right)
\]
By some elementary manipulations, the right hand side is isomorphic to
\[\pi_{-*+n}\pointedmaps\left({\classifying\Sigma_\infty}_+\wedge \bigvee_{d=1}^n (\Comp(n, d)_+)_{h\Sigma_d}, R\right).
\]
This in turn is isomorphic to 
\[
\HH^{*-n}\left(\classifying\Sigma_\infty \times \left(\coprod_d \Comp(n, d)\times_{\Sigma_d} E\Sigma_d\right); R\right).
\]
This proves the first statement of~\eqref{item:stable exterior}. The second statement is just a reformulation of the first.

\eqref{item:stable Gamma} Once again, by tensoring with $R$ the result of Proposition~\ref{prop:exterior and divided powers}\eqref{eq:dividedpower} we find that $\widehat{\Gamma^n_R}(X)=\Sigma^{-2n}\widetilde R[(SX)^{\wedge n}_{\Sigma_n}]$. Therefore $\ce_d\widehat{\Gamma^n_R}(S^0, \ldots, S^0)=\Sigma^{-2n}R[S^n\wedge_{\Sigma_n} \sur(n, d)_+]$. We saw already in the proof of Corollary~\ref{cor: tensor-others}\eqref{item:tensor-dividedP} that $S^n\wedge_{\Sigma_n} \sur(n, d)_+\simeq *$ unless $n=d$. when $n=d$ we have $S^n\wedge_{\Sigma_n} \sur(n, d)_+\simeq S^d$, on which $\Sigma_d$ acts by permutation. It follows that 
\[
\prod_{d=0}^n \ce_d\widehat{\Gamma^n_R}(S^0, \ldots, S^0)^{h\Sigma_d} \simeq \Sigma^{-2n}\left(\widetilde R[S^n]\right)^{h\Sigma_n}\simeq \Sigma^{-n}R[-1]^{h\Sigma_n}.
\]
Applying Theorem~\ref{thm:Djament reinterpreted} one more time, we obtain the isomorphism
\[
\HH^*_s(\aut; \Gamma^n_R)\cong \pi_{-*}\pointedmaps({\classifying\Sigma_\infty}_+, \Sigma^{-n}R[-1]^{h\Sigma_n}).
\]
And it is easy to see that the right hand side is isomorphic to $\HH^{*-n}(\classifying \Sigma_\infty\times \classifying \Sigma_n; R[-1]).$
\end{proof}
When $R=\Q$ the formulas of Proposition~\ref{prop:stable cohomology calculations} simplify, unsurprisingly. The following corollary is originally due to Oscar Randal-Williams~\cite[Theorem A (ii) and Corollary D]{RW18}. See also~\cite[Theorem 4]{Vespa2018}, with the difference that we state the result in terms of stable cohomology rather than homology. 
\begin{cor}\label{cor:rational stable cohomology}
We have the following isomorphisms of graded groups
  \begin{enumerate}
      \item \[
      \HH^*_s(\aut; \tensor^n_{\Q})\cong \Sigma^n\Q^{\Bell(n)}.
      \] 
  \item\[
   \HH^*_s(\aut; \Lambda^n_{\Q}) \cong \Sigma^n \Q^{\Part(n)}  
  \]
  where $\Part(n)$ denotes the number of all partitions of the number $n$.
  \item
  \[
    \HH^*_s(\aut; \Gamma^n_{\Q}) \cong \HH^*_s(\aut; \symm^n_{\Q})\cong \left\{\begin{array}{cc}
        \Sigma\Q & n=1 \\
        0 & n>1
    \end{array}\right. .
  \]
   \end{enumerate} 
\end{cor}
\begin{proof}
  The first two statements of the corollary follow from the corresponding statements of Proposition~\ref{prop:stable cohomology calculations} because rational cohomology of $\Sigma_n$, and also of $\Sigma_\infty$ is the same as of a point.

For the third statement of the corollary, first of all recall that there is a rational isomorphism of functors $\symm^n \xrightarrow{\simeq_{\Q}} \Gamma^n$. So it is enough to prove the claim for the functor $\Gamma^n_Q$. The third statement of the corollary follows from the third statement of the proposition because for $n>1$ the rational cohomology of $\Sigma_n$ with coeffitients in the sign representation is trivial in all dimensions. 
\end{proof}
\begin{exmps}
Let us use Proposition~\ref{prop:stable cohomology calculations} to calculate the stable cohomology of $\aut(\free{n})$ with coefficients in $\tensor^3\circ \abelianization$, $\Lambda^3\circ \abelianization$, and $\Gamma^3\circ \abelianization$. It is easy to see that 
\begin{enumerate}
\item $\Bell(3)=5$, 
\item $\Comp(3, 1)$ consists of a single element $(3)$, whose group of automorphisms is trivial.
\item $\Comp(3, 2)$ consists of two elements $(1, 2)$ and $(2,1)$, which lie in the same orbit of the action of $\Sigma_2$ on $\Comp(3, 2)$ and have trivial stabilizers.
\item $\Comp(3, 3)$ consists of a single element $(1, 1, 1)$, whose stabilizer is $\Sigma_3$
\end{enumerate}
We conclude that there are isomorphisms
\[
\HH^*_s(\aut; \tensor^3\circ \abelianization)\cong \HH^{*-3}(B\Sigma_\infty)^5
\]
\[
\HH^*_s(\aut; \Lambda^3\circ \abelianization)\cong \HH^{*-3}(\classifying\Sigma_\infty\sqcup \classifying\Sigma_\infty \sqcup \classifying\Sigma_\infty\times \classifying\Sigma_3 ).
\]
\[
\HH^*_s(\aut; \Gamma^3\circ \abelianization)\cong \HH^{*-3}(\classifying \Sigma_\infty \times \classifying \Sigma_3; \Z[-1]).
\]
\end{exmps}
\bibliographystyle{alpha}
\bibliography{92Biblio}

\end{document}